\documentclass{amsart}

\usepackage{amssymb, comment, amsthm, todonotes}

\usepackage{color}
\usepackage[normalem]{ulem}

\newcommand{\mf}{\mathfrak}
\newcommand{\bb}{\mathbb}
\newcommand{\cal}{\mathcal}

\newcommand{\ch}{\mathsf{CH}}
\newcommand{\zfc}{\mathsf{ZFC}}

\newcommand{\Add}{\operatorname{Add}}

\newcommand{\lra}{\longrightarrow}
\newcommand{\lb}{\left\{}
\newcommand{\rb}{\right\}}
\newcommand{\la}{\langle}
\newcommand{\ra}{\rangle}
\newcommand{\res}{\restriction}
\newcommand{\bsl}{\backslash}
\newcommand{\seq}{\subseteq}
\newcommand{\es}{\emptyset}

\newcommand{\we}{\,\wedge\,}

\newcommand{\ps}{\mathbb{P}}
\newcommand{\R}{\mathbb{R}}
\newcommand{\Q}{\mathbb{Q}}

\newcommand{\ran}{\operatorname{ran}}
\newcommand{\dom}{\operatorname{dom}}
\newcommand{\base}{\operatorname{base}}
\newcommand{\htt}{\operatorname{ht}}
\newcommand{\ind}{\operatorname{index}}

\newcommand{\al}{\alpha}
\newcommand{\be}{\beta}
\newcommand{\ga}{\gamma}
\newcommand{\de}{\delta}
\newcommand{\De}{\Delta}
\newcommand{\ka}{\kappa}
\newcommand{\si}{\sigma}
\newcommand{\om}{\omega}
\newcommand{\vp}{\varphi}

\newenvironment{assumptiona}[1][]{\par\medskip
\noindent \textbf{Assumption~\thesection.1(a) #1} \rmfamily}{\medskip}

\newenvironment{assumptionb}[1][]{\refstepcounter{definition}\par\medskip
\noindent \textbf{Assumption~\thedefinition(b) #1} \rmfamily}{\medskip}

\theoremstyle{definition}

\newtheorem{definition}{Definition}[section]

\newtheorem{remark}[definition]{Remark}

\newtheorem{example}[definition]{Example}

\theoremstyle{plain}

\newtheorem{theorem}[definition]{Theorem}

\newtheorem{corollary}[definition]{Corollary}

\newtheorem{lemma}[definition]{Lemma}

\newtheorem{proposition}[definition]{Proposition}

\newtheorem{claim}[definition]{Claim}

\newtheorem{assumption}[definition]{Assumption}

\begin{document}

\title{Abraham-Rubin-Shelah Open Colorings and a Large Continuum}

\author{Thomas Gilton and Itay Neeman}
\thanks{This material is based upon work supported by the National Science Foundation under grant No. DMS-1764029.}

\date{March 2020}

\address{Department of Mathematics, University of California, Los Angeles, Box 951555, Los Angeles, CA 90095-1555}

\begin{abstract}
We show that the Abraham-Rubin-Shelah Open Coloring Axiom is consistent with a large continuum, in particular, consistent with $2^{\aleph_0}=\aleph_3$. This answers one of the main open questions from \cite{ARS}. As in \cite{ARS}, we need to construct names for so-called \emph{preassignments of colors} in order to add the necessary homogeneous sets. However, the known constructions of preassignments (ours in particular) only work assuming the $\ch$. In order to address this difficulty, we show how to construct such names with very strong symmetry conditions. This symmetry allows us to combine them in many different ways, using a new type of poset called a \emph{partition product}. Partition products may be thought of as a restricted memory iteration with stringent isomorphism and coherent-overlap conditions on the memories. We finally construct, in $L$, the partition product which gives us a model of $\mathsf{OCA}_{ARS}$ in which $2^{\aleph_0}=\aleph_3$.
\end{abstract}


%\ccode{Mathematics Subject Classification 2000: 14N35, 14J45}

\maketitle

%\tableofcontents

\section{Introduction}

Ramsey's Theorem, regarding colorings of tuples of $\om$, is a fundamental result in combinatorics. Naturally, set theorists have studied generalizations of this theorem which concern colorings of pairs of countable ordinals, that is to say, colorings on $\om_1$. The most straightforward generalization of this theorem is the assertion that any coloring of pairs of countable ordinals has an uncountable homogeneous set. However,  this naive generalization is provably false, at least in $\zfc$ (see \cite{SierpinskiNoRamsey}). One way to obtain consistent generalizations of Ramsey's Theorem to uncountable sets, $\om_1$ in particular, is to place various topological constraints on the colorings. This results in principles  known as Open Coloring Axioms, which we discuss presently. In what follows, we will use the notation $[A]^2$ to denote all two-element subsets of $A$.

\begin{definition} Let $A$ be a set and $\chi:[A]^2\lra\lb 0,1\rb$ a function. $H\seq A$ is said to be \emph{0-homogeneous (respectively 1-homogeneous) with respect to $\chi$} if $\chi$ takes the constant value 0 (respectively 1) on $[H]^2$. $H$ is said to be \emph{$\chi$-homogeneous} if it is either 0 or 1 homogeneous with respect to $\chi$.

 A function $\chi:[\om_1]^2\lra\lb 0,1\rb$ is said to be an \emph{open coloring} if $\chi^{-1}(\lb 0\rb)$ and $\chi^{-1}(\lb 1\rb)$ are both open in the product topology with respect to some second countable, Hausdorff topology on $\om_1$.

The \emph{Abraham-Rubin-Shelah Open Coloring Axiom}, abbreviated $\mathsf{OCA}_{ARS}$, states that for any open coloring $\chi$ on $\om_1$, there exists a partition $\om_1=\bigcup_{n<\om}A_n$ such that each $A_n$ is $\chi$-homogeneous.
\end{definition}

Abraham and Shelah (\cite{AS}) first studied the restriction of this axiom to colorings arising from injective functions $f:A\lra\R$, where $|A|=\aleph_1$ and used these ideas to show that Martin's Axiom does not imply Baumgartner's Axiom (\cite{Baumgartner}). The full version made its debut in \cite{ARS}, where the authors  studied it alongside a number of other axioms about $\aleph_1$-sized sets of reals. In particular, they showed that $\mathsf{OCA}_{ARS}$ is consistent with $\zfc$.

A little later, Todor{\v c}evi{\' c} isolated the following axiom (\cite{Todorcevic}):

\begin{definition} The \emph{Todor{\v c}evi{\' c} Open Coloring Axiom}, abbreviated $\mathsf{OCA}_T$, states the following: let $A$ be a set of reals, and suppose that $\chi:[A]^2\lra\lb 0,1\rb$ is a coloring so that $\chi^{-1}(\lb 0\rb)$ is open in $A\times A$. Then either there is an uncountable $A_0\seq A$ so that $A_0$ is 0-homogeneous with respect to $\chi$ or there exists a partition $A=\bigcup_{n<\om}B_n$ so that each $B_n$ is 1-homogeneous with respect to $\chi$.

If we restrict our attention to sets of reals $A$ with size $\aleph_1$, we denote this axiom by $\mathsf{OCA}_T(\aleph_1)$.
\end{definition}

Both of these versions of open coloring axioms imply that the $\ch$ is false. Indeed, $\mathsf{OCA}_{ARS}$ implies that if $A\seq\R$ has size $\aleph_1$ and $f:A\lra\R$ is injective, then $f$ is a union of countably-many monotonic subfunctions. However, under the $\ch$, there exists a function $f:\R\lra\R$ which is not continuous on any uncountable set, and hence not monotonic on any uncountable set (see $C_{62}$ of \cite{Sierpinski}). In fact, under the $\ch$, there is even an uncountable, injective, partial $f:\R\lra\R$ with no uncountable monotonic subfunction (see \cite{DM}), and thus even continuous colorings may fail to have large homogeneous sets under the $\ch$. With regards to $\mathsf{OCA}_T$, this axiom implies that the bounding number $\mf{b}$ is $\aleph_2$.

It is therefore of interest whether or not these axioms, individually or jointly, actually decide the value of the continuum. In the case of $\mathsf{OCA}_T$, Farah has shown in an unpublished note that $\mathsf{OCA}_T(\aleph_1)$ is consistent with an arbitrarily large value of the continuum, though it is not known whether the full $\mathsf{OCA}_T$ is consistent with larger values of the continuum than $\aleph_2$.  On the other hand, Moore has shown in \cite{Moore} that $\mathsf{OCA}_T + \mathsf{OCA}_{ARS}$ does decide that the continuum is exactly $\aleph_2$.

The question of whether $\mathsf{OCA}_{ARS}$ is powerful enough to decide the value of the continuum on its own, first asked in \cite{ARS}, has remained open. There are a number of difficulties in obtaining a model of $\mathsf{OCA}_{ARS}$ with a ``large continuum," i.e., with $2^{\aleph_0}>\aleph_2$. Chief among these difficulties is to construct so-called \emph{preassignments of colors}. The authors of \cite{AS} first discovered the technique of preassigning colors.  As used in \cite{ARS}, a preassignment of colors is a function which decides, in the ground model, whether the forcing will place a countable ordinal $\al$ inside \emph{some} 0-homogeneous or \emph{some} 1-homogeneous set, with respect to a fixed coloring. The key to the consistency of $\mathsf{OCA}_{ARS}$ is to construct preassignments in such a way that the posets which add the requisite homogeneous sets, as guided by the preassignments, are c.c.c.

However, the known constructions of ``good" preassignments (ours in particular) only work under the $\ch$. The construction of a preassignment involves diagonalizing out of closed sets in finite products of the space, ensuring that no such closed set can be the closure of an uncountable antichain of conditions in the poset (see the discussion preceding Remark \ref{remark:blerg}). Under the $\ch$, this task can be completed in $\om_1$-many steps, since the spaces we consider are second countable (and so there are at most $\aleph_1$-many such closed sets). Since forcing iterations whose strict initial segments satisfy the $\ch$ can only lead to a model where the continuum is at most $\aleph_2$, this creates considerable difficulties for obtaining models of $\mathsf{OCA}_{ARS}$ in which the continuum is, say,  $\aleph_3$.

In this paper, we prove the following theorem, thereby providing an answer to this question:

\begin{theorem}\label{theorem:main} If $\zfc$ is consistent, then so is $\zfc+\mathsf{OCA}_{ARS}+2^{\aleph_0}=\aleph_3$.
\end{theorem}

The general theme of the paper is the following: short iterations are necessary to preserve the $\ch$ and thereby construct effective preassignments; longer iterations, built out of these smaller ones in specific ways, can be used to obtain models with a large continuum. 

One can check using cardinality arguments that executing this theme requires the construction of names for preassignments with substantial symmetry. Let us briefly explain what we mean by this. For the purposes of the introduction, if $\chi$ is a coloring as in the definition of $\mathsf{OCA}_{ARS}$ and $f:\om_1\lra 2$ is an arbitrary function (any such is called a preassignment), then we use $\Q(\chi,f)$ to denote the poset to decompose $\om_1$ into countably-many $\chi$-homogeneous sets as guided by $f$ (see Section 4.2 for a precise description). In \cite{ARS}, the authors show that if $\ps$ is a sufficiently nice c.c.c.\ poset preserving the $\ch$ and $\dot{\chi}$ is a $\ps$-name for an open coloring, then there exists a $\ps$-name $\dot{f}$ for a preassignment so that $\ps$ outright forces that $\Q(\dot{\chi},\dot{f})$ is c.c.c. In our case, we are able to construct a \emph{single} $\ps$-name $\dot{f}$ so that the following (and much more besides) holds: if $G_L$ and $G_R$ are mutually generic filters for $\ps$, then the poset
$$
\Q(\dot{\chi}[G_L],\dot{f}[G_L])\times\Q(\dot{\chi}[G_R],\dot{f}[G_R])
$$
is c.c.c.\ in $V[G_L\times G_R]$. Constructing names for symmetric preassignments takes us beyond the techniques of \cite{ARS}. Since there are only $\aleph_2$ names for preassignments named in short iteration forcings, many of them must be reused more than once in a long iteration, and therefore this kind of symmetry proves to be necessary for executing our theme. Determining exactly how much ``symmetry" the names for preassignments need to satisfy leads us to the notion of a \emph{Partition Product}, which is a type of restricted memory iteration with various isomorphism and coherent overlap conditions on the memories.

Our method for proving Theorem \ref{theorem:main} is general enough that it can be adapted  to incorporate posets of size $\leq\aleph_1$ with the Knaster property, thereby adding the forcing axiom $\mathsf{FA}(\aleph_2,\text{Knaster}(\aleph_1))$ to the conclusion. This forcing axiom asserts that for any poset $\ps$ of size $\leq\aleph_1$ which has the Knaster property and any sequence $\la D_i:i<\om_2\ra$ of $\aleph_2$ dense subsets of $\ps$, there is a filter for $\ps$ which meets each of the $D_i$. Thus we also obtain the following theorem:

\begin{theorem}\label{theorem:FA} If $\zfc$ is consistent, then so is $\zfc+\mathsf{OCA}_{ARS}+2^{\aleph_0}=\aleph_3+\mathsf{FA}(\aleph_2,\emph{Knaster}(\aleph_1))$.
\end{theorem}

The remainder of the paper is structured as follows: in Section 2, we introduce the axiomatic definition of a \emph{partition product} and prove a number of general facts about this type of poset. In Section 3, we develop the machinery to combine partition products in a variety of ways. The main goal of Section 4 is to isolate exactly what we need our names for preassignments to satisfy. To do so, we introduce and develop the notion of a \emph{finitely generated} partition product. We also use the machinery developed so far to count these; the counting arguments will be crucial for the diagonalization arguments involved in constructing preassignments. Section 5 includes the actual construction of our highly symmetric names for preassignments. Finally, in section 6, we show how to construct partition products in $L$. It would be helpful, though not necessary, for the reader to be familiar with the first few sections of the paper \cite{ARS}, in particular, their construction of preassignments and the role which preassignments play in showing the consistency of $\mathsf{OCA}_{ARS}$.

We would like to thank the referee for investing a considerable amount of time in providing extensive comments which have helped, we believe, to improve  the exposition of this paper.

\section{Partition Products}

\subsection{Notation and Conventions}

We begin with some remarks about notation and conventions. First, if $f$ is a function and $A\seq\dom(f)$, then we will in general use $f[A]$ to denote $\lb f(x):x\in A\rb$, the pointwise image of $A$ under $f$.

Second, given two sets $X$ and $Y$, we will use $X \biguplus Y$ to denote their disjoint union, and if $X$ and $Y$ are also topological spaces with respective topologies $\tau_X$ and $\tau_Y$, we take the topology on $X\biguplus Y$ to be the disjoint union of the respective topologies on $X$ and $Y$, denoted $\tau_X\biguplus\tau_Y$. Similarly, if $f_X$ and $f_Y$ are functions with domains $X$ and $Y$, we take $f_X\biguplus f_Y$ to be the disjoint union of these functions defined in the natural way on $X\biguplus Y$. 

Additionally, we will often be working in the context of a poset $\R$ as well as various other posets related to it; these other posets will have notational decorations, for example, $\R^*$. If $\dot{G}$ is the canonical $\R$-name for a generic filter, we use the corresponding decorations, such as $\dot{G}^*$, to denote the related names.

Finally, regarding iterations, we will want to consider forcing iterations where the domain of the iteration is not necessarily an ordinal, but possibly a non-transitive set of ordinals. This will help smooth over various technicalities later. We view elements in an iteration as partial functions where each value of the function is forced by the restriction of the function to be a condition in the appropriate name for a poset.

\subsection{Definition and Basic Facts} 

Our first main goal in this section is to define the notion of a \emph{partition product} and prove a few basic lemmas. Before giving the definition, we have a few paragraphs of remarks which will help provide motivation. 

Roughly speaking, the class of partition products consists of finite support iterations with restricted memories which are built in very specific ways, but which is rich enough to be closed under products, closed under products of iterations taken over a common initial segment, and closed under more general ``partitioned products'' of segments of the iterations taken over common earlier segments. 

Iterations with restricted memory first appeared in Shelah's work (see \cite{ShelahNullCov}; see also \cite{ShelahNullSat} for further applications to the null ideal and \cite{MildenbergerShelah} for applications to cardinal characteristics). The following definition captures restricted memory in a convenient way for us. The memory limitation is in condition (\ref{restricted-memory-U}) of the definition. 

\begin{definition}
\label{restricted-memory}
$\R$ is a finite support \emph{restricted memory iteration} on $X$, of iterand names $\la \dot{\mathbb U}_\xi\mid \xi\in X\ra$, with \emph{memory function} $\xi\mapsto b_\R(\xi)$ ($\xi\in X$), if the following conditions hold:
\begin{enumerate}
\item 
\label{restricted-memory-fnct}
for each $\xi\in X$, $b_\R(\xi)\subseteq X\cap\xi$. If $\zeta\in b_\R(\xi)$ then $b_\R(\zeta)\subseteq b_\R(\xi)$;
\item
\label{restricted-memory-U}
$\dot{\mathbb U}_\xi$ is an $\R\res b_\R(\xi)$-name;
\item
\label{restricted-memory-R}
conditions in $\R$ are finite partial functions $p$ on $X$. For each $\xi\in \dom(p)$, $p(\xi)$ is a canonical $\R\res b_\R(\xi)$-name for an element of $\dot{\mathbb U}_\xi$;
\item 
\label{restricted-memory-order}
$q\leq p$ in $\R$ if $\dom(q)\supseteq \dom(p)$ and $q\res b_\R(\xi)\Vdash_{\R\res b_\R(\xi)} q(\xi)\leq p(\xi)$ for each $\xi\in \dom(p)$.
\end{enumerate}
$X$ is the \emph{domain} of the iteration, denoted $\dom(\R)$. A function satisfying condition (\ref{restricted-memory-fnct}) is a \emph{memory function} on $X$. The poset $\R\res b_\R(\xi)$ in conditions (\ref{restricted-memory-U})--(\ref{restricted-memory-order}) is the restriction of $\R$ to conditions $p$ with $\dom(p)\subseteq b_\R(\xi)$. Using condition (\ref{restricted-memory-fnct}) one can check that $\R\res b_\R(\mu)$ is a restricted memory iteration of $\la \dot{\mathbb U}_\xi \mid \xi\in b_\R(\mu)\ra$. More generally, $Y\subseteq X$ is \emph{memory closed}, also called \emph{base closed}, if for all $\xi\in Y$, $b_\R(\xi)\subseteq Y$. Then $\R\res Y$ is the restriction of $\R$ to conditions $p$ with $\dom(p)\subseteq Y$. One can check this is exactly the restricted memory iteration of $\la \dot{\mathbb U}_\xi\mid \xi\in Y\ra$ with memory function $b_\R\res Y$. 
\end{definition}

\begin{remark}
\label{remark:densesubset}
Strictly speaking $\R$ as in Definition \ref{restricted-memory} is not actually an iteration, even if $X$ is transitive. This is because we take $p(\xi)$ to be an $\R\res b_\R(\xi)$-name rather than an $\R\res \xi$-name. But, because $\dot{\mathbb U}_\xi$ is an $\R\res b_\R(\xi)$-name, every $\R\res\xi$-name for an element of $\dot{\mathbb U}_\xi$ can be forced in $\R\res\xi$ to be equal to an $\R\res b_\R(\xi)$-name. Using this and the finiteness of the supports one can check that $\R$ is isomorphic to a dense subset of an iteration. Working with $\R$, instead of the actual iteration, simplifies our definitions.  
\end{remark}

\begin{definition}
\label{acceptable-rearrangement}
Let $b$ be a memory function on $X$. A bijection $\sigma\colon X \lra X^*$ is an \emph{acceptable rearrangement} of $X,b$ if for all $\zeta,\xi\in X$, $\zeta\in b(\xi)$ implies $\sigma(\zeta)<\sigma(\xi)$. We then define $\sigma(b)$, the \emph{$\sigma$-rearrangement} of $b$, to be the function on $X^*$ given by $\sigma(b)(\sigma(\xi))=\sigma[b(\xi)]$, and we note that $\sigma(b)$ is a memory function on $X^*$. 
\end{definition}

\begin{definition}
\label{rearrangement-restricted-memory}
Suppose $\R$ is a restricted memory iteration on $X$, of $\la \dot{\mathbb U}_\xi\mid \xi\in X\ra$, with memory function $b_\R$. Let $\sigma\colon X\lra X^*$ be an acceptable rearrangement of $X,b_\R$. Then $\sigma$ induces an isomorphism, which we also denote $\sigma$, between $\R$ and a restricted memory iteration $\R^*$ on $X^*$. This isomorphism in turn extends to act on $\R$-names $\dot{u}$ and maps them to $\R^*$-names $\sigma(\dot{u})$ so that $\dot{u}[G]=\sigma(\dot{u})[\sigma[G]]$. These liftings are determined uniquely by the following conditions:
\begin{enumerate}
\item 
$\R^*$ is a restricted memory iteration of $\la \dot{\mathbb U}^*_\xi\mid \xi\in X^*\ra$ with memory function $\sigma(b_\R)$;
\item
for $p\in \R$, $\dom(\sigma(p))=\sigma[\dom(p)]$ and $\sigma(p)(\sigma(\xi))=\sigma(p(\xi))$; 
\item
for an $\R$-name $\dot{u}$, $\sigma(\dot{u})$ is the $\R^*$-name $\{\la \sigma(\dot{z}),\sigma(p)\ra\mid \la\dot{z},p\ra\in \dot{u}\}$;
\item
$\dot{\mathbb U}^*_{\si(\mu)}=\sigma(\dot{\mathbb U}_\mu)$.
\end{enumerate}
The last three conditions are recursive. Knowledge of $\dot{\mathbb U}^*_{\si(\xi)}$ for $\xi<\mu$ allows determining $\sigma(p)$ for $p\in \R\res \mu$, $\sigma(\dot{u})$ for $\R\res\mu$-names $\dot{u}$, and $\dot{\mathbb U}^*_{\si(\mu)}$. We refer to $\R^*$ as the \emph{$\sigma$-rearrangement} of $\R$, denoted $\sigma(\R)$. We also refer to $\sigma(\dot{u})$ as the \emph{$\sigma$-rearrangement} of $\dot{u}$, and we say that $\sigma$ is a rearrangement of $\R$.  
\end{definition}

A partition product is a restricted memory iteration with two additional structural requirements. The first requirement is that the iterand names (up to isomorphisms induced by rearrangements) are all taken from a restricted alphabet. This is phrased precisely in Definition \ref{simplified-partition-product}. The second requirement is that the memory sets $b_\R(\xi)$, when not disjoint, intersect in very specific ways. This is phrased precisely in Definitions \ref{coherent-collapse} and \ref{partition-product}. The first requirement captures the sense in which we construct long iterations from a small set of short building blocks. The addition of the second requirement helps us bound the class of partition products of size $\aleph_1$, up to isomorphism.

\begin{definition}
\label{restricted-memory-alphabet}
Let $C\subseteq\omega_2\setminus\omega$. A \emph{restricted memory alphabet} on $C$ is a pair of sequences $\underline{\ps}=\la\ps_\de:\de\in C\ra$ and  $\underline{\dot{\Q}}=\la\dot{\Q}_\de:\de\in C\ra$, where each $\ps_\delta$ is a restricted memory iteration, and each $\dot{\Q}_\de$ is a $\ps_\de$-name for a poset.
\end{definition}

\begin{definition}
\label{simplified-partition-product}
Let $\underline{\ps}$ and $\dot{\underline{\Q}}$ be a restricted memory alphabet on $C$. A \emph{simplified partition product} based upon $\underline{\ps}$ and $\underline{\dot{\Q}}$ is a restricted memory iteration $\R$ of iterands $\la \dot{\mathbb U}_\xi\mid \xi\in X\ra$, with memory function $b_\R$, and additional functions $\ind_\R$ on $X$ and $\pi^\R_\xi$ for each $\xi\in X$, so that: 
\begin{enumerate}
\item
\label{simplified-partition-product-ind} 
$\ind_\R(\xi)\in C$; 
\item
\label{simplified-partition-product-pi}
$\pi^\R_\xi\colon \dom(\ps_{\ind_\R(\xi)}) \lra b_\R(\xi)$ is an acceptable rearrangement of $\ps_{\ind_\R(\xi)}$;
\item
\label{simplified-partition-product-restriction}
$\R\res b_\R(\xi)$ is exactly equal to the $\pi^\R_\xi$-rearrangement of $\ps_{\ind_\R(\xi)}$;
\item
\label{simplified-partition-product-iterand}
$\dot{\mathbb U}_\xi=\pi^\R_\xi(\dot{\Q}_{\ind_\R(\xi)})$. 
\end{enumerate}
We set $\base_\R(\xi)=\la b_\R(\xi),\pi^\R_\xi\ra$, referring to $\base_\R$ as the \emph{base function} of $\R$. $\ind_\R$ is the \emph{index function} of $\R$, and the $\pi^\R_\xi$ are the \emph{rearrangement functions}.

For any base closed $Y\subseteq X$, the restriction of the above simplified partition product to $Y$ consists of the restriction of $\R$ to conditions with domain contained in $Y$, the restriction of $b_\R$ and $\ind_\R$ to $Y$, and the functions $\pi^\R_\xi$ for $\xi\in Y$. One can check that  this restriction is itself a simplified partition product. 
\end{definition}

\begin{definition}
\label{rearrangement-simplified-partition-product}
Let $\R$ be a simplified partition product. Let $\sigma$ be an acceptable rearrangement of $\R$ viewed as a restricted memory iteration. Let $\R^*$ be the $\sigma$-rearrangement of $\R$ viewed as a restricted memory iteration. Then $\R^*$ can be enriched to a simplified partition product, based upon the same alphabet as $\R$, by setting $\ind_{\R^*}(\sigma(\xi))=\ind_\R(\xi)$, and $\pi^{\R^*}_{\sigma(\xi)}=\sigma\circ \pi^\R_\xi$, so that $\base_{\R^*}(\sigma(\xi))=\la \sigma[b_\R(\xi)],\sigma\circ \pi^\R_\xi\ra$. We refer to this simplified partition product as the \emph{$\sigma$-rearrangement} of $\R$. We denote $\ind_{\R^*}$ and $\base_{\R^*}$ by $\sigma(\ind_\R)$ and $\sigma(\base_\R)$. 
\end{definition}

\begin{remark}
\label{rearrangement-pullback}
An analogue of Definition \ref{rearrangement-simplified-partition-product} works also for $\sigma^{-1}$: Suppose $\R$ is a restricted memory iteration, $\sigma$ an acceptable rearrangement, and $\R^*$ the $\sigma$-rearrangement of $\R$. If $\R^*$ enriches to a simplified partition product, with functions $\ind_{\R^*}$ and $\pi^{\R^*}_\xi$ say, then pulling these functions back by $\sigma^{-1}$ gives an enrichment of $\R$ to a simplified partition product. The $\sigma$-rearrangement of this enrichment of $\R$ is exactly the enrichment of $\R^*$ we started from.   
\end{remark}

\begin{example}
\label{restricted-memory-example}
Before proceeding to the full definition of partition products, we give a few examples of simplified partition products. 
\begin{enumerate}
\item 
\label{restricted-memory-example-1}
First, let $C\subseteq \omega_2\setminus \omega$ be non-empty, and let $\al_0<\om_2$ be the least element of $C$. Let $\ps_{\al_0}$ be the trivial forcing (viewed as a restricted memory iteration), and let $\dot{\Q}_{\al_0}$ be a $\ps_{\al_0}$-name for the poset $\Add(\om,1)$ to add a single Cohen real. Then any finite support power of $\ps_{\al_0}\ast \dot{\Q}_{\al_0}$ (which we may identify with Cohen forcing) is (enrichable to) a simplified partition product. Indeed, let $X$ be a set of ordinals, and for $\xi\in X$, set $b_\xi=\es$, $\pi_\xi=\es$, $\ind(\xi)=\al_0$, and $\dot{\mathbb U}_\xi=\dot{\Q}_{\al_0}$. These assignments generate a simplified partition product $\R$, and it is clear that $\R$ is just the finite support $X$-power of Cohen forcing. 
\item 
\label{restricted-memory-example-2}
Let $\al_1>\al_0$ be the next element of $C$ above $\al_0$. Let $\ps_{\al_1}$ be the simplified partition product of part (\ref{restricted-memory-example-1}) of the example, with $X=\omega_1$, viewed as a restricted memory iteration. This adds $\omega_1$ Cohen reals. Let $\dot{\Q}_{\al_1}$ be a $\ps_{\al_1}$-name for the poset adding a real almost disjoint from the $\omega_1$ Cohen reals added by $\ps_{\al_1}$.

For $\eta\in\om_2\bsl\om_1$, consider the poset that adds $\eta$ Cohen reals and then a real almost disjoint from them. This is (enrichable to) a simplified partition product over the alphabet so far. To see this, let $X=\eta+1$. For $\xi<\eta$ set $\ind(\xi)=\al_0$, $\dot{\mathbb U}_\xi=\dot{\Q}_{\al_0}$, $b(\xi)=\emptyset$, and $\pi(\xi)=\emptyset$. Set $\ind(\eta)=\al_1$, $b(\eta)=\eta$, $\pi_\eta\colon \omega_1\lra \eta$ a bijection, and $\dot{\mathbb U}_\eta=\pi_\eta(\dot{\Q}_{\al_1})$. 

\item 
\label{restricted-memory-example-3}
Finally, we show how to build a partition product with many different copies of $\dot{\Q}_{\al_1}$. Fix a sequence $\vec{B}=\la B_\gamma\mid \gamma<\omega_2\ra$ where each $B_\ga\seq\om_2$ has size $\omega_1$, and fix bijections $\tau_\gamma\colon\omega_1\lra B_\gamma$. Consider the poset that adds $\omega_2$ Cohen reals, and then for each $\gamma<\omega_2$, adds a real almost disjoint from the Cohen reals added at coordinates in $B_\gamma$. We check that this is (enrichable to) a simplified partition product. To see this, set $X=\omega_2+\omega_2$. For $\xi<\omega_2$ set $\ind(\xi)=\al_0$, $\dot{\mathbb U}_\xi=\dot{\Q}_{\al_0}$, $b(\xi)=\emptyset$, and $\pi(\xi)=\emptyset$. For $\xi=\omega_2+\gamma\in[\omega_2,\omega_2+\omega_2)$, set $\ind(\xi)=\al_1$, $b(\xi)=B_\gamma$, $\pi_\xi=\tau_\gamma$, and $\dot{\mathbb U}_\xi=\pi_\xi(\dot{\Q}_{\al_1})$. 

\item 
\label{restricted-memory-example-4}
Let $\R$ be the poset of the simplified partition product of the previous item. We can make quite a bit of product-like behavior appear in $\R$, both pure product and product over a common initial part, by tailoring how the $B_\ga$ overlap. For instance, if $B_\ga\cap B_\de=\es$, then $\R$ will contain an isomorphic copy of the product $(\ps_{\al_1}\ast\dot{\Q}_{\al_1})\times (\ps_{\al_1}\ast\dot{\Q}_{\al_1})$. On the other hand, if $B_\ga=B_\de$, then $\R$ will contain a copy of $\ps_{\al_1}\ast(\dot{\Q}_{\al_1}\times\dot{\Q}_{\al_1})$. In between the extremes of disjointness and equality of $B_\ga$ and $B_\de$, we can also have non-trivial overlaps of any kind. 
\end{enumerate}
\end{example}

Partition products are simplified partition products satisfying additional constraints on how the memory sets $b_\R(\xi)$ overlap and how the bijections $\pi^\R_\xi$ interact in cases of overlap. These constraints will help us limit the number of partition products in a certain class over any fixed alphabet, up to rearrangement of course. This is done in Subsection \ref{sec.counting} and is an important part of our construction of preassignments of colors. At the same time, these constraints are not so bad as to prevent us from constructing the partition product posets we need for proving the consistency of $\mathsf{OCA}_{ARS}$ with large continuum.

\begin{definition}
\label{simplified-partition-product-alphabet}
A \emph{basic alphabet} on a set $C\subseteq \omega_2\setminus \omega$ is a pair of sequences  $\underline{\ps}=\la\ps_\de:\de\in C\ra$ and  $\underline{\dot{\Q}}=\la\dot{\Q}_\de:\de\in C\ra$ so that:  
\begin{enumerate}
\item 
each $\ps_\delta$ is a simplified partition product based upon $\underline{\ps}\res\delta$ and $\underline{\dot{\Q}}\res\delta$, and $\dot{\Q}_\de$ is a $\ps_\de$-name for a poset; and
\item
for each $\delta\in C$, $\dom(\ps_\delta)$ is an ordinal $\rho_\delta\leq\delta^+$.
\end{enumerate}
A \emph{collapsing system} for the basic alphabet is a sequence $\vec{\varphi}=\la\varphi_{\delta,\mu}\mid \delta\in C, \mu<\rho_\delta\ra$ so that $\varphi_{\delta,\mu}\colon \delta\lra \mu$ is surjective.  
\end{definition}

Compared with the restricted memory alphabet in Definition \ref{restricted-memory-alphabet}, here each $\ps_\delta$ carries not only a memory function, but also an index function and a rearrangement function, that build $\ps_\delta$ from the components in the alphabet up to $\delta$.

\begin{definition}
\label{coherent-collapse}
Let $\underline{\ps}$ and $\underline{\dot{\Q}}$ be a basic alphabet on $C$, and let $\vec{\varphi}$ be a collapsing system. Let $\bar{\delta}\leq\delta$ both belong to $C$, let $\bar{\mu}<\rho_{\bar{\delta}}$, and $\mu<\rho_\delta$. We say that $A\subseteq \mu$ \emph{coherently collapses} $\la\delta,\mu\ra$ to $\la\bar{\delta},\bar{\mu}\ra$ if:
\begin{enumerate}
\item 
\label{coherent-collapse-critical}
{(Hull)} $A$ is of the form $\varphi_{\delta,\mu}[\bar{\delta}]$;
\item
\label{coherent-collapse-closure}
{(Closure)} $A$ is countably closed in $\mu$, meaning that any limit point of $A$ below $\mu$ of cofinality $\omega$ belongs to $A$; 
\item
\label{coherent-collapse-comp}
{(Collapse)} 
letting $j$ denote the transitive collapse map of $A$, we have that 
$$
j\circ \varphi_{\delta,\mu} \res \bar{\delta} = \varphi_{\bar{\delta},\bar{\mu}}.
$$
\end{enumerate}
\end{definition}

The existence of coherently collapsing sets $A$ with 
$\delta>\bar{\delta}$ depends on a reasonably coherent choice of the collapsing system $\vec{\varphi}$, since condition (\ref{coherent-collapse-comp}) of Definition \ref{coherent-collapse} requires $\varphi_{\delta,\mu}$ to collapse to $\varphi_{\bar{\delta},\bar{\mu}}$. If the iteration lengths $\rho_\delta$ are sufficiently small that $\varphi_{\delta,\mu}$ can be constructed through some recipe that is uniform in $\delta$, then the coherence is easy to achieve. Alternatively, the coherence can be achieved in a universe that satisfies condensation, for example in $L$. This is what we will do in Section \ref{sec.constructing}. A sequence picked generically by initial segments will also satisfy enough coherence for the existence of some coherently collapsing sets with 
$\delta>\bar{\delta}$. No special choice is needed for the existence of coherently collapsing sets with 
$\delta=\bar{\delta}$. 

{The Hull and Closure conditions in Definition \ref{coherent-collapse} are used to identify initial segments of the coherently collapsing set using the system $\vec{\varphi}$ of surjections (for example see Corollary \ref{cor:definable}), and to identify how parts of one base set can sit inside another (see Lemma \ref{lemma:sborder}). Among other things, this leads to a counting argument of certain partition products, up to isomorphism, in Lemma \ref{lemma:M}. The Collapse condition is used in a certain triangle lemma for coherence (Lemma \ref{lemma:cohere}), which is essential later on, for the base case of the construction in Section \ref{sec.preassign}, through its use in Lemma \ref{lemma:productcase}.}

\begin{definition}
\label{partition-product}
Let $\underline{\ps}$ and $\dot{\underline{\Q}}$ be a basic alphabet on $C$, and let $\vec{\varphi}$ be a collapsing system for the alphabet. A \emph{partition product} based upon $\underline{\ps}$ and $\underline{\dot{\Q}}$ (with respect to $\vec{\varphi}$) is a simplified partition product $\R$ (based on the reduction of $\underline{\ps},\underline{\dot{\Q}}$ to a restricted memory alphabet), with memory function $b$, index function $\ind$, and rearrangement functions $\pi_\xi$ say, satisfying the following additional properties:  
\begin{enumerate}
\item
\label{partition-product-pi}
$\pi_\xi$ is an acceptable rearrangement of $\ps_{\ind(\xi)}$ and $\R\res b(\xi)$ is exactly equal to the $\pi_\xi$-rearrangement of $\ps_{\ind(\xi)}$, not only as restricted memory iterations, but as simplified partition products;
\item
\label{partition-product-coherent}
let $\xi_1,\xi_2\in\dom(\R)$. Suppose $\ind(\xi_1)\leq\ind(\xi_2)$, $b(\xi_1)\cap b(\xi_2)\not=\emptyset$, and $\zeta\in b(\xi_1)\cap b(\xi_2)$. Set $b_1=b(\xi_1)$, $b_2=b(\xi_2)$, $\delta_1=\ind(\xi_1)$, $\delta_2=\ind(\xi_2)$, $\mu_1=\pi^{-1}_{\xi_1}(\zeta)$, and $\mu_2=\pi^{-1}_{\xi_2}(\zeta)$. Set $A_1=\pi_{\xi_1}[\mu_1]$ and $A_2=\pi_{\xi_2}[\mu_2]$. We view these as ``initial segments'' up to $\zeta$ of $b_1$ and $b_2$ respectively. Then $A_1\subseteq A_2$, and $\pi_{\xi_2}^{-1}[A_1]$ coherently collapses $\la\delta_2,\mu_2\ra$ to $\la\delta_1,\mu_1\ra$. 
\end{enumerate}

For any base closed $Y\subseteq X$, the restriction of the above partition product to $Y$ is obtained as in Definition \ref{simplified-partition-product}. One can check that this restriction is itself a partition product. 
\end{definition}

\begin{example}
The simplified partition products of items (\ref{restricted-memory-example-1})-(\ref{restricted-memory-example-2}) in Example \ref{restricted-memory-example} are in fact partition products, over the alphabet defined in the example with ${\mathbb P}_{\al_0}$ and ${\mathbb P}_{\al_1}$ viewed as simplified partition products rather than reduced to restricted memory iterations. This is easy to check; the coherent collapse condition is trivial since the various memory sets have empty intersection. In item (\ref{restricted-memory-example-3}) of the example one can obtain a partition product by making sure that the sets $B_\gamma$ intersect at initial segments. Specifically, select $B_\gamma$ and $\tau_\gamma$ so that for every $\gamma,\delta<\omega_2$, there is a $\mu$ so that $B_\gamma\cap B_\delta=\tau_\gamma[\mu]=\tau_\delta[\mu]$. It is then easy to check the coherent collapse condition for item (\ref{restricted-memory-example-3}) of the example.
\end{example}

\begin{remark}
\label{rearrangements-partition-products}
Suppose $\sigma$ is an acceptable rearrangement of a simplified partition product $\R$, and let $\R^*$ be the $\sigma$-rearrangement of $\R$. If $\R$ is a partition product, then so is $\R^*$. Similarly, if $\R^*$ is a partition product, then so is $\R$. 
\end{remark}

It follows from Remark \ref{rearrangements-partition-products} that if a coordinate $\delta$ is actually used as an index in a partition product $\R$, say as $\ind_\R(\xi)$, then $\ps_\delta$ is itself a partition product. Thus there is no loss of generality, in the sense that no partition product posets are lost, if we make the following additional demands on $\underline{\ps}$ and $\underline{\dot{\Q}}$:

\begin{definition}
An \emph{alphabet} $\underline{\ps},\underline{\dot{\Q}}$ is a basic alphabet with the additional property that each $\ps_\delta$ is a partition product (not merely a simplified partition product) based upon $\underline{\ps}\res \delta, \underline{\dot{\Q}}\res\delta$. 
\end{definition}

Working over an alphabet, the definition and remark below summarize all the conditions that go into the definition of partition products:

\begin{definition}
\label{def:support}
Let $\underline{\ps},\underline{\dot{\Q}}$ be an alphabet on a set $C\subseteq\omega_2\setminus\omega$. Let $\vec{\varphi}$ be a collapsing system for the alphabet. Let $\base_\delta$, $\pi^\delta_\mu$, $b_\delta$, and $\ind_\delta$ denote the corresponding functions in the partition product $\ps_\delta$. 

We say that functions $\base$ and $\ind$ \emph{support a partition product on} $X$ \emph{based upon} $\underline{\ps}$ \emph{and} $\underline{\dot{\Q}}$ if:
\begin{enumerate}
\item 
\label{def:support-1}
for each $\xi\in X$, $\ind(\xi)\in C$ and $\base(\xi)$ is a pair $(b(\xi),\pi_\xi)$, where $b(\xi)\seq X\cap\xi$ and $\pi_\xi:\rho_{\ind(\xi)}\lra b(\xi)$ is an acceptable rearrangement of $\ps_{\ind(\xi)}$;
\item 
\label{def:support-2}
let $\xi\in X$, and set $\de=\ind(\xi)$. Let $\mu\in \rho_\delta$ and let $\zeta=\pi_\xi(\mu)\in b(\xi)$. Then $b(\zeta)=\pi_\xi[b_\de(\mu)]$, $\pi_\zeta=\pi_\xi\circ\pi^\de_{\mu}$ and $\ind(\zeta)=\ind_{\de}(\mu)$;
\item 
\label{def:support-3}
let $\xi_1,\xi_2\in X$. Suppose $\ind(\xi_1)\leq\ind(\xi_2)$, $b(\xi_1)\cap b(\xi_2)\not=\emptyset$, and $\zeta\in b(\xi_1)\cap b(\xi_2)$. Set $b_1=b(\xi_1)$, $b_2=b(\xi_2)$, $\delta_1=\ind(\xi_1)$, $\delta_2=\ind(\xi_2)$, $\mu_1=\pi^{-1}_{\xi_1}(\zeta)$, and $\mu_2=\pi^{-1}_{\xi_2}(\zeta)$. Set $A_1=\pi_{\xi_1}[\mu_1]$ and $A_2=\pi_{\xi_2}[\mu_2]$. Then $A_1\subseteq A_2$, and $\pi_{\xi_2}^{-1}[A_1]$ coherently collapses $\la\delta_2,\mu_2\ra$ to $\la\delta_1,\mu_1\ra$. 
\end{enumerate}
\end{definition}

\begin{remark}
\label{partition-product-reformulation}
Let $\underline{\ps},\underline{\dot{\Q}}$ be an alphabet on a set $C\subseteq\omega_2\setminus\omega$. Let $\vec{\varphi}$ be a collapsing system for the alphabet. Then $\R$ is a partition product with domain $X$, based upon $\underline{\ps},\underline{\dot{\Q}}$, with $\base$ and $\ind$ as its base and index functions, iff: 
\begin{enumerate}
\item
$\base$ and $\ind$ support a partition product on $X$ based upon $\underline{\ps}$ and $\underline{\dot{\Q}}$;
\item 
$\R$ consists of all finite partial functions $p$ with $\dom(p)\subseteq X$ so that for all $\xi\in\dom(p)$, $p(\xi)$ is a canonical $\pi_\xi(\ps_{\ind(\xi)})$-name for an element of $\dot{\mathbb U}_\xi=\pi_\xi(\dot{\Q}_{\ind(\xi)})$;
\item
\label{partition-product-reformulation-order}
$q\leq p$ iff $\dom(q)\supseteq \dom(p)$ and for all $\xi\in\dom(p)$, $q\res b(\xi)\Vdash_{\pi_\xi(\ps_{\ind(\xi)})} q(\xi)\leq_{\dot{\mathbb U}_\xi} p(\xi)$. 
\end{enumerate}
Note that $\R$ is determined uniquely from (the alphabet and) the functions $\base$ and $\ind$. The forcing requirement in condition (\ref{partition-product-reformulation-order}) is equivalent to the requirement that $q\res b(\xi) \Vdash_{\R\res b(\xi)} q(\xi)\leq_{\dot{\mathbb U}_\xi} p(\xi)$; indeed one can check inductively that $\R\res b(\xi)$ is exactly $\pi_\xi(\ps_{\ind(\xi)})$. This uses the assumption that $\base$ and $\ind$ support a partition product. 
\end{remark}

The definition of a partition product refers to $C$, the alphabet sequences $\underline{\ps}$ and $\underline{\dot{\Q}}$, and the collapsing system $\vec{\varphi}$. We suppress some or all of these objects when they are understood from the context.

\begin{lemma}\label{lemma:basic} Let $\base$ and $\ind$ support a partition product with domain $X$, and let $\xi\in X$. Then $b(\xi)$ is base-closed, and for each $\zeta\in b(\xi)$, $\ind(\zeta)<\ind(\xi)$.
\end{lemma}
\begin{proof} 
Immediate from conditions (\ref{def:support-1}) and (\ref{def:support-2}) in Definition \ref{def:support}. To get $\ind(\zeta)<\ind(\xi)$ we use the fact that $\ps_{\ind(\xi)}$ is based upon $\underline{\ps}\res\ind(\xi),\underline{\dot{\Q}}\res\ind(\xi)$. 
\end{proof}

\begin{lemma}\label{lemma:baseclosedisgreat} Suppose that $\R$ is a partition product with domain $X$ and that $B\seq X$ is base-closed. Then $\base\res B$ and $\ind\res B$ support a partition product on $B$, and this partition product is exactly $\R\res B$. Moreover, if there is a $\be\in C$ such that $\lb\emph{index}(\xi):\xi\in B\rb\seq\be$, then $\R\res B$ is a partition product based upon $\underline{\ps}\res\be$ and $\underline{\dot{\Q}}\res\be$. Finally, $\R\res B$ is a complete subposet of $\R$. 
\end{lemma}
\begin{proof} 
Clear from the definitions. For the final part, prove that if $p\in \R$, $q\in \R\res B$, and $q\leq p\res B$, then $p^*$ with domain $\dom(p)\cup \dom(q)$, mapping $\xi\in B$ to $q(\xi)$ and $\xi\not\in B$ to $p(\xi)$, is a condition in $\R$ and extends both $p$ and $q$. 
\end{proof}

If $\R$, $X$, and $B$ are as in the previous lemma, and if $G$ is $V$-generic for $\R$, we use $G\res B$ to denote $\lb p\res B:p\in G\rb$, which is $V$-generic for $\R\res B$.

\subsection{Rearranging Partition Products}

It will be convenient later on to rearrange partition products in ways that make specific base closed sets into initial segments of the partition product, and ways that isolate the use of the largest index (if there is one) as a product over the restriction to smaller indexes. This will be done in Lemma \ref{rearrange1}, Corollary \ref{cor:omegashuffle}, and Lemma \ref{lemma:products}. 

We begin with a couple of obvious results, and connections between rearrangements and Mostowski collapses.

\begin{lemma}\label{lemma:RL}\emph{(Rearrangement Lemma)} Suppose that $\R$ is a partition product with domain $X$ and that $\si:X\lra X^*$ is an acceptable rearrangement of $\R$. Recall from Definitions \ref{acceptable-rearrangement} and \ref{rearrangement-simplified-partition-product} that $\sigma(b_\R)(\sigma(\xi))=\sigma[b_\R(\xi)]$, and $\sigma(\base_\R)(\sigma(\xi))=\la \sigma[b_\R(\xi)],\sigma\circ \pi^\R_\xi\ra$. 

Then $\si(\base_\R)$ and $\si(\ind_\R)$ support a partition product, denoted $\si(\R)$, on $X^*$. Moreover, $\si$ lifts to act on conditions in $\R$ and on $\R$-names through the inductive equations that $\sigma(p)(\sigma(\xi))=\sigma(p(\xi))$ and $\sigma(\dot{u})=\{\la \sigma(\dot{v}),\sigma(p)\ra \mid \la \dot{v},p\ra\in \dot{u}\}$. This gives an isomorphism from $\R$ to $\si(\R)$, $\sigma$ lifts to act on $\R$ generics $G$ by $\sigma(G)=\sigma[G]$, and for any $\R$-name $\dot{u}$ we have $\sigma(\dot{u})[\sigma(G)]=\dot{u}[G]$. 
\end{lemma}
\begin{proof} 
This summarizes facts from the previous subsection. 
\end{proof}

\begin{remark}\label{remark:ambiguous} Suppose that $M$ and $M^*$ are transitive, satisfy enough of $\zfc-\mathsf{Powerset}$, and that $\si:M\lra M^*$ is an elementary embedding. Also, suppose that $\R\in M$ is a partition product, say with domain $X$, and that $\R$ is based upon $\underline{\ps}\res\ka$ and $\underline{\dot{\Q}}\res\ka$. Since $\sigma$ is an elementary embedding, the ordinal function $\pi:=\si\res X$ is order-preserving and therefore provides an acceptable rearrangement of $\R$. 

There is now a potential conflict between the $\pi$-rearrangements of conditions in $\R$ and the images of these conditions under the embedding $\si$. However, these two notions are the same if the embedding $\si$ does not move any members of the alphabet $\underline{\ps}\res\ka$ and $\underline{\dot{\Q}}\res\ka$. The next lemma summarizes what we need about this situation and will be used crucially in the final proof of Theorem \ref{theorem:main} in Section 6. For the next lemma, we will continue to use $\pi$ to denote the ordinal function which lifts to act, for instance, on conditions in $\R$, and we will keep $\si$ as the elementary embedding.
\end{remark}

\begin{lemma}\label{lemma:ambiguous} Let $\si:M\lra M^*$, $\R$, $X$,  $\ka$, and $\pi$ be as in Remark \ref{remark:ambiguous}. Further suppose that for each $\de\in C\cap\ka$, $\si$ is the identity on every element of $\ps_\de\ast\dot{\Q}_\de\cup\lb\ps_\de,\dot{\Q}_\de\rb$. Then for each $p\in\R$, $\pi(p)=\si(p)$.

Furthermore, setting $\R^*:=\si(\R)$, $\si[X]$ is a base-closed subset of $\dom(\R^*)$, and $\R^*\res\si[X]$ equals $\pi(\R)$, the $\pi$-rearrangement of $\R$.

Additionally, suppose that $G$ is $V$-generic for $\R$, $G^*$ is $V$-generic for $\R^*$, and $\si$ extends to an elementary embedding $\si^*:M[G]\lra M^*[G^*]$.  Suppose also that $\dot{\tau}$ is an $\R$-name (not necessarily in $M$) and $\pi(\dot{\tau})$ is the $\pi$-rearrangement of $\dot{\tau}$. Then $\pi(\dot{\tau})$ is an $\R^*$-name, and $\dot{\tau}[G]=\pi(\dot{\tau})[G^*]$. Finally, if $\dot{\Q}$ is an $\R$-name in $M$ of $M$-cardinality $<\operatorname{crit}(\si)$ and names a poset contained in $\operatorname{crit}(\si)$, then $\si(\dot{\Q})=\pi(\dot{\Q})$.
\end{lemma}
\begin{proof} We only prove the second and third parts. For the second part, fix some $\xi\in X$. Then $b_\R(\xi)$ is in bijection, via a bijection in $M$, with some $\rho_\alpha$, for $\alpha<\kappa$. However, $\rho_\alpha$ is below $\operatorname{crit}(\si)$, since $\si$ is the identity on $\ps_\alpha$. Therefore,
$$
b_{\R^*}(\si(\xi))=\si(b_\R(\xi))=\si[b_\R(\xi)],
$$
where the first equality holds by the elementarity of $\si$ and the second since $\operatorname{crit}(\si)>|b_\R(\xi)|$. This implies that $\si[X]$ is base-closed, and therefore $\R^*\res\si[X]$ is a partition product by Lemma \ref{lemma:baseclosedisgreat}. By the first part of the current lemma, we see that every condition in $\R^*\res\si[X]$ is in the image of $\si$. However, $\pi(p)=\si(p)$ for each condition $p\in\R$, and consequently $\R^*\res\si[X]$ equals $\pi(\R)$, the $\pi$-rearrangement of $\R$.

For the third part, let $G$ and $G^*$ be as in the statement of the lemma. Also let $\pi(G)$ denote the $\pi$-rearrangement of the filter $G$, as defined in Lemma \ref{lemma:RL}, so that $\dot{\tau}[G]=\pi(\dot{\tau})[\pi(G)]$. We also see that $\pi(\dot{\tau})$ is an $\R^*$-name, since it is a $\pi\left(\R\right)$-name and since, by the second part of the current lemma, $\pi\left(\R\right)=\R^*\res\si[X]$ and $\si[X]$ is base-closed. Furthermore, $\si[G]$ (the pointwise image) is a subset of $G^*$, by the elementarity of $\si^*$. However, by the first part of the current lemma, $\si[G]=\lb\si(p):p\in G\rb=\lb\pi(p):p\in G\rb=\pi\left(G\right)$, and therefore 
$$
\dot{\tau}[G]=\pi(\dot{\tau})[\pi\left(G\right)]=\pi(\dot{\tau})[G^*].
$$
Finally, if $\dot{\Q}\in M$ and satisfies the assumptions in the statement of the lemma, then $\si(\dot{\Q})=\si[\dot{\Q}]$, and $\si[\dot{\Q}]=\pi(\dot{\Q})$. This completes the proof of the lemma.
\end{proof}

Before we give applications of the Rearrangement Lemma, we record our definition of an embedding.

\begin{definition}\label{def:embedding} Suppose that $\R$ and $\R^*$ are partition products with respective domains $X$ and $X^*$. We say that an injection $\si:X\lra X^*$ \emph{embeds} $\R$ into $\R^*$ if $\si:X\lra\ran(\si)$ is an acceptable rearrangement of $\R$, and if $\si\left(\base_\R\right)=\base_{\R^*}\res\ran(\si)$ and $\si\left(\ind_\R\right)=\ind_{\R^*}\res\ran(\si)$.
\end{definition}

It is straightforward to check that if $\si$ is an embedding as in Definition \ref{def:embedding}, and if $G^*$ is generic over $\R^*$, then the filter $\si^{-1}\left(G^*\right):=\lb p\in\R:\si(p)\in G^*\rb$ is generic over $\R$. We also remark that, in the context of the above definition, $\si\left(\R\right)=\R^*\res\ran(\si)$.

\begin{lemma}\label{rearrange1} Suppose that $\R$ is a partition product with domain $X$ and $B\seq X$ is base-closed. Then $\R$ is isomorphic to a partition product $\R^*$ with a domain $X^*$ such that $B$ is an initial segment of $X^*$ and $\R^*\res B=\R\res B$.
\end{lemma}
\begin{proof} We define a map $\si$ with domain $X$ which will lift to give us $\R^*$. Let $\xi\in X$. If $\xi\in B$, then set $\si(\xi)=\xi$. On the other hand, if $\xi\in X\bsl B$, say that $\xi$ is the $\ga$th element of $X\bsl B$, then we define $\si(\xi)=\sup(X)+1+\ga$. 

We show that $\si$ is an acceptable rearrangement of $\R$, and then we may set $\R^*:=\si\left(\R\right)$ by Lemma \ref{lemma:RL}. So suppose that $\zeta,\xi\in X$ and $\zeta\in b(\xi)$; we check that $\si(\zeta)<\si(\xi)$. There are two cases. On the one hand, if $\xi\in B$, then $b(\xi)\seq B$, since $B$ is base-closed, and therefore $\zeta\in B$. Then $\si(\zeta)=\zeta<\xi=\si(\xi)$. On the other hand, if $\xi\notin B$, then either $\zeta\in B$ or not. If $\zeta\in B$, then $\si(\zeta)=\zeta<\sup(X)+1\leq\si(\xi)$, and if $\zeta\notin B$, then $\si(\zeta)<\si(\xi)$ since $\si$ is order-preserving on $X\bsl B$.
\end{proof}

It will be helpful later on to know that we can apply Lemma \ref{rearrange1} $\om$-many times, as in the following corollary.

\begin{corollary}\label{cor:omegashuffle} Suppose that $\R$ is a partition product with domain $X$ and that for each $n<\om$, $\pi_n$ is an acceptable rearrangement of $\R$. Suppose that $\la B_n:n\in\om\ra$ is a $\seq$-increasing sequence of base-closed subsets of $X$ where $B_0=\es$ and where $X=\bigcup_n B_n$. Then there is a partition product $\R^*$ which has domain an ordinal $\rho^*$ and an acceptable rearrangement $\si:X\lra \rho^*$ of $\R$ which lifts to an isomorphism of $\R$ onto $\R^*$ and which also satisfies that for each $n<\om$, $\si[B_n]$ is an ordinal and $\pi_n\circ\si^{-1}$ is order-preserving on $\si[B_{n+1}\bsl B_n]$.
\end{corollary}
\begin{proof} We aim to recursively construct a sequence $\la\R_n:n<\om\ra$ of partition products, where $\R_n$ has domain $X_n$, and a sequence $\la\si_n:n<\om\ra$ of bijections, where $\si_n:X\lra X_n$, so that
\begin{enumerate}
\item $\si_n$ is an acceptable rearrangement of $\R$;
\item $\si_n[B_n]$ is an ordinal, and in particular, an initial segment of $X_n$;
\item for each $k<m<\om$, $\si_k[B_k]=\si_m[B_k]$;
\item for each $n<\om$, $\pi_n\circ\si_{n+1}^{-1}$ is order-preserving on $\si_{n+1}[B_{n+1}\bsl B_n]$.
\end{enumerate}

Suppose that we can do this. Then we define a map $\si$ on $X$, by taking $\si(\xi)$ to be the eventual value of the sequence $\la\si_n(\xi):n<\om\ra$; we see that this sequence is eventually constant by (3) and the assumption that $\bigcup_n B_n=X$. By (2) and (3), $\si[B_n]$ is an ordinal, for each $n<\om$, and therefore the range of $\si$ is an ordinal, which we call $\rho^*$. Furthermore, $\pi_n\circ\si^{-1}$ is order-preserving on $\si[B_{n+1}\bsl B_n]$ by (4), and since $\si$ and $\si_{n+1}$ agree on $B_{n+1}$. Finally, by (1) we see that $\si$ is an acceptable rearrangement of $\R$, and we thus take $\R^*$ to be the partition product isomorphic to $\R$ via $\si$, by Lemma \ref{lemma:RL}.

We now show how to create the above objects. Suppose that $\la\R_m:m<n\ra$ and  $\la\si_m:m<n\ra$ have been constructed. If $n=0$, we take $\R_0=\R$ and $\si_0$ to be the identity; since $B_0=\es$, this completes the base case. So suppose $n>0$. Apply Lemma \ref{rearrange1} to the partition product $\R_{n-1}$  and the base-closed subset $\si_{n-1}[B_n]$ of $X_{n-1}$ to create a partition product $\R_n$ on a set $X_n$ which is isomorphic to $\R_{n-1}$ via the acceptable rearrangement $\tau_n:X_{n-1}\lra X_n$ and which satisfies that $\si_{n-1}[B_n]$ is an initial segment of $X_n$. Since $\si_{n-1}[B_{n-1}]$ is an ordinal, by (2) applied to $n-1$, and since $\si_{n-1}[B_n]$ is an initial segment of $X_n$, we see that $\tau_n$ is the identity on $\si_{n-1}[B_{n-1}]$. Also, by composing $\tau_n$ with a further function and relabeling if necessary, we may assume that $\pi_{n-1}\circ\tau_n^{-1}$ just shifts the ordinals in $\si_{n-1}[B_n\bsl B_{n-1}]$ in an order-preserving way and that $\tau_n\circ\si_{n-1}[B_n]$ is an ordinal. We now take $\si_n$ to be $\tau_n\circ\si_{n-1}$, and we see that $\si_n$ and $\R_n$ satisfy the recursive hypotheses. 
\end{proof}

\begin{lemma}\label{lemma:products} Suppose that $\be\in C\cap\ka$ and that $\R$ is a partition product with domain $X$ based upon $\underline{\ps}\res(\be+1)$ and $\underline{\dot{\Q}}\res(\be+1)$. Then, letting $B:=\lb\xi\in X:\ind(\xi)<\be\rb$ and $I:=\lb\xi\in X:\ind(\xi)=\be\rb$, $B$ is base-closed, and $\R$ is isomorphic to 
$$
(\R\res B)\ast\prod_{\xi\in I}\dot{\Q}_\be\left[\pi_\xi^{-1}\left(\dot{G}_B\res b(\xi)\right)\right],
$$
where $\dot{G}_B$ is the canonical $\R\res B$-name for the generic filter.
\end{lemma}
\begin{proof} To see that $B$ is base-closed, fix $\xi\in B$. Then for all $\zeta\in b(\xi)$, $\ind(\zeta)<\ind(\xi)<\be$ by Lemma \ref{lemma:basic}, and so $\zeta\in B$. Thus by Lemma \ref{rearrange1}, we may assume that $B$ is an initial segment of $X$, and hence $I$ is a tail segment of $X$. Now let $G_B$ be generic for $\R\res B$, and for each $\xi\in I$, let $G_{B,\xi}$ denote $\pi_\xi^{-1}\left(G_B\res b(\xi)\right)$, which is generic for $\ps_\be$. The sequence of posets $\la \dot{\Q}_\be[G_{B,\xi}]:\xi\in I\ra$ is in $V[G_B]$, and consequently the finite support iteration of $\la \dot{\Q}_\be[G_{B,\xi}]:\xi\in I\ra$ in $V[G_B]$ is isomorphic to the (finite support) product $\prod_{\xi\in I}\dot{\Q}_\be[G_{B,\xi}]$. Therefore, in $V$, $\R$ is isomorphic to the poset in the statement of the lemma.
\end{proof}

\begin{remark} The previous lemma shows that a partition product does indeed have product-like behavior, and it is part of the justification for our term ``partition product."
\end{remark}

\begin{lemma}
\label{lemma.limit-ccc}
Suppose that $\R$ is a partition product with domain $X$ based upon $\underline{\ps}\res\kappa$ and $\underline{\dot{\Q}}\res\kappa$. Suppose $\kappa$ is a limit ordinal, and let $\langle \kappa_\alpha \mid \alpha<\delta\rangle$ be cofinal in $\kappa$. Finally, let $B_\alpha:=\lb\xi\in X :  \ind(\xi)<\alpha\rb$. Then each $B_\alpha$ is base closed, $\R\res B_\alpha$ is a complete subposet of $\R\res B_\beta$ when $\alpha\leq\beta$, and $\R$ is the direct limit of the posets $\R\res B_\alpha$. In particular if each $\R\res B_\alpha$ is c.c.c., then so is $\R$. 
\end{lemma}
\begin{proof}
Base closure follows as in Lemma \ref{lemma:products}. By Lemma \ref{lemma:baseclosedisgreat}, $\R\res B_\alpha$ is a complete subposet of $\R\res B_\beta$ when $\alpha\leq\beta$. $\R$ is the union of the posets $\R\res B_\alpha$ since $\bigcup_{\alpha<\delta} B_\alpha = X$. The countable chain condition is preserved under this union since the supports are finite. 
\end{proof}

\subsection{Further Remarks on Coherent Overlaps} 

In this subsection we state and prove a few consequences of the 
Hull and Closure conditions (\ref{coherent-collapse-critical}), (\ref{coherent-collapse-closure}) 
of Definition \ref{coherent-collapse}. These results will, in combination with the ability to rearrange a partition product, allow us to find isomorphism types of sufficiently simple partition products inside many countably-closed $M\prec H(\om_3)$, as well as their transitive collapses (see Lemma \ref{lemma:M}).

\begin{lemma}\label{lemma:definable} Let $\R$ be a partition product, say with domain $X$, based upon $\underline{\ps}$ and $\underline{\dot{\Q}}$. Let $\xi_1,\xi_2\in X$, set $\de_i=\ind(\xi_i)$, for $i=1,2$, and suppose that $\de_1\leq\de_2$ and $\rho_{\de_2}<\omega_3$. Finally, let $A=\pi_{\xi_2}^{-1}[b(\xi_1)\cap b(\xi_2)]$. Then $A$ is definable in $H(\om_3)$ from $\vec{\vp}$, the ordinals $\de_1$ and $\de_2$, and any cofinal $Z\seq A$.
\end{lemma}
\begin{proof} Let $Z\seq A$ be cofinal. For each $\al\in Z$, we have from Definition \ref{def:support}(\ref{def:support-3}) and condition (\ref{coherent-collapse-critical}) of Definition \ref{coherent-collapse} that $A\cap\al=\vp_{\de_2,\al}[\de_1]$. Therefore $A=\bigcup_{\al\in Z}\vp_{\de_2,\al}[\de_1]$, which is 
definable in $H(\omega_3)$ from $Z$, $\vec{\varphi}$, $\delta_1$, and $\delta_2$.
\end{proof}

\begin{corollary}\label{cor:definable} Let $\R$, $X$, $\xi_1,\xi_2$, and $A$ be as in Lemma \ref{lemma:definable}. Assume that for all $\xi\in C$, $\rho_\xi<\om_2$. Let $M\prec H(\om_3)$ be countably-closed containing the objects $\underline{\ps}$, $\underline{\dot{\Q}}$, $\vec{\vp}$, and $\de_1,\de_2$. Then $A$ is a member of $M$ as well as the transitive collapse of $M$.
\end{corollary}
\begin{proof} First observe that $A$ is a subset of $\rho_{\de_2}$, which is a member of $M$. Since $\rho_{\de_2}<\om_2$ and $M$ contains $\om_1$ as a subset, $\rho_{\de_2}\seq M$. In particular, $\sup(A)$ is an element of $M$.

Consider the case that $\sup(A)$ has countable cofinality. Then by the countable closure of $M$, we can find a cofinal subset $Z$ of $A$ inside $M$. By Lemma \ref{lemma:definable}, we then conclude that $A\in M$.

Now suppose that $\sup(A)$ has uncountable cofinality. Recall from condition (\ref{coherent-collapse-closure}) of Definition \ref{coherent-collapse} that $A$ is countably closed in $\sup(A)$. Moreover, since $A\cap\al=\vp_{\de_2,\al}[\de_1]$ for each $\al\in A$, we know that the sequence of sets $\la\vp_{\de_2,\al}[\de_1]:\al\in A\ra$ is $\seq$-increasing. By the elementarity of $M$, we may find an $\om$-closed, cofinal subset $Z$ of $\sup(A)$ such that $Z\in M$ for which the sequence of sets $\la\vp_{\de_2,\al}[\de_1]:\al\in Z\ra$ is $\seq$-increasing. Combining this with the fact that $Z\cap A$ is also $\om$-closed and cofinal in $\sup(A)$, we have that
$$
A=\bigcup_{\al\in A\cap Z}\vp_{\de_2,\al}[\de_1]=\bigcup_{\al\in Z}\vp_{\de_2,\al}[\de_1],
$$
and hence $A$ is in $M$, as $\bigcup_{\al\in Z}\vp_{\de_2,\al}[\de_1]$ is in $M$ by elementarity. Finally, since $A$ is bounded in the ordinal $M\cap\om_2$, $A$ is fixed by the transitive collapse map.
\end{proof}

\section{Combining Partition Products}

In this section, we develop the machinery necessary to combine partition products in various ways. This will be essential for later arguments where, in the context of working with a partition product $\R$, we will want to create another partition product $\R^*$ into which $\R$ embeds in a variety of ways. Forcing with $\R^*$ will then add plenty of generics for $\R$, with various amounts of agreement or mutual genericity.

The main result of this section is a so-called ``grafting lemma" which gives conditions under which, given partition products $\ps$ and $\R$, we may extend $\R$ to another partition product $\R^*$ in such a way that $\R^*$ subsumes an isomorphic image of $\ps$; in this case $\ps$ is, in some sense, ``grafted onto" $\R$. One trivial way of doing this, we will show, is to take the partition product $\ps\times\R$. However, the issue becomes  somewhat delicate if we desire, as later on we often will, that $\R$ and the isomorphic copy of $\ps$ in $\R^*$ have coordinates in common, and hence share some part of their generics. Doing so requires that we keep track of more information about the structure of a partition product, and we begin with the relevant definition in the first subsection.

\subsection{Shadow Bases}

\begin{definition}\label{def:shadowbase} A triple $\la x,\pi_x,\al\ra$ is said to be a \emph{shadow base} if the following conditions are satisfied: $\al\in C$, $\pi_x$ has domain $\ga_x$ for some $\ga_x\leq\rho_\al$, and $\pi_x:\ga_x\lra x$ is an acceptable rearrangement of $\ps_\al\res\ga_x$.

Moreover, if $\R$ is a partition product, say with domain $X$, we say that a shadow base $\la x,\pi_x,\al\ra$ is an $\R$-\emph{shadow base} if $x\seq X$ is base-closed and if $\pi_x$ embeds $\ps_\al\res\ga_x$ into $\R\res x$.
\end{definition}

For example, if $\R$ is a partition product with domain $X$, then for any $\xi\in X$ the triple $\la b(\xi),\pi_\xi,\ind(\xi)\ra$ is an $\R$- ``shadow" base; this is part of the motivation for the term. In practice, a shadow base will be an initial segment, in a sense we will specify soon, of such a triple.

\begin{definition}\label{def:cohere} Suppose that $\la x,\pi_x,\al\ra$ and $\la y,\pi_y,\be\ra$ are two shadow bases. We say that they \emph{cohere} if the following holds: suppose that $\al\leq\be$ and that there is some $\zeta\in x\cap y$. Define $\mu_x:=\pi_x^{-1}(\zeta)$ and $\mu_y:=\pi_y^{-1}(\zeta)$. Then 
\begin{enumerate}
\item $\pi_x[\mu_x]\seq\pi_y[\mu_y]$; and
\item $\pi_y^{-1}[\pi_x[\mu_x]]$ coherently collapses $\la\be,\mu_y\ra$ to $\la\al,\mu_x\ra$.
\end{enumerate}
A collection $\mathcal{B}$ of shadow bases is said to \emph{cohere} if any two elements of $\mathcal{B}$ cohere.
\end{definition}

Note that with this definition, item (3) of Definition \ref{def:support} could be rephrased as saying that the two shadow bases $\la b(\xi_1),\pi_{\xi_1},\ind(\xi_1)\ra$ and $\la b(\xi_2),\pi_{\xi_2},\ind(\xi_2)\ra$ cohere.

\begin{remark}\label{remark:definable} It is straightforward to check that Corollary \ref{cor:definable} holds for shadow bases too, in the following sense. Suppose that $\la x,\pi_x,\al\ra$ and $\la y,\pi_y,\be\ra$ are two coherent shadow bases, say with $\al\leq\be$. Then $\pi_y^{-1}[x\cap y]$ is a member of any $M$ as in the statement of Corollary \ref{cor:definable}, provided that $\al$ and $\be$, as well as the additional parameters $\underline{\ps}\res\be$, $\underline{\dot{\Q}}\res\be$, and $\vec{\vp}$, are all in $M$.
\end{remark}

\begin{definition}\label{def:extends} Given a shadow base $\la x,\pi_x,\al\ra$ and some $a\seq x$, we say that $a$ is an \emph{initial segment} of $\la x,\pi_x,\al\ra$ if $a$ is of the form $\pi_x[\mu]$ for some $\mu\leq\dom(\pi_x)$.

Given two shadow bases $\la {x_0},\pi_{x_0},\al_0\ra$ and $\la x,\pi_x,\al\ra$, we say that  $\la {x_0},\pi_{x_0},\al_0\ra$ is an \emph{initial segment} of $\la x,\pi_x,\al\ra$ if $\al_0=\al$, $x_0$ is an initial segment of $\la x,\pi_x,\al\ra$, and $\pi_x\res\dom(\pi_{x_0})=\pi_{x_0}$.
\end{definition}

\begin{remark}\label{remark:extends} A simple but useful observation is that if $\la x_0,\pi_{x_0},\al\ra$ and $\la y,\pi_y,\be\ra$ are two coherent shadow bases, $\la x_0,\pi_{x_0},\al\ra$ is an initial segment of $\la x,\pi_x,\al\ra$, and $(x\bsl x_0)\cap y=\es$, then $\la x,\pi_x,\al\ra$ and $\la y,\pi_y,\be\ra$ cohere.
\end{remark}

\begin{lemma}\label{lemma:initialsegment} Suppose that $\la x,\pi_x,\al\ra$ and $\la y,\pi_y,\be\ra$ are coherent shadow bases and $\al\leq\be$. Then $\pi_x^{-1}[x\cap y]$ is an ordinal $\leq\dom(\pi_x)$, and hence $x\cap y$ is an initial segment of $\la x,\pi_x,\al\ra$.
\end{lemma}
\begin{proof} Fix $\xi\in x\cap y$. By the definition of coherence and the fact that $\al\leq\be$, we see that $\pi_x^{-1}(\xi)+1\seq\pi_x^{-1}[x\cap y]$. Thus 
$$
\pi_x^{-1}[x\cap y]=\bigcup_{\xi\in x\cap y}(\pi_x^{-1}(\xi)+1),
$$
and therefore $\pi_x^{-1}[x\cap y]$ is an ordinal.
\end{proof}

\begin{lemma}\label{lemma:sborder} Suppose that $\la x,\pi_x,\al\ra$ and $\la y,\pi_y,\be\ra$ are two coherent shadow bases, where $\al\leq\be$. Let $\zeta\in x\cap y$, and define $\mu_x:=\pi_x^{-1}(\zeta)$ and $\mu_y:=\pi_y^{-1}(\zeta)$. Then $\pi_y^{-1}\circ\pi_x$ is an order preserving map from $\mu_x$ into $\mu_y$. In particular, $\mu_x\leq\mu_y$, and $\pi_x^{-1}\circ\pi_y$ is the transitive collapse of $\pi_y^{-1}[\pi_x[\mu_x]]$.
\end{lemma}
\begin{proof} By Definition \ref{def:cohere} (1), we know that $\pi_x[\mu_x]$ is a subset of $\pi_y[\mu_y]$, and so $\pi_y^{-1}\circ\pi_x$ is indeed a map from $\mu_x$ into $\mu_y$. Let us abbreviate $\pi_y^{-1}\circ\pi_x$ by $j$. Suppose that $\zeta<\eta<\mu_x$, and we show $j(\zeta)<j(\eta)$. Set $\zeta_y=j(\zeta)$ and $\eta_y=j(\eta)$. Since $\pi_x(\eta)=\pi_y(\eta_y)\in x\cap y$, Definition \ref{def:cohere} (1) implies that $\pi_x[\eta]\seq\pi_y[\eta_y]$. Next, as $\zeta<\eta$, $\pi_x(\zeta)\in\pi_x[\eta]$, and so $\pi_y(\zeta_y)\in\pi_y[\eta_y]$. Finally, since $\pi_y$ is a bijection we conclude that $\zeta_y\in\eta_y$, i.e., $j(\zeta)<j(\eta)$.
\end{proof}

As a result of the previous lemma, if two coherent shadow bases have the same ``index", then their intersection is an initial segment of both.

\begin{corollary}\label{cor:sbsameindex} Suppose that $\la x,\pi_x,\al\ra$ and $\la y,\pi_y,\al\ra$ are two coherent shadow bases and that $\zeta\in x\cap y$. Then $\zeta_0:=\pi_x^{-1}(\zeta)=\pi_y^{-1}(\zeta)$, and in fact, $\pi_x\res(\zeta_0+1)=\pi_y\res(\zeta_0+1)$. 
\end{corollary}
\begin{proof} Fix $\eta\in x\cap y$. Since both shadow bases have index $\al$, we know from Lemma \ref{lemma:sborder} that $\pi_x^{-1}(\eta)=\pi_y^{-1}(\eta)$. Since this holds for any $\eta\in x\cap y$, the result follows.
\end{proof}

\begin{remark}\label{remark:ord} In the context of Corollary \ref{cor:sbsameindex}, we note that $\pi_x^{-1}[x\cap y]=\pi_y^{-1}[x\cap y]$ is an ordinal $\leq\rho_\al$, and if $x\neq y$, then this ordinal is strictly less than $\rho_\al$.
\end{remark}

We conclude this subsection with a very useful lemma.

\begin{lemma}\label{lemma:cohere} Suppose that $\la x,\pi_x,\al\ra$, $\la y,\pi_y,\be\ra$, and $\la z,\pi_z,\ga\ra$ are shadow bases such that $\al,\be\leq\ga$. Suppose further that $x\cap y\seq z$, that $\la x,\pi_x,\al\ra$ and $\la z,\pi_z,\ga\ra$ cohere, and that $\la y,\pi_y,\be\ra$ and $\la z,\pi_z,\ga\ra$ cohere. Then $\la x,\pi_x,\al\ra$ and $\la y,\pi_y,\be\ra$ cohere.
\end{lemma}
\begin{proof} By relabeling if necessary, we assume that $\al\leq\be$. Let $\zeta\in x\cap y$, and we will show that (1) and (2) of Definition \ref{def:cohere} hold. Define $\mu_x:=\pi_x^{-1}(\zeta)$ and $\mu_y:=\pi^{-1}_y(\zeta)$. As $x\cap y\seq z$, $\zeta\in z$, and therefore we may also define $\mu_z:=\pi^{-1}_z(\zeta)$. Applying the coherence assumptions in the statement of the lemma, we conclude that 
$$
\pi_z^{-1}\left[\pi_x[\mu_x]\right]=\vp_{\ga,\mu_z}[\al]\,\text{ and }\,\pi_z^{-1}[\pi_y[\mu_y]]=\vp_{\ga,\mu_z}[\be].
$$
Since $\al\leq\be$, it then follows that $\pi_x[\mu_x]\seq\pi_y[\mu_y]$.

We next show that $\pi_y^{-1}[\pi_x[\mu_x]]=\vp_{\be,\mu_y}[\al]$. By Lemma \ref{lemma:sborder} applied to the shadow bases $\la y,\pi_y,\be\ra$ and $\la z,\pi_z,\ga\ra$, we conclude that $\pi^{-1}_y\circ\pi_z$, which we abbreviate as $j_{z,y}$, is the transitive collapse of $\pi_z^{-1}[\pi_y[\mu_y]]$. Furthermore, the definition of coherence also implies that $j_{z,y}\circ\vp_{\ga,\mu_z}\res\be=\vp_{\be,\mu_y}$. Since $\al\leq\be$ and since $\pi_z^{-1}[\pi_x[\mu_x]]=\vp_{\ga,\mu_z}[\al]$, we apply $j_{z,y}$ to conclude that $\pi_y^{-1}[\pi_x[\mu_x]]=\vp_{\be,\mu_y}[\al]$.

Now let $j_{y,x}$ denote the transitive collapse of $\pi_y^{-1}[\pi_x[\mu_x]]$; we check that $j_{y,x}\circ\vp_{\be,\mu_y}\res\al=\vp_{\al,\mu_x}$. We also let $j_{z,x}$ be the transitive collapse of $\pi_z^{-1}[\pi_x[\mu_x]]$. From Lemma \ref{lemma:sborder}, we know that $j_{y,x}=\pi_x^{-1}\circ\pi_y$ and $j_{z,x}=\pi_x^{-1}\circ\pi_z$. Thus $j_{z,x}=j_{y,x}\circ j_{z,y}$. Since $j_{z,y}\circ\vp_{\ga,\mu_z}\res\be=\vp_{\be,\mu_y}$ and $\al\leq\be$, we conclude that $\vp_{\al,\mu_x}=j_{y,x}\circ\vp_{\be,\mu_y}\res\al$, completing the proof.
\end{proof}

Note that the proof of Lemma \ref{lemma:cohere} uses the Collapse condition (\ref{coherent-collapse-comp}) of Definition \ref{coherent-collapse} for $y$ and $z$, in order to prove one of the {\em other} conditions, namely the Hull condition (\ref{coherent-collapse-critical}), for $x$ and $y$.

\subsection{Enriched Partition Products} 

In this subsection, we will consider in greater detail how shadow bases interact with partition products. We begin with the following definition.

\begin{definition} Let $\R$ be a partition product with domain $X$. A collection $\mathcal{B}$ of $\R$-shadow bases is said to be \emph{$\R$-full} if for all $\xi\in X$, $\la b(\xi),\pi_\xi,\ind(\xi)\ra\in\mathcal{B}$. $\mathcal{B}$ is said to be an $\R$\emph{-enrichment} if $\mathcal{B}$ is both coherent and $\R$-full.

An \emph{enriched partition product} is a pair $(\R,\mathcal{B})$ where $\mathcal{B}$ is an enrichment of $\R$.
\end{definition}

The next definition is a strengthening of the notion of a  base-closed subset which allows us to restrict an enrichment.

\begin{definition} Let $(\R,\mathcal{B})$ be an enriched partition product with domain $X$. A base-closed subset $B\seq X$ is said to \emph{cohere} with $(\R,\mathcal{B})$ if for all triples $\la x,\pi_x,\al\ra$ in $\mathcal{B}$ and for every $\zeta\in B\cap x$, if $\zeta=\pi_x(\zeta_0)$, say, then $\pi_x[\zeta_0]\seq B$.
\end{definition}

\begin{lemma}\label{lemma:restrictenrichment} Suppose that $(\R,\mathcal{B})$ is an enriched partition product with domain $X$ and that $B\seq X$ coheres with $(\R,\mathcal{B})$. Let $\la x,\pi_x,\al\ra\in\mathcal{B}$, and define $\pi_{x\cap B}$ to be the restriction of $\pi_x$ mapping onto $x\cap B$. Then $\la x\cap B,\pi_{x\cap B},\al\ra$ is a shadow base.

Additionally, if we define 
$$
\mathcal{B}\res B:=\lb\la x\cap B,\pi_{x\cap B},\al\ra:\la x,\pi_x,\al\ra\in\mathcal{B}\rb,
$$
then $(\R\res B,\mathcal{B}\res B)$ is an enriched partition product.
\end{lemma}
\begin{proof} To see that $\la x\cap B,\pi_{x\cap B},\al\ra$ is a shadow base, it suffices to show that $\pi^{-1}_x[x\cap B]$ is an ordinal. This holds since for each $\xi\in x\cap B$, by the coherence of $B$ with $(\R,\mathcal{B})$, $\pi_x^{-1}(\xi)+1\seq \pi_x^{-1}[x\cap B]$.

Now we need to verify that $(\R\res B,\mathcal{B}\res B)$ is an enriched partition product. It is straightforward to see that $\mathcal{B}\res B$ is $(\R\res B)$-full, since $B$ is base-closed and since the base and index functions for $\R\res B$ are exactly the restrictions of those for $\R$. Similarly, we see that each shadow base in $\mathcal{B}\res B$ is in fact an $(\R\res B)$-shadow base. Thus we need to check that any two elements of $\mathcal{B}\res B$ cohere. Fix $\la x,\pi_x,\al\ra$ and $\la y,\pi_y,\be\ra$ in $\mathcal{B}$, and suppose that there exists $\zeta\in (x\cap B)\cap (y\cap B)$. Let $\mu_x<\rho_\al$ be such that $\zeta=\pi_{x\cap B}(\mu_x)$, and let $\mu_y<\rho_\be$ be such that $\zeta=\pi_{y\cap B}(\mu_y)$. Then since $B$ coheres with $(\R,\mathcal{B})$, $\pi_x\res(\mu_x+1)=\pi_{x\cap B}\res(\mu_x+1)$, and similarly $\pi_y\res(\mu_y+1)=\pi_{y\cap B}\res(\mu_y+1)$. Therefore conditions (1) and (2) of Definition \ref{def:cohere} at $\zeta$ follow from their applications to $\la x,\pi_x,\al\ra$ and $\la y,\pi_y,\be\ra$ at $\zeta$.
\end{proof}

\begin{definition} Suppose that $\ps$ and $\R$ are partition products and $\si$ embeds $\ps$ into $\R$. If $\la x,\pi_x,\al\ra$ is a $\ps$-shadow base, we define $\si(\la x,\pi_x,\al\ra)$ to be the triple 
$$
\la\si[x],\si\circ\pi_x,\al\ra.
$$ 
If $\mathcal{B}$ is a collection of $\ps$-shadow bases, we define $\si\left(\mathcal{B}\right):=\lb\si(t):t\in\mathcal{B}\rb$.
\end{definition}

The proof of the following lemma is routine.

\begin{lemma}\label{lemma:mapsshadowbases} Suppose that $\ps$ and $\R$ are partition products, $\si$ embeds $\ps$ into $\R$, and $\mathcal{B}$ is a collection of $\ps$-shadow bases. Then $\si\left(\mathcal{B}\right)$ is a collection of $\R$-shadow bases.
\end{lemma}

The following technical lemma will be of some use later.

\begin{lemma}\label{lemma:meow}Suppose that $\R$ and $\R^*$ are partition products, $\si_1,\si_2$ are embeddings of $\R$ into $\R^*$, and  $\la x,\pi_x,\al\ra$ and $\la y,\pi_y,\be\ra$ are two coherent $\R$-shadow bases, with $\al\leq\be$. Let $a$ be an initial segment of $x$ such that $a\seq y$, $\si_1\res a=\si_2\res a$, and $\si_1[x\bsl a]$ is disjoint from $\si_2[y\bsl a]$. Then $\si_1(\la x,\pi_x,\al\ra)$ and $\si_2(\la y,\pi_y,\be\ra)$ are coherent $\R^*$-shadow bases.
\end{lemma}
\begin{proof} From Lemma \ref{lemma:mapsshadowbases}, we see that $\si_1(\la x,\pi_x,\al\ra)$ and $\si_2(\la y,\pi_y,\be\ra)$ are $\R^*$-shadow bases. Furthermore, if $\zeta^*\in\si_1[x]\cap\si_2[y]$, then $\zeta^*$ must be in $\si_1[a]\cap\si_2[a]$, since $\si_1[x\bsl a]\cap\si_2[y\bsl a]=\es$ and since $\si_1\res a=\si_2\res a$. As the injections $\si_1$ and $\si_2$ are equal on $a$, we then have that $\si_1^{-1}(\zeta^*)=\si_2^{-1}(\zeta^*)=:\zeta$. Thus $\zeta\in x\cap y$, and the coherence of the original triples at $\zeta$ implies the coherence of their images at $\zeta^*$. 
\end{proof}

We next define a notion of embedding for enriched partition products.

\begin{definition} Suppose that $(\ps,\mathcal{B})$ is an enriched partition product with domain $X$, $(\R,\mathcal{D})$ is an enriched partition product with domain $Y$, and $\si:X\lra Y$ is a function. We say that $\si$ \emph{embeds} $(\ps,\mathcal{B})$ \emph{into} $(\R,\mathcal{D})$ if $\si$ embeds $\ps$ into $\R$, as in Definition \ref{def:embedding}, and if $\si\left(\mathcal{B}\right)\seq\mathcal{D}$.\end{definition}

We may now state and prove the Grafting Lemma; proving this lemma is one of the main reasons we introduced shadow bases.

\begin{lemma}\label{lemma:graft}\emph{(Grafting Lemma)} Let $(\ps,\mathcal{B})$ and $(\R,\mathcal{D})$  be enriched partition products with respective domains $X$ and $Y$. Suppose that $\hat{X}\seq X$ coheres with $(\ps,\mathcal{B})$ and that there is a map $\si:\hat{X}\lra Y$ which embeds $(\ps\res\hat{X},\mathcal{B}\res\hat{X})$ into $(\R,\mathcal{D})$.Then there is an enriched partition product $(\R^*,\mathcal{D}^*)$ with domain $Y^*$ such that $Y\seq Y^*$, $\R^*\res Y=\R$, $\mathcal{D}\seq\mathcal{D}^*$, and such that there is an extension $\si^*$ of $\si$ which embeds $(\ps,\mathcal{B})$ into $(\R^*,\mathcal{D}^*)$ and which satisfies $\si^*\left[X\bsl\hat{X}\right]=Y^*\bsl Y$.
\end{lemma}
\begin{proof} We first define the map $\si^*$ extending $\si$: if $\xi\in\hat{X}$, then set $\si^*(\xi):=\si(\xi)$. If $\xi\in X\bsl\hat{X}$, say $\xi$ is the $\ga$th such element, then we set $\si^*(\xi):=\sup(Y)+1+\ga$. Then $\si^*$ is an acceptable rearrangement, since $\hat{X}$ is base-closed. Let $Y^*:=Y\cup\ran(\si^*)$. Recalling that $\si$ embeds $\ps\res\hat{X}$ into $\R$, we know that $\si^*\left(\base_\ps\right)\res\ran(\si)=\base_\R\res\ran(\si)$ and that $\si^*\left(\ind_\ps\right)\res\ran(\si)=\ind_\R\res\ran(\si)$. Thus if we define $\base^*:=\base_\R\cup\,\si^*\left(\base_\ps\right)$ and $\ind^*:=\ind_\R\cup\,\si^*\left(\ind_\ps\right)$, then $\base^*$ and $\ind^*$ are functions.

Before we check that $\base^*$ and $\ind^*$ support a partition product on $Y^*$, we need to check that $\mathcal{D}\cup\si^*\left(\mathcal{B}\right)$ consists of a coherent collection of shadow bases. To facilitate the discussion, we set $\mathcal{B}^*:=\si^*\left(\mathcal{B}\right)$ and $\mathcal{D}^*:=\mathcal{D}\cup\mathcal{B}^*$. So fix $\la x,\pi_x,\al\ra\in\mathcal{B}$ and  $\la y,\pi_y,\be\ra$ in $\mathcal{D}$, and we check that $\la y,\pi_y,\be\ra$ and $\la x^*,\pi_{x^*},\al\ra$ cohere, where $x^*:=\si^*[x]$ and $\pi_{x^*}:=\si^*\circ\pi_x$. By our assumption that $\si$ embeds $(\ps\res\hat{X},\mathcal{B}\res\hat{X})$ into $(\R,\mathcal{D})$, we know that $\la y,\pi_y,\be\ra$ and $\la\si[x\cap\hat{X}],\si\circ\pi_{x\cap\hat{X}},\al\ra$ cohere. However, $\la\si[x\cap\hat{X}],\si\circ\pi_{x\cap\hat{X}},\al\ra$ is an initial segment of $\la x^*,\pi_{x^*},\al\ra$, as in Definition \ref{def:extends}. Therefore by Remark \ref{remark:extends}, since $\si^*[X\bsl\hat{X}]$ is disjoint from $y$, we have that $\la y,\pi_y,\be\ra$ and $\la x^*,\pi_{x^*},\al\ra$ cohere.

We now check that $\base^*$ and $\ind^*$ support a partition product on $Y^*$. Conditions (1) and (2) of Definition \ref{def:support} for $\base^*$ and $\ind^*$ follow because they hold for $\base_\R$ and $\ind_\R$, as well as $\si^*\left(\base_\ps\right)$ and $\si^*\left(\ind_\ps\right)$ individually, and since $\base^*$ and $\ind^*$ are functions. Thus we need to verify condition (3). For this it suffices to check that it holds for $\xi_1\in Y$ and $\xi_2\in Y^*\bsl Y$. Rephrasing, we need to show that the triples $\la b^*(\xi_1),\pi^*_{\xi_1},\ind^*(\xi_1)\ra$ and $\la b^*(\xi_2),\pi^*_{\xi_2},\ind^*(\xi_2)\ra$ cohere. The first triple equals $\la b_\R(\xi_1),\pi^\R_{\xi_1},\ind_\R(\xi_1)\ra$ and so is in $\mathcal{D}$ since $\mathcal{D}$ is $\R$-full. The second triple is in $\mathcal{B}^*$, since it equals $\la\si^*[b_\ps(\hat{\xi}_2)],\si^*\circ\pi^\ps_{\hat{\xi}_2},\ind_\ps(\hat{\xi}_2)\ra$, where $\si^*(\hat{\xi}_2)=\xi_2$. Consequently, both shadow bases are in $\mathcal{D}^*$ and are therefore coherent,  by the previous paragraph. Thus condition (3) of Definition \ref{def:support} is satisfied.

Thus $\base^*$ and $\ind^*$ support a partition product on $Y^*$, which we call $\R^*$. Since the restrictions of $\base^*$ and $\ind^*$ to $Y$ equal $\base_\R$ and $\ind_\R$, respectively, we have that $\R^*\res Y=\R$. Additionally, $\si^*$ embeds $\ps$ into $\R^*$, since $\base^*$ and $\ind^*$ restricted to $\ran(\si^*)$ equal $\si^*\left(\base_\ps\right)$ and $\si^*\left(\ind_\ps\right)$ respectively. Thus it remains to check that $\mathcal{D}^*$ is an enrichment of $\R^*$, and for this, it only remains to check that $\mathcal{D}^*$ is $\R^*$-full. However, $\mathcal{D}$ is $\R$-full, and since $\mathcal{B}$ is $\ps$-full, $\mathcal{B}^*$ is full with respect to $\R^*\res\ran(\si^*)$. Thus $\mathcal{D}^*$ is $\R^*$-full.
\end{proof}

\begin{definition} Let $(\ps,\mathcal{B})$, $(\R,\mathcal{D})$, $(\R^*,\mathcal{D}^*)$, $\hat{X}$, $\si$, and $\si^*$ be as in Lemma \ref{lemma:graft}. We will say in this case that $(\R^*,\mathcal{D}^*)$ is \emph{the extension of} $(\R,\mathcal{D})$ \emph{by grafting} $(\ps,\mathcal{B})$ \emph{over} $\si$, and we will call $\si^*$ the \emph{grafting embedding}.
\end{definition}

Note that as a corollary, we get that the product of two partition products is isomorphic to a partition product; this fact could also be proven directly from the definitions.

\begin{corollary}\label{productgraft} Suppose that $\ps$ and $\R$ are partition products with respective domains $X$ and $Y$. Then $\ps\times\R$ is isomorphic to a partition product $\R^*$. 

In fact, by Lemma \ref{lemma:RL} we may assume that $X\cap Y=\es$, that $\R^*$ is a partition product on $X\cup Y$, and that $\R^*\res X=\ps$ and $\R^*\res Y=\R$. Finally, in this case, if $\mathcal{B}$ and $\mathcal{D}$ are enrichments of $\ps$ and $\R$ respectively, then $\mathcal{B}\cup\mathcal{D}$ is an enrichment of $\R^*$.
\end{corollary}

The following technical lemma gives a situation under which, after creating a single grafting embedding, we may extend a number of other embeddings without further grafting; it will be used later in constructing preassignments (see Lemma \ref{lemma:main}).

\begin{lemma}\label{restoretrim} Let $(\ps,\mathcal{B})$ and $(\R,\mathcal{D})$ be enriched partition products with domains $X$ and $Y$ respectively. Suppose that $X$ can be written as $X= X_0\cup X_1$, where both $ X_0$ and $X_1$ cohere with $(\ps,\mathcal{B})$. Let $\mathcal{F}$ be a finite collection of maps which embed $(\ps\res X_0,\mathcal{B}\res X_0)$ into $(\R,\mathcal{D})$, and suppose that for each $\si_0,\si_1\in\mathcal{F}$, 
$$
\si_0[ X_0\cap X_1]=\si_1[ X_0\cap X_1].
$$
Finally, fix a particular $\si_0\in\mathcal{F}$, let $(\R^*,\mathcal{D}^*)$ be the extension of $(\R,\mathcal{D})$ by grafting $(\ps,\mathcal{B})$ over $\si_0$, and let $\si_0^*$ be the grafting embedding. Then for all $\si\in\mathcal{F}$, the map
$$
\si^*:=\si\cup(\si_0^*\res( X_1\bsl X_0))
$$
embeds $(\ps,\mathcal{B})$ into $(\R^*,\mathcal{D}^*)$.
\end{lemma}
\begin{proof} Fix $\si\in\mathcal{F}$. Before we continue, we note that $\si^*$ and $\si^*_0$ agree on all of $X_1$, since they agree on $X_0\cap X_1$ by assumption and on $X_1\bsl X_0$ by definition. 

We first verify that $\si^*$ provides an acceptable rearrangement  of $\ps$. So let $\zeta,\xi\in X$ so that $\zeta\in b_\ps(\xi)$. If $\xi\in X_0,$ then $\zeta$ is too, since $X_0$ is base-closed.  Then $\si^*(\zeta)=\si(\zeta)<\si(\xi)=\si^*(\xi)$, since $\si$ is an acceptable rearrangement  of $\ps\res X_0$. On the other hand, if $\xi\in X_1$, then $\zeta\in X_1$. Since $\si^*\res X_1=\si^*_0\res X_1$, and $\si^*_0\res X_1$ is an acceptable rearrangement of $\ps\res X_1$, we get that $\si^*(\zeta)<\si^*(\xi)$.

We may now see that $\si^*$ embeds $\ps$ into $\R^*$, as follows: let $\base^*$ and $\ind^*$ be the functions which support $\R^*$. Then $\si^*\left(\ind_\ps\right)$ and $\si^*\left(\base_\ps\right)$ agree with $\ind^*$ and $\base^*$ on $\ran(\si)$, since $\si$ embeds $\ps\res X_0$ into $\R$. Furthermore, $\si^*\left(\ind_\ps\right)$ and $\si^*\left(\base_\ps\right)$ agree with $\ind^*$ and $\base^*$ on $\si^*[X_1]$, since they are equal, respectively, to $\si_0^*\left(\ind_\ps\right)$ and $\si^*_0\left(\base_\ps\right)$ restricted to $\si^*_0[X_1]$. Thus $\si^*\left(\ind_\ps\right)$ and $\si^*\left(\base_\ps\right)$ are equal to the restriction of $\ind^*$ and $\base^*$ to $\ran(\si^*)$, and consequently, $\si^*$ embeds $\ps$ into $\R^*$.

We finish the proof of the lemma by showing that $\si^*\left(\mathcal{B}\right)\seq\mathcal{D}^*$. To see this, fix some $\la x,\pi_x,\al\ra\in\mathcal{B}$. We first claim that either $x\seq X_0$ or $x\seq X_1$. If this is false, then there exist $\al\in x\bsl X_0$ and $\be\in x\bsl X_1$. Since $X_0\cup X_1=X$, we then have $\al\in X_1$ and $\be\in X_0$. We suppose, by relabeling if necessary, that $\al_0:=\pi_x^{-1}(\al)<\pi_x^{-1}(\be)=:\be_0$. By the coherence of $X_0$ with $(\ps,\mathcal{B})$, we conclude that $\pi_x[\be_0]\seq X_0$. However, $\al=\pi_x(\al_0)\in\pi_x[\be_0]$, and therefore $\al\in X_0$, a contradiction.

We now show that the shadow base $\la x^*,\pi_{x^*},\al\ra$ is in $\mathcal{D}^*$, where $x^*:=\si^*[x]$ and $\pi_{x^*}=\si^*\circ\pi_x$. On the one hand, if $x\seq X_0$, then the shadow base $\la x,\pi_x,\al\ra$ is in $\mathcal{B}\res X_0$, and therefore $\la\si[x],\si\circ\pi_x,\al\ra$ is a member of $\mathcal{D}\seq\mathcal{D}^*$. Since $\si=\si^*\res X_0$, $\la x^*,\pi_{x^*},\al\ra=\la\si[x],\si\circ\pi_x,\al\ra$, completing this subcase. On the other hand, if $x\seq X_1$, then we see that $\la x^*,\pi_{x^*},\al\ra=\la\si_0^*[x],\si_0^*\circ\pi_x,\al\ra$, since $\si^*\res X_1=\si_0^*\res X_1$. It is therefore a member of $\mathcal{D}^*$, which finishes the proof.
\end{proof}

\section{Simplifying the Task: Finitely Generated Partition Products}

In this section we begin the basic inductive step of constructing a single highly symmetric preassignment name. We make the following assumption, which we will secure through an induction later on. 

\begin{assumptiona}\label{assumptiona} The $\ch$ holds. $\ka<\omega_2$ is in $C$, and for each $\xi\in C\cap\ka$, $\rho_\xi$ is below $\om_2$. Additionally, the $\ka$-alphabetical partition product $\ps_\ka$ is defined, and in particular, $\ps_\ka$ is a partition product based upon $\underline{\ps}\res\ka$ and $\underline{\dot{\Q}}\res\ka$. We also assume that the $\ps_\ka$-names $\dot{S}_\ka$ and $\dot{\chi}_\ka$ are defined and satisfy that $\dot{S}_\ka$ names a countable basis for a second countable, Hausdorff topology on $\om_1$ and $\dot{\chi}_\ka$ names a coloring, as in the definition of $\mathsf{OCA}_{ARS}$, which is open with respect to the topology generated by $\dot{S}_\ka$.
\end{assumptiona}

Recall that we want substantial symmetry for the poset $\dot{\mathbb Q}_\kappa$. In the introduction we mentioned, for example, that if $G_L\times G_R$ is generic for ${\mathbb P}_\kappa\times {\mathbb P}_\kappa$, then we need $\dot{\mathbb Q}_\kappa[G_L]\times \dot{\mathbb Q}_\kappa[G_R]$ to be c.c.c.\ in $V[G_L\times G_R]$. In fact we need to handle not just interpretations of $\dot{\mathbb Q}_\kappa$ by generics arising through powers of ${\mathbb P}_\kappa$, but any combination of generics arising through copies of ${\mathbb P}_\kappa$ in partition products based upon $\underline{\ps}\res \kappa$ and $\underline{\dot{\Q}}\res\kappa$. We need to know that any product of such interpretations is c.c.c. In other words, we need to construct $\dot{\mathbb Q}_\kappa$ in such a way that any partition product based upon $\underline{\ps}\res(\kappa+1)$ and $\underline{\dot{\bb{Q}}}\res(\kappa+1)$ is c.c.c.

We will inductively assume:

\begin{assumptionb} Any partition product based upon $\underline{\ps}\res\kappa$ and $\underline{\dot{\bb{Q}}}\res\kappa$ is c.c.c.
\end{assumptionb}

We refer to the conjunction of Assumption 4.1(a) and 4.1(b) as Assumption 4.1. Our basic step goal then is to construct, under Assumption 4.1, a ${\mathbb P}_\kappa$-name $\dot{\mathbb Q}_\kappa$ which secures Assumption 4.1(b) at $\kappa+1$. In this section we make a series of reductions to streamline and simplify this goal. We will complete the basic step construction in Section 5, and in Section 6 we will use this as part of an inductive construction of a sequence $\underline{\mathbb P}$ which provides the right building blocks for our main theorem.

\subsection{$\ka$-Suitable Collections}

We now consider how various copies of $\ps_\ka$ fit into a partition product $\R$, where $\R$ is based upon $\underline{\ps}\res\ka$ and $\underline{\dot{\Q}}\res\ka$. Even though we have yet to construct the name $\dot{\Q}_\ka$, we would still like to isolate the relevant behavior of copies of $\ps_\ka$ inside such an $\R$ which these copies \emph{would} have if $\R$ \emph{were} of the form 
$$
\R=\R^*\res\lb\xi\in\dom(\R^*):\ind(\xi)<\ka\rb,
$$
for some partition product $\R^*$ based upon $\underline{\ps}\res(\ka+1)$ and $\underline{\dot{\Q}}\res(\ka+1)$. Put differently, we need a mechanism to enforce that the copies of $\ps_\ka$ behave like they would after we construct $\dot{\Q}_\ka$. This leads to the following definition.

\begin{definition}\label{def:suitable} Let $\R$ be a partition product with domain $X$ based upon $\underline{\ps}\res\ka$ and $\underline{\dot{\Q}}\res\ka$. Let $\lb\la B_\iota,\psi_\iota\ra:\iota\in I\rb$ be a set of pairs, where each $B_\iota\seq X$ is base-closed and where $\psi_\iota:\rho_\ka\lra B_\iota$ is a bijection which embeds $\ps_\ka$ into $\R$. We say that the collection $\lb \la B_\iota,\psi_\iota\ra:\iota\in I\rb$ is $\ka$-\emph{suitable with respect to} $\R$ if
$$
\lb\la B_\iota,\psi_\iota,\ka\ra:\iota\in I\rb\cup\lb\la b(\xi),\pi_\xi,\ind(\xi)\ra:\xi\in X\rb
$$ 
is a coherent set of $\R$-shadow bases.

Moreover, if $(\R,\mathcal{B})$ is an enriched partition product, we say that $\lb\la B_\iota,\psi_\iota\ra:\iota\in I\rb$ is $\ka$-\emph{suitable with respect to} $(\R,\mathcal{B})$ if $\lb\la B_\iota,\psi_\iota,\ka\ra:\iota\in I\rb\seq\mathcal{B}$ and if $\al\leq\ka$ for all $\la x,\pi_x,\al\ra\in\mathcal{B}$.\end{definition}

The remaining results in this subsection will develop in detail a number of technical properties of $\ka$-suitable collections. The main point to keep in mind is that in a $\ka$-suitable collection, any two copies of $\ps_\ka$ agree on some initial segment of $\ps_\ka$ and are disjoint afterwards. In the later construction of preassignments (see Lemma \ref{lemma:main}) this allows a natural induction on the ``height" of the collection, namely, the maximal such agreement.

As the next lemma shows, $\ka$-suitable collections give us subsets which cohere with the original partition product, since the indices of the triples in the enrichment do not exceed $\ka$.

\begin{lemma}\label{lemma:suitablecohere} Suppose that $\lb\la B_\iota,\psi_\iota\ra:\iota\in I\rb$ is $\ka$-suitable with respect to an enriched partition product $(\R,\mathcal{B})$. Then for any $I_0\seq I$, $\bigcup_{\iota\in I_0}B_\iota$ coheres with $(\R,\mathcal{B})$.
\end{lemma}
\begin{proof} Let $\la x,\pi_x,\al\ra\in\mathcal{B}$, and suppose that there exists $\zeta\in (\bigcup_{\iota\in I_0}B_\iota)\cap x$. Fix some $\iota\in I_0$ such that $\zeta\in B_\iota\cap x$. Then $\la B_\iota,\psi_\iota,\ka\ra$ is in $\mathcal{B}$. Furthermore, $\al\leq\ka$, by definition of $\ka$-suitability with respect to $(\R,\mathcal{B})$. Since $\mathcal{B}$ is coherent, by definition of an enrichment, and since $\al\leq\ka$, we have by Definition \ref{def:cohere} that
$$
\pi_x[\pi_x^{-1}(\zeta)]\seq\psi_\iota[\psi_\iota^{-1}(\zeta)].
$$
Since ran$(\psi_\iota)=B_\iota$, this finishes the proof.
\end{proof}

We will often be interested in the following strengthening of the notion of an embedding, one which preserves the $\ka$-suitable structure.

\begin{definition}\label{def:suitablemap} Let $\R$ and $\R^*$ be two partition products, and let $\mathcal{S}=\lb\la B_\iota,\psi_\iota\ra:\iota\in I\rb$ and $\mathcal{S}^*=\lb\la B^*_\eta,\psi^*_\eta\ra:\eta\in I^*\rb$ be $\ka$-suitable collections with respect to $\R$ and $\R^*$ respectively. An embedding $\si$ of $\R$ into $\R^*$ is said to be $(\mathcal{S},\mathcal{S}^*)$-\emph{suitable} if for each $\iota\in I$, there is some $\eta\in I^*$ such that $\si\res B_\iota$ isomorphs $\R\res B_\iota$ onto $\R^*\res B^*_\eta$ and $\psi^*_\eta=\si\circ\psi_\iota$. A collection $\mathcal{F}$ of embeddings is said to be $(\mathcal{S},\mathcal{S}^*)$-\emph{suitable} if each $\si\in\mathcal{F}$ is $(\mathcal{S},\mathcal{S}^*)$-suitable.

If $\si$ is $(\mathcal{S},\mathcal{S}^*)$-suitable, we let $h_\si$ denote the 
map from $I$ into $I^*$ such that 
$\eta=h_\sigma(\iota)$ witnesses suitability at $\iota$, for each $\iota\in I$.
\end{definition}

The following technical lemmas will be used in the next section.

\begin{lemma}\label{lemma:uhh} Suppose that $\lb\la B_\iota,\psi_\iota\ra:\iota\in I\rb$ is $\ka$-suitable with respect to an enriched partition product $(\R,\mathcal{B})$ and that the elements of $\lb B_\iota:\iota\in I\rb$ are pairwise disjoint. Then for any $\la x,\pi_x,\al\ra\in\mathcal{B}$, $x\cap(\bigcup_{\iota\in I}B_\iota)=x\cap B_{\iota_0}$ for a unique $\iota_0\in I$.
\end{lemma}
\begin{proof} Suppose otherwise, and fix $\la x,\pi_x,\al\ra\in\mathcal{B}$ as well as distinct $\iota_0,\iota_1\in I$ such that $x\cap B_{\iota_0}\neq\es$ and $x\cap B_{\iota_1}\neq\es$. Let $\zeta\in x\cap B_{\iota_0}$ and $\eta\in x\cap B_{\iota_1}$. Then $\zeta\neq\eta$, since $B_{\iota_0}\cap B_{\iota_1}=\es$. Define $\zeta_0:=\pi_x^{-1}(\zeta)$ and $\eta_0:=\pi_x^{-1}(\eta)$. Since $\zeta_0\neq\eta_0$, we suppose, by relabeling if necessary, that $\zeta_0<\eta_0$. By definition of an enrichment, we know that $\la B_{\iota_1},\psi_{\iota_1},\ka\ra$ and $\la x,\pi_x,\al\ra$ cohere, and since $\al\leq\ka$ and $\eta\in B_{\iota_1}\cap x$, we conclude that $\pi_x[\eta_0]\seq B_{\iota_1}$. However, $\zeta_0<\eta_0$, and so $\zeta=\pi_x(\zeta_0)\in\pi_x[\eta_0]$, which implies that $\zeta\in B_{\iota_1}$. This contradicts the fact that $B_{\iota_0}\cap B_{\iota_1}=\es$.
\end{proof}

The next lemma gives a sufficient condition for creating an enrichment. It will be used as part of the base case of the main lemma which constructs preassignments (see Lemma \ref{lemma:main}), in which the $\ka$-suitable collection is isomorphic to a product (i.e., there is no overlap between the various copies of $\ps_\ka$ in the $\ka$-suitable collection).

\begin{lemma}\label{lemma:productcase} Suppose that $\mathcal{S}=\lb\la B_\iota,\psi_\iota\ra:\iota\in I\rb$ is $\ka$-suitable with respect to an enriched partition product $(\R,\mathcal{B})$ and that the elements of $\lb B_\iota:\iota\in I\rb$ are pairwise disjoint. Suppose further that $\R^*$ is a partition product with domain $X^*$ and that $\mathcal{S}^*=\lb\la B^*_\eta,\psi^*_\eta\ra:\eta\in I^*\rb$ is $\ka$-suitable with respect to $\R^*$. Finally, set $\hat{X}:=\bigcup_{\iota\in I}B_\iota$, and suppose that there exists a finite collection $\mathcal{F}$ of $(\mathcal{S},\mathcal{S}^*)$-suitable embeddings of $\R\res\hat{X}$ into $\R^*$  such that for any distinct $\iota_0,\iota_1\in I$ and any (not necessarily distinct) $\pi_0,\pi_1\in\mathcal{F}$, 
$$
B^*_{h_{\pi_0}(\iota_0)}\cap B^*_{h_{\pi_1}(\iota_1)}=\es,
$$
where for each $\pi\in\mathcal{F}$, $h_\pi$ is the associated 
map from Definition \ref{def:suitablemap}. Then
$$
\mathcal{B}^*:=\lb\la b^*(\xi),\pi^*_\xi,\ind^*(\xi)\ra:\xi\in X^*\rb\cup\bigcup_{\pi\in\mathcal{F}}\pi\left(\mathcal{B}\res\hat{X}\right)\cup\lb\la B^*_\eta,\psi^*_\eta,\ka\ra:\eta\in I^*\rb
$$
is an enrichment of $\R^*$ and $\mathcal{S}^*$ is $\ka$-suitable with respect to $(\R^*,\mathcal{B}^*)$.
\end{lemma}
\begin{proof} We will first show that $\bigcup_{\pi\in\mathcal{F}} \pi\left(\mathcal{B}\res\hat{X}\right)$ is a coherent collection of $\R^*$-shadow bases. Since each $\pi\in\mathcal{F}$ is an embedding of $\R\res\hat{X}$ into $\R^*$, Lemma \ref{lemma:mapsshadowbases} implies that this is a set of $\R^*$-shadow bases. Thus we check coherence. 

Fix $\pi_0,\pi_1\in\mathcal{F}$ and $\la x,\pi_x,\al\ra,\la y,\pi_y,\be\ra\in\mathcal{B}\res\hat{X}$, and assume, by relabeling if necessary, that $\al\leq\be$. We show that $\la x^*,\pi_{x^*},\al\ra$ and $\la y^*,\pi_{y^*},\be\ra$ cohere, where $x^*:=\pi_0[x]$ and $\pi_{x^*}:=\pi_0\circ\pi_x$, and where $y^*:=\pi_1[y]$, $\pi_{y^*}:=\pi_1\circ\pi_y$. By Lemma \ref{lemma:uhh}, and since $x$ and $y$ are subsets of $\hat{X}$, we may fix $\iota_0,\iota_1\in I$ such that $x=x\cap\hat{X}=x\cap B_{\iota_0}$ and $y=y\cap\hat{X}=y\cap B_{\iota_1}$. There are two cases. 

First suppose that $\iota_0\neq\iota_1$. Then we must have that $x^*\cap y^*=\es$. To see this, observe that
$$
x^*=\pi_0[x]=\pi_0[x\cap B_{\iota_0}]\seq B^*_{h_{\pi_0}(\iota_0)},
$$
and
$$
y^*=\pi_1[y]=\pi_1[y\cap B_{\iota_1}]\seq B^*_{h_{\pi_1}(\iota_1)}.
$$
Therefore $x^*\cap y^*=\es$, as $B^*_{h_{\pi_0}(\iota_0)}\cap B^*_{h_{\pi_1}(\iota_1)}=\es$, by assumption. We thus trivially have the coherence of $\la x^*,\pi_{x^*},\al\ra$ and $\la y^*,\pi_{y^*},\be\ra$ in this case. 

On the other hand, suppose that $\iota:=\iota_0=\iota_1$. Define $a\seq x$ to be the largest initial segment (see Definition \ref{def:extends}) of $\la x,\pi_x,\al\ra$ on which $\pi_0$ and $\pi_1$ agree, and set $a^*:=\pi_0[a]=\pi_1[a]$. In order to see that $\la x^*,\pi_{x^*},\al\ra$ and $\la y^*,\pi_{y^*},\be\ra$ cohere, it suffices, in light of Lemma \ref{lemma:meow}, to show that $\pi_0[x\bsl a]$ is disjoint from $\pi_1[y\bsl a]$. Towards this end, fix some $\zeta^*\in x^*\cap y^*$, and suppose for a contradiction that $\zeta^*\notin a^*$. Define $\mu_x:=\pi_{x^*}^{-1}(\zeta^*)$, and observe that $\mu_x$ is greater than the ordinal $\pi_x^{-1}[a]$, since $\zeta^*\notin a^*$. Using the abbreviation $\eta_i:=h_{\pi_i}(\iota)$, for $i\in\lb 0,1\rb$, we see that $\zeta^*\in B^*_{\eta_0}\cap B^*_{\eta_1}$, as $x^*=\pi_0[x\cap B_\iota]\seq B^*_{\eta_0}$, and as $y^*=\pi_1[y\cap B_\iota]\seq B^*_{\eta_1}$. Set $\zeta_0:=(\psi^*_{\eta_0})^{-1}(\zeta^*)$. Since the $\R^*$-shadow bases $\la B^*_{\eta_0},\psi^*_{\eta_0},\ka\ra$ and $\la B^*_{\eta_1},\psi^*_{\eta_1},\ka\ra$ cohere, Corollary \ref{cor:sbsameindex} implies that $\psi^*_{\eta_0}\res(\zeta_0+1)=\psi^*_{\eta_1}\res(\zeta_0+1)$.

Now we observe that
$$
\pi_{x^*}(\mu_x)=\pi_0(\pi_x(\mu_x))=\zeta^*=\psi_{\eta_0}^*(\zeta_0)=\pi_0(\psi_\iota(\zeta_0)),
$$
and therefore $\pi_x(\mu_x)=\psi_\iota(\zeta_0)$. Let us call this ordinal $\zeta$. Since $\zeta\in B_\iota\cap x$, the coherence of $\la x,\pi_x,\al\ra$ with $\la B_\iota,\psi_\iota,\ka\ra$ and the fact that $\al\leq\ka$ imply that 
$$
\pi_x[\mu_x+1]\seq\psi_\iota[\zeta_0+1].
$$
As noted above, $\psi^*_{\eta_0}\res(\zeta_0+1)=\psi^*_{\eta_1}\res(\zeta_0+1)$, 
and since $\psi^*_{\eta_i}=\pi_i\circ \psi_\iota$ by Definition \ref{def:suitablemap}, $\pi_0$ and $\pi_1$ agree on $\psi_\iota[\zeta_0+1]$. In particular, they agree on $\pi_x[\mu_x+1]$. Thus $\pi_x[\mu_x+1]$ is an initial segment of $\la x,\pi_x,\al\ra$ on which $\pi_0$ and $\pi_1$ agree. Since $\zeta=\pi_x(\mu_x)\notin a$, this contradicts the maximality of $a$.

At this point, we have shown that $\bigcup_{\pi\in\mathcal{F}}\pi\left(\mathcal{B}\res\hat{X}\right)$ is a coherent collection of $\R^*$-shadow bases. We introduce the abbreviation  
$$
\mathcal{B}^*_0:=\lb\la b^*(\xi),\pi^*_\xi,\ind^*(\xi)\ra:\xi\in X^*\rb\cup\lb\la B^*_\eta,\psi^*_\eta,\ka\ra:\eta\in I^*\rb.
$$ 
We know that $\mathcal{B}^*_0$ is a coherent set of $\R^*$-shadow bases, by the definition of $\ka$-suitability. Therefore, to finish showing that $\mathcal{B}^*$ is an enrichment of $\R^*$, we now check that if $\la y,\pi_y,\be\ra\in\mathcal{B}^*_0$, $\pi\in\mathcal{F}$, and $\la x,\pi_x,\al\ra\in\mathcal{B}\res\hat{X}$, then $\la y,\pi_y,\be\ra$ and $\la x^*,\pi_{x^*},\al\ra$ cohere, where $x^*:=\pi[x]$ and $\pi_{x^*}=\pi\circ\pi_x$. By Lemma \ref{lemma:uhh}, let $\iota\in I$ be such that $x=x\cap\hat{X}=x\cap B_\iota$. Then $x^*=\pi[x]=\pi[x\cap B_\iota]\seq B^*_{h_\pi(\iota)}$. Now $\la x,\pi_x,\al\ra$ and $\la B_\iota,\psi_\iota,\ka\ra$ cohere, and moreover,  $\pi$ isomorphs $\R\res B_\iota$ onto $\R^*\res B^*_{h_\pi(\iota)}$ and satisfies that $\psi^*_{h_\pi(\iota)}=\pi\circ\psi_\iota$. It is straightforward to see from this that $\la x^*,\pi_{x^*},\al\ra$ and $\la B^*_{h_\pi(\iota)},\psi^*_{h_\pi(\iota)},\ka\ra$ cohere. However, $\la y,\pi_y,\be\ra$ and $\la B^*_{h_\pi(\iota)},\psi^*_{h_\pi(\iota)},\ka\ra$ also cohere, by definition of $\ka$-suitability. Since $\al,\be\leq\ka$, Lemma \ref{lemma:cohere} therefore implies that $\la y,\pi_y,\be\ra$ and $\la x^*,\pi_{x^*},\al\ra$ cohere, which is what we wanted to show.
\end{proof}

We note that the final case in the proof of Lemma \ref{lemma:productcase} makes use of Lemma \ref{lemma:cohere}, and through this the first essential use of the Collapse condition (\ref{coherent-collapse-comp}) of Definition  \ref{coherent-collapse}.

\subsection{What Suffices: Finitely Generated Partition Products}

Given a (possibly partial) 2-coloring $\chi$ on $\om_1$ and a function $f$ from $\om_1$ into $\lb 0,1\rb$, we use $\Q(\chi,f)$ to denote the poset to decompose $\om_1$ into countably-many $\chi$-homogeneous sets which respect the function $f$. More precisely, a condition is a finite partial function $q$ with $\dom(q)\seq\om$ such that for each $n\in\dom(q)$, $q(n)$ is a finite subset of $\om_1$ on which $f$ is constant, say with value $i$, and $q(n)$ is $\chi$-homogeneous with color $i$, meaning that if $x,y\in q(n)$ and $\langle x,y\rangle\in \dom(\chi)$, then $\chi(x,y)=i$. The ordering is $q_1\leq q_0$ iff $\dom(q_0)\seq\dom(q_1)$, and for each $n\in\dom(q_0)$, $q_0(n)\seq q_1(n)$. This forcing was introduced and extensively analyzed in \cite{ARS}.

Following \cite{ARS}, we refer to any such $f$ as a \emph{preassignment of colors}. Our main goal in this section is to come up with a $\ps_\ka$-name $\dot{f}$ for a particularly nice preassignment of colors for $\dot{\chi}_\ka$, in the following sense:

\begin{proposition}\label{prop:name} There is a $\ps_\ka$-name $\dot{f}$ for a preassignment of colors so that for any partition product $\R$ based upon $\underline{\ps}\res\ka$ and $\underline{\dot{\Q}}\res\ka$, any generic $G$ for $\R$, and any finite collection $\lb\la B_\iota,\psi_\iota\ra:\iota\in I\rb$ which is $\ka$-suitable with respect to $\R$, the poset 
$$
\prod_{\iota\in I}\Q\left(\dot{\chi}_\ka\left[\psi_\iota^{-1}(G\res B_\iota)\right],\dot{f}\left[\psi_\iota^{-1}(G\res B_\iota)\right]\right)
$$ 
is c.c.c.\ in $V[G]$.
\end{proposition}

\begin{remark}  Observe that in the previous proposition, the same name $\dot{f}$ is interpreted in a variety of ways, namely, by various generics for $\ps_\ka$ added by forcing with $\R$. Moreover, $\dot{f}$ is strong enough that the product of the induced homogeneous set posets is c.c.c. This is what we mean by referring to the name as ``symmetric."
\end{remark}

\begin{corollary}\label{cor:name} Let $\dot{f}_\ka$ be a name witnessing Proposition \ref{prop:name}, and set $\dot{\Q}_\ka$ to be the $\ps_\ka$-name $
\Q(\dot{\chi}_\ka,\dot{f}_\ka)$. Then any partition product based upon $\underline{\ps}\res(\ka+1)$ and $\underline{\dot{\Q}}\res(\ka+1)$ is c.c.c.
\end{corollary}
\begin{proof}[Proof of Corollary \ref{cor:name}] Let $\R$ be a partition product based upon $\underline{\ps}\res(\ka+1)$ and $\underline{\dot{\Q}}\res(\ka+1)$, and let $X$ be the domain of $\R$. Set $\hat{X}:=\lb\xi\in X:\ind(\xi)<\ka\rb$, and let $I:=\lb\xi\in X:\ind(\xi)=\ka\rb$. By Lemma \ref{lemma:products}, $\R$ is isomorphic to 
$$
(\R\res\hat{X})\ast\prod_{\xi\in I}\dot{\Q}_\ka\left[\pi_\xi^{-1}\left(\dot{G}\res b(\xi)\right)\right],
$$
and $\R\res \hat{X}$ is a partition product based upon $\underline{\ps}\res\ka$ and $\underline{\dot{\Q}}\res\ka$. By Assumption 4.1, $\R\res\hat{X}$ is c.c.c. It is also straightforward to check that $\lb\la b(\xi),\pi_\xi\ra:\xi\in I\rb$ is $\ka$-suitable, by the definition of $\R$ as a partition product based upon $\underline{\ps}\res(\ka+1)$ and $\underline{\dot{\Q}}\res(\ka+1)$. Finally, from Proposition \ref{prop:name}, we know each finitely-supported subproduct of
$$
\prod_{\xi\in I}\dot{\Q}_\ka\left[\pi_\xi^{-1}\left(G\res b(\xi)\right)\right]
$$ 
is c.c.c.\ in $V[G\res \hat{X}]$, and hence the entire product is c.c.c. Since $\R\res\hat{X}$ is c.c.c.\ in $V$, this finishes the proof.
\end{proof}

As mentioned in the introductory remarks for this section, we will prove Proposition \ref{prop:name} by working backwards through a series of reductions; the final proof of Proposition \ref{prop:name} occurs in Subsection 5.2. We first want to see what happens if a finite product $\prod_{l<m}\Q(\chi_l,f_l)$ is \emph{not} c.c.c., where each $\chi_l$ is an open coloring on $\om_1$ with respect to some second countable, Hausdorff topology $\tau_l$ on $\om_1$ and $f_l:\om_1\lra\lb 0,1\rb$ is an arbitrary preassignment. The specific reductions at this stage are very similar to those of \cite{ARS}.

Thus fix a sequence $\la \tau_l:l<m\ra$ of second countable, Hausdorff topologies on $\om_1$ with respective open colorings $\la\chi_l:l<m\ra$ and preassignments $\la f_l:l<m\ra$. Recalling the notational remarks at the end of the introduction, let us define $\tau:=\biguplus \tau_l$, a topology on $X:=\biguplus_{l<m}\om_1$, as well as  $f:=\biguplus f_l$ and $\chi:=\biguplus\chi_l$. So, for example, if $x\in X$, then $f(x)=f_l(x)$, where $l$ is unique s.t. $x$ is in the $l$th copy of $\om_1$, and if $x,y\in X$ then $\chi(x,y)$ is defined iff $x$ and $y$ are distinct and belong to the same copy of $\om_1$, say the $l$th, and in this case, $\chi(x,y)=\chi_l(x,y)$. With this notation, we may view a condition in the product $\prod_{l<m}\Q(\chi_l,f_l)$ as a condition in ${\mathbb Q}(\chi,f)$. Note that $\chi$ is partial, and this is the only reason we allowed partial colorings in the definition of ${\mathbb Q}(\chi,f)$. 

Now suppose that $\prod_{l<m}\Q(\chi_l,f_l)$ has an uncountable antichain. Then we claim that there exists an $n<\om$, an uncountable subset $A$ of $X^n$ and a closed (in $X^n$) set $F\supseteq A$ so that 
\begin{enumerate}
\item the function $\langle x(0),\dots,x(n-1)\rangle\mapsto \langle f(x(0)),\dots,f(x(n-1))\rangle$ is constant on $A$, say with value $d\in 2^n$. Abusing notion we also denote this function by $f$;
\item no two tuples in $A$ have any elements in common; 
\item for every distinct $x,y\in F$, there exists some $i<n$ so that $\chi(x(i),y(i))$ is defined and $\chi(x(i),y(i))\neq d(i)$.
\end{enumerate}
To see that this is true, take an antichain of size $\aleph_1$ in the product $\prod_{l<m}\Q(\chi_l,f_l)$, and first thin it to assume that for each $l,k$ all conditions contribute the same number of elements to the $k$th homogeneous set for $\chi_l$. Now viewing the elements in the antichain as sequences arranging the members according to the coloring and homogeneous set they contribute to, call the resulting set $A$. Let $n$ be the length of each sequence in $A$. We further thin $A$ to secure (1). Next, thin $A$ to become a $\De$-system, and note that by taking $n$ to be minimal, we secure (2). Now observe that, for each $x\in A$, if $i<j<n$ and $x(i)$ and $x(j)$ are part of the same homogeneous set for the same coloring $\chi_l,$ say with color $c$, then as $\chi_l$ is an open coloring, there exists a pair of open sets $U_{i,j}\times V_{i,j}$ in $\tau_i\times \tau_j$ such that 
$$
\la x(i),x(j)\ra\in U_{i,j}\times V_{i,j}\seq\chi^{-1}_l(\lb c\rb).
$$ 
With this $x$ still fixed, by intersecting at most finitely-many open sets around each $x(i)$, we may remove the dependence on coordinates $j\neq i$, and thereby obtain for each $i$, an open set $U_i$ around $x(i)$ witnessing the values of $\chi$. In particular, for any $i<j<n$ such that $x(i)$ and $x(j)$ are in the same homogeneous set for the same coloring, say $\chi_l$, we have 
$$
\la x(i),x(j)\ra\in U_i\times U_j\seq\chi_l^{-1}(\lb c\rb),
$$
where $c=\chi_l(x(i),x(j))$. By using basic open sets, of which there are only countably-many, we may thin $A$ to assume that the sequence of open sets $\la U_i:i<n\ra$ is the same for all $x\in A$. As a result of fixing these open sets, and since $A$ is an antichain, we see that (3) holds for the elements of $A$. Since $\chi$ is an open coloring, (3) also hold for $F$, the closure of $A$ in $X^n$.

\begin{remark}\label{remark:blerg} The conditions in the previous paragraph are equivalent to the existence of $n<\om$, $d\in 2^n$, and a closed set $F\seq X^n$ so that (i) for any distinct $x,y\in F$, $\chi(x(i),y(i))$ is defined for all $i<n$, and for some $i<n$, $\chi(x(i),y(i))\neq d(i)$; and (ii) for every countable 
$z\subseteq X$, there exists 
$x\in F\cap (X\bsl z)^n$ so that $f\circ x=d$. Indeed, it is immediate that (1)-(3) give (i) and (ii), and for the other direction, iterate (ii) to obtain the uncountable set $A$.
\end{remark}

Any $F$ as in Remark \ref{remark:blerg} is a closed subset of a second countable space, and so $F$ is coded by a real. Thus if $\R$ is a partition product as in the statement of Proposition \ref{prop:name}, then any $\R$-name $\dot{F}$ for such a closed set will only involve conditions intersecting countably-many support coordinates, since $\R$, by Assumption 4.1, is c.c.c. This motivates the following definition and subsequent remark.

\begin{definition}\label{def:fg} A partition product $\R$ with domain $X$, say, based upon $\underline{\ps}\res\ka$ and $\dot{\underline{\Q}}\res\ka$ is said to be \emph{finitely generated} if there is a finite, $\ka$-suitable collection $\lb\la B_\iota,\psi_\iota\ra:\iota\in I\rb$, and a countable $Z\seq X$ disjoint from $\bigcup_{\iota\in I}B_\iota$, such that 
$$
X=Z\cup\bigcup_{\xi\in Z} b(\xi)\cup\bigcup_{\iota\in I} B_\iota.
$$
In this case, we will refer to $Z$ as the \emph{auxiliary part}.
\end{definition}

Note that in the above definition, it poses no loss of generality to assume that $Z$ is disjoint from $\bigcup_\iota B_\iota$, since the $B_\iota$ are base-closed.

\begin{remark}\label{remark:fg} Assuming the conclusion of Proposition \ref{prop:name} fails for a partition product ${\mathbb R}$, the discussion above produces the objects in Remark \ref{remark:blerg}. In fact it produces these objects for a finitely generated partition product which is a regular suborder of $\R$. To see this, simply take the suborder generated by the finitely-many $B_\iota$'s and a countable auxiliary part large enough to give the real coding the closed set $F$. 
\end{remark}

We further remark that Definition \ref{def:fg} refers implicitly to the following objects: index$_\de$, $\base_\de$, and $\vp_{\de,\mu}$ for $\de<\ka$, as well as $\ps_\ka$, index$_\ka$, $\base_\ka$, and $\vp_{\ka,\mu}$, which are needed in order to define a suitable collection.

We now define the notion of the ``height" of the finite suitable collection, a natural measure of how much the images of $\ps_\ka$ agree in the partition product. As mentioned in the paragraph after Definition \ref{def:suitable}, the height will be one component in our later inductive construction of preassignments (see Lemma \ref{lemma:main}).

\begin{definition}\label{def:height} Let $(\R,\cal{B})$ be a finitely-generated partition product with $\ka$-suitable collection $\mathcal{S}=\lb\la B_\iota,\psi_\iota\ra:\iota\in I\rb$. For each $\iota_0\neq\iota_1$ both in $I$, we define the \emph{height} of $\la B_{\iota_0},B_{\iota_1}\ra$, denoted $\htt(B_{\iota_0},B_{\iota_1})$, to be equal to the ordinal $\psi_{\iota_0}^{-1}[B_{\iota_0}\cap B_{\iota_1}]$; this also equals $\psi_{\iota_1}^{-1}[B_{\iota_0}\cap B_{\iota_1}]$ by Remark \ref{remark:ord}. We define the height of $\cal{S}$, denoted $\htt(\cal{S})$, to be the ordinal
$$
\max\left\{\htt(B_{\iota_0},B_{\iota_1}):\iota_0,\iota_1\in I\we\iota_0\neq\iota_1\right\}.
$$
\end{definition}

We close this subsection with the following straightforward lemma.

\begin{lemma}\label{lemma:graftfg} Let $(\ps,\mathcal{B})$, $(\R,\mathcal{D})$, and $\si$ be as in Lemma \ref{lemma:graft}. Suppose that both $(\ps,\mathcal{B})$ and $(\R,\mathcal{D})$ are finitely generated and that $(\R^*,\mathcal{D}^*)$ is the extension of $(\R,\mathcal{D})$ by grafting $(\ps,\mathcal{B})$ over $\si$. Then $(\R^*,\mathcal{D}^*)$ is also finitely generated.
\end{lemma}

\subsection{Counting Finitely Generated Partition Products}
\label{sec.counting}

The main reason we prefer to work with finitely generated partition products is that we are now able, using the machinery developed up to this point, to carry out counting arguments. More specifically, we can show that there are only $\aleph_1$-many isomorphism types of such partition products.

\begin{lemma}\label{lemma:M0} Let $M\prec H(\om_3)$ be countably closed with $\underline{\ps}\res(\ka+1)$, $\dot{\underline{\Q}}\res\ka\in M$. 
If $\R$ is a finitely generated partition product based upon $\underline{\ps}\res\ka$ and $\dot{\underline{\Q}}\res\ka$, then $\R$ is isomorphic to a partition product which has domain an ordinal $\rho$ below $M\cap\om_2$.
\end{lemma}
\begin{proof} Fix such an $M$, and let $\R$ be a finitely generated partition product based upon $\underline{\ps}\res\ka$ and $\dot{\underline{\Q}}\res\ka$, say with domain $X$. Let $\lb\la B_m,\psi_m\ra:m<n\rb$ be the $\ka$-suitable collection and $Z$ the auxiliary part, where we assume that $Z$ is disjoint from the union of the $B_m$. Let us enumerate $Z$ as $\la\xi_k:k<\om\ra$ and set $\de_k:=\ind(\xi_k)$ for each $k<\om$. Furthermore, we let $\pi_k$ be the rearrangement of $\ps_{\de_k}$ associated to $\base(\xi_k)$. 

We intend to apply Corollary \ref{cor:omegashuffle}, and so we define a sequence $\la\tau_m:m<\om\ra$ of rearrangements of $\R$ and base-closed subsets $\la D_m:m<\om\ra$ of $X$. For each $m<n$, set $\tau_m$ to be the rearrangement which first shifts the ordinals in $X\bsl B_m$ past $\sup(B_m)$ and then acts as $\psi_m^{-1}$ on $B_m$. For each $m\geq n$, say $m=n+k$, we set $\tau_m$ to be the rearrangement which first shifts the ordinals in $X\bsl (b(\xi_k)\cup\lb\xi_k\rb)$ past $\xi_k$ and then acts as $\pi_k^{-1}$ on $b(\xi_k)$ and sends $\xi_k$ to $\rho_{\de_k}$. We set $D_0:=\es$, $D_{m+1}:=\bigcup_{k\leq m} B_k$ for $m< n$, and $D_{n+1+k}:=D_n\cup\bigcup_{l\leq k}(b(\xi_l)\cup\lb\xi_l\rb)$ for $k<\om$.

By Corollary \ref{cor:omegashuffle}, let $\si$ be a rearrangement of $\R$ so that $\ran(\si)$ is an ordinal $\rho$ and so that for each $m<\om$, $\si[D_m]$ is an ordinal and $\tau_m\circ\si^{-1}$ is order-preserving on $\si[D_{m+1}\bsl D_m]$. We then see that $\rho$ equals $\sum_{m<\om}\text{ot}(\si[D_{m+1}\bsl D_m])$. However, if $m<n$, then $\text{ot}(\si[D_{m+1}\bsl D_m])$ is no larger than $\rho_\kappa$, and if $m\geq n$, then $\text{ot}(\si[D_{m+1}\bsl D_m])$ is no larger than $\rho_{\de_k}+1$, where $m=n+k$. Therefore
$$
\rho=\sum_{m<\om}\text{ot}(\si[D_{m+1}\bsl D_m])\leq\rho_\kappa\cdot n +\sum_{k<\om}(\rho_{\de_k}+1).
$$
By the elementarity and countable closure of $M$, the ordinal on the right hand side is an element of $M\cap\om_2$. Thus $\rho\in M\cap\om_2$ since $M\cap\om_2$ is an ordinal.  
\end{proof}

\begin{lemma}\label{lemma:M} Let $M\prec H(\om_3)$ be countably closed containing $\underline{\ps}\res(\ka+1)$, $\dot{\underline{\Q}}\res\ka$ and $\vec{\vp}$ as members. 
If $\R$ is a finitely generated partition product based upon $\underline{\ps}\res\ka$ and $\dot{\underline{\Q}}\res\ka$, then $\R$ is isomorphic to a partition product which belongs to $M$, as well as the transitive collapse of $M$.
\end{lemma}

Lemma \ref{lemma:M} is the {\em counting argument} that we mentioned earlier in the paper, and it relies on the Hull and Closure conditions (\ref{coherent-collapse-critical}) and (\ref{coherent-collapse-closure}) of Definition \ref{coherent-collapse}. These conditions come in through the use of Remark \ref{remark:definable} and Lemma \ref{lemma:sborder}.

\begin{proof} Let $M$ be fixed as in the statement of the lemma, and let $\R$ be finitely generated. Let $\lb\la B_k,\psi_k\ra:k<n\rb$ be the $\ka$-suitable collection and $Z$ the auxiliary part associated to $\R$. By Lemma \ref{lemma:M0}, we may assume that $\R$ is a partition product on some ordinal $\rho$ and that $\rho\in M\cap\om_2$. Since $M\cap\om_2$ is an ordinal, $\rho\seq M$. Then $Z\seq M$, and so by the countable closure of $M$, $Z$ is a member of $M$. Hence by the elementarity and countable closure of $M$, setting $\de_\xi:=\ind(\xi)$ for each $\xi\in Z$, the sequence $\la\de_\xi:\xi\in Z\ra$ is in $M$.

Now fix $k<n$ and $\xi\in Z$, and note that by Remark \ref{remark:definable}, since $\de_\xi$ and $\ka$ are in $M$, $\psi_k^{-1}[B_k\cap b(\xi)]$ is in $M$. Next consider the relation in $\mu,\nu$ which holds iff $\pi_\xi(\mu)=\psi_k(\nu)$, and observe that by Lemma \ref{lemma:sborder}, this holds iff $\nu$ is the $\mu$th element of $\psi_k^{-1}[B_k\cap b(\xi)]$. Therefore, this relation is a member of $M$. By the countable closure of $M$, the relation in $\xi,k,\mu,\nu$ which holds iff $\pi_\xi(\mu)=\psi_k(\nu)$ is in $M$ too. Similarly, the relation (in $\xi,\zeta,\mu,\nu$) which holds iff $\pi_\xi(\mu)=\pi_\zeta(\nu)$ and the relation (in $k,l,\mu,\nu$) which holds iff $\psi_k(\mu)=\psi_l(\nu)$ are also in $M$.

We now apply the elementarity of $M$ to find a finitely generated partition product $\R^*$ with domain $\rho$ which has the following properties, where $\base^*$ and $\ind^*$ denote the functions supporting $\R^*$:
\begin{enumerate}
\item $\R^*$ has $\ka$-suitable collection $\lb\la B^*_k,\psi^*_k\ra:k<n\rb$ and auxiliary part $Z$; moreover, for each $\xi\in Z$, $\ind^*(\xi)=\de_\xi$;
\item for each $\mu,\nu<\rho$ and each $\xi,\zeta\in Z$, $\pi_\xi(\mu)=\pi_\zeta(\nu)$ iff $\pi^*_\xi(\mu)=\pi^*_\zeta(\nu)$, and similarly with one of the $\psi_k$ (resp. $\psi^*_k$) replacing one or both of the $\pi_i$ (resp. $\pi^*_i)$.
\end{enumerate}
It is also straightforward to see that $\R^*$ is a member of the transitive collapse of $M$, as it is an iteration of length below $M\cap\om_2$ of posets of size $\leq\aleph_1$, and hence is not moved by the transitive collapse map.

We now define a bijection $\si:\rho\lra\rho$ which will be the rearrangement witnessing that $\R$ and $\R^*$ are isomorphic. Set $\si(\al)=\be$ iff $\al=\be$ are both in $Z$; or for some $\xi\in Z$, $\al=\pi_\xi(\mu)$ and $\be=\pi^*_\xi(\mu)$; or for some $k<n$, $\al=\psi_k(\mu)$ and $\be=\psi^*_k(\mu)$. By (2), we see that $\si$ is well-defined, i.e., there is no conflict when some of these conditions overlap. It is also straightforward to see that $\si$ is an acceptable rearrangement of $\R$ and in fact, $\si\left(\base\right)=\base^*$ and $\si\left(\ind\right)=\ind^*$, so that $\si$ is an isomorphism from $\R$ onto $\R^*$.
\end{proof}

Recall that we are assuming the $\ch$ holds (Assumption 4.1). Thus for the rest of Section 4, we fix a countably closed $M\prec H(\om_3)$ satisfying the conclusion of Lemma \ref{lemma:M} such that $|M|=\aleph_1$. We write $M=\bigcup_{\ga<\om_1}M_\ga$, for a continuous, increasing sequence of elementary, countable $M_\ga$, such that the relevant parameters are in $M_0$. 

\begin{remark}\label{remark:slice} The crucial use of the $\ch$ in the paper is to fix the model $M$. We will use the decomposition $M=\bigcup_{\ga<\om_1}M_\ga$ to partition a tail of $\om_1$ into the slices $[M_\ga\cap\om_1,M_{\ga+1}\cap\om_1)$. We will show that it suffices to define the preassignment one slice at a time, with the values of the preassignment on one slice independent of the others. As Lemma \ref{lemma:suffices} below shows, the preassignment restricted to the slice $[M_\ga\cap\om_1,M_{\ga+1}\cap\om_1)$ only needs to anticipate ``partition product names" which are members of $M_\ga$. This idea that the preassignment need only work in the above slices goes back to Lemma 3.2 of \cite{ARS}. Furthermore, the proof of our Lemma \ref{lemma:annoying} is more or less the same as Lemma 3.2 of  \cite{ARS}; we are simply working in slightly greater generality in order to analyze products of posets rather than just a single poset.
\end{remark}

\subsection{Further Reductions}

In this subsection we make our final reduction in preparation for the basic step construction in Section 5. We formulate a concrete condition on the preassignment name $\dot{f}$ which implies that $\dot{f}$ satisfies Proposition \ref{prop:name}.

We recall that $\dot{S}_\ka$ names a countable basis for a second countable, Hausdorff topology on $\omega_1$ and that $\dot{\chi}_\ka$ names a coloring from $[\om_1]^2$ into $2$ which is continuous with respect to the topology generated by $\dot{S}_\ka$.

\begin{lemma}\label{lemma:annoying} Suppose that $\dot{f}$ is a $\ps_\ka$-name for a function from $\om_1$ into $\lb 0,1\rb$ which satisfies the following: for any finitely generated partition product $\R$, with $\ka$-suitable collection $\lb\la B_\iota,\psi_\iota\ra:\iota\in I\rb$ and auxiliary part $Z$, say, all of which are in $M$; for every $\ga$ sufficiently large so that $\R$, the $\ka$-suitable collection, and $Z$ are in $M_\ga$; for any $\R$-name $\dot{F}$ in $M_\ga$ for a set of $n$-tuples in $X:=\biguplus_{\iota\in I}\om_1$, which is closed in $\left(\biguplus_\iota\dot{S}_\ka[\psi_\iota^{-1}(\dot{G}\res B_\iota)]\right)^n$; for any generic $G$ for $\R$; and for any $x$ with
$$
x\in\dot{F}[G]\cap(M_{\ga+1}[G]\bsl M_\ga[G])^n,
$$
there exist pairwise distinct tuples $y,y'$ in $\dot{F}[G]\cap M_\ga[G]$ so that for every $i<n$ and $\iota\in I$, if $x(i)$ is in the $\iota$-th copy of $\om_1$, then so are $y(i)$ and $y'(i)$, and 
$$
\dot{\chi}_\ka[\psi_\iota^{-1}(G\res B_\iota)](y(i),y'(i))=\dot{f}[\psi_\iota^{-1}(G\res B_\iota)](x(i)).
$$
Then $\dot{f}$ satisfies Proposition \ref{prop:name}.
\end{lemma}
\begin{proof} Let $\dot{f}$ be as in the statement of the lemma, and suppose that $\dot{f}$ failed to satisfy Proposition \ref{prop:name}. By Remarks \ref{remark:blerg} and \ref{remark:fg} there exist a finitely generated partition product $\R$, a condition $p\in\R$, an integer $n<\om$, a sequence $d\in 2^n$, and an $\R$-name for a closed set $\dot{F}$ of $n$-tuples such that $p$ forces that these objects satisfy Remark \ref{remark:blerg}.  We may assume that $\R\in M$ by Lemma \ref{lemma:M}. Since $M$ is countably closed and contains $\R$, and since $\R$ is c.c.c.\ (by Assumption 4.1), we know that the name $\dot{F}$ belongs to $M$ too. Thus we may find some $\ga<\om_1$ such that $\dot{F}$ and all other relevant objects are in $M_{\ga}$. 

Now let $G$ be a generic for $\R$ containing the condition $ p$. Let $S:=\biguplus_\iota\dot{S}_\ka[\psi_\iota^{-1}(G\res B_\iota)]$, let $f:=\biguplus_\iota\dot{f}[\psi_\iota^{-1}(G\res B_\iota)]$, and let $\chi:=\biguplus_\iota\dot{\chi}_\ka[\psi_\iota^{-1}(G\res B_\iota)]$. By (ii) of Remark \ref{remark:blerg}, we may find some $x\in F\cap (X\bsl M_{\ga}[G])^n$, so that $f\circ x=d$, where $F:=\dot{F}[G]$. We now want to consider how the models $\la M_\be:\ga\leq\be<\om_1\ra$ separate the elements of $x$, and then we will apply the assumptions of the lemma to each $\be\in[\ga,\om_1)$ such that $M_{\be+1}[G]\bsl M_\be[G]$ contains an element of $x$. Indeed, consider the finite set $a$ of $\be\in[\ga,\om_1)$ such that $x$ contains at least one element in $M_{\be+1}[G]\bsl M_\be[G]$, and let $\la\ga_k:k<l\ra$ be the increasing enumeration of $a$. Further, let $x_k$, for each $k<l$, be the 
restriction of $x$ to the set $e_k$ of $i<n$ so that $x(i)$ is inside $M_{\ga_k+1}[G]\bsl M_{\ga_k}[G]$. 

We now work downwards from $l$ to define a sequence of formulas $\la\vp_k:k\leq l\ra$. We will maintain as recursion hypotheses that  if 
$0\leq k <l$, then (i) $\vp_{k+1}(x_0,\dots,x_k)$ is satisfied, and that (ii) the parameters of $\vp_{k+1}$ are in $M_{\ga_0}[G]$. 
Recall $e_k=\dom(x_k)$.
Let $\vp_l(u_0,\dots,u_{l-1})$ state that 
$u_0\cup\dots\cup u_{l-1}\in F\land \bigwedge_{k<l}\dom(u_k)=e_k$; then (i) and (ii) are satisfied. Now suppose that 
$0\leq k <l$ and that $\vp_{k+1}$ is defined. Let $F_k$ be the closure of the set of all tuples $u$ such that $\vp_{k+1}(x_0,\dots,x_{k-1},u)$ is satisfied. By (ii) and the fact that 
$x_0\cup\dots\cup x_{k-1}\in M_{\ga_k}[G]$, we see that $F_k$ is in $M_{\ga_k}[G]$. Furthermore, $x_k\in F_k$. Therefore, by the assumptions of the lemma, we may find pairwise distinct tuples $v_{k,L},v_{k,R}$ in $M_{\ga_k}[G]\cap F_k$, with the same domain as $x_k$, such that for every 
$i\in\dom(x_k)$ and $\iota\in I$, if $x_k(i)$ is in the $\iota$-th copy of $\om_1$, then so are $v_{k,L}(i)$ and $v_{k,R}(i)$, and 
$$
\dot{\chi}_\ka[\psi_\iota^{-1}(G\res B_\iota)](v_{k,L}(i),v_{k,R}(i))=\dot{f}[\psi_\iota^{-1}(G\res B_\iota)](x_k(i)).
$$ 
For each such $i$, fix a pair of disjoint, basic open sets $U_i,V_i$ from $\dot{S}_\ka[\psi_\iota^{-1}(G\res B_\iota)]$ witnessing this coloring statement. By definition of $F_k$, we may find two further tuples $u_{k,L},u_{k,R}$ such that for each $Z\in\lb L,R\rb$, $\vp_{k+1}(x_0,\dots,x_{k-1},u_{k,Z})$ is satisfied, and such that the pair $\la u_{k,L}(i),u_{k,R}(i)\ra$ is in $U_i\times V_i$. Now define $\vp_k(u_0,\dots,u_{k-1})$ to be the following formula:
$$
\exists w_{k,L},w_{k,R}\;\left(\bigwedge_{Z\in\lb L,R\rb} \vp_{k+1}(u_0,\dots,u_{k-1},w_{k,Z})\we \bigwedge_i\Big( \la w_{k,L}(i),w_{k,R}(i)\ra\in U_i\times V_i \Big)\right).
$$
Then (i) is satisfied, and since the only additional parameters are the basic open sets $U_i$ and $V_i$, (ii) is also satisfied.

This completes the construction of the sequence $\la\vp_k:k\leq l\ra$. Now using the fact that the sentence $\vp_0$ is true and has only parameters in $M_{\ga_0}$, we may work our way upwards through the sequence $\vp_0,\vp_1,\dots,\vp_l$ in order to find two tuples $x_L,x_R$ of the same length as $x$ such that $x_L,x_R\in F$, and such that for each $i<n$, $\la x_L(i),x_R(i)\ra\in U_i\times V_i$. In particular, for each $i<n$,
$$
\dot{\chi}_\ka[\psi_\iota^{-1}(G\res B_\iota)](x_L(i),x_R(i))=\dot{f}[\psi_\iota^{-1}(G\res B_\iota)](x(i)),
$$
where $\iota$ is such that $x(i)$ is in the $\iota$-th copy of $\om_1$. However, recalling Remark \ref{remark:blerg} and the assumptions about the condition $p$, this contradicts the fact that $f\circ x=d$, and that there is some $i<n$ so that $\chi(x_L(i),x_R(i))\neq d(i)$.
\end{proof}

The following lemma gives a nice streamlining of the previous one and 
uses any collection $\dot{U}$ of $n$-tuples in $\om_1$ in the hypothesis, not just collections $\dot{F}$ which are closed in the appropriate topology. The greater generality in the hypothesis here is only apparent, as we can always take closures and obtain, because the colorings are open, the 
general hypothesis from its restriction to closed sets. However, it is technically convenient 
to forget the topology in stating the lemma. Also, as a matter of notation, for each $\ga<\om_1$, we fix an enumeration $\la\nu_{\ga,n}:n<\om\ra$ of the slice $[M_\ga\cap\om_1,M_{\ga+1}\cap\om_1)$.

\begin{lemma}\label{lemma:suffices} Suppose that $\dot{f}$ is a $\ps_\ka$-name for a function from $\om_1$ into $\lb 0,1\rb$ satisfying the following: for any finitely generated partition product $\R$, say with $\ka$-suitable collection $\lb\la B_\iota,\psi_\iota\ra:\iota\in I\rb$ and auxiliary part $Z$, all of which are in $M$; for any $\ga$ sufficiently large such that $M_\ga$ contains $\R$, $\lb\la B_\iota,\psi_\iota\ra:\iota\in I\rb$, and $Z$; for any $l<\om$; for any $\R$-name $\dot{U}$ in $M_\ga$ for a set of $l$-tuples in $\om_1$; and for any generic $G$ for $\R$,  if $\la\nu_{\ga,0},\dots,\nu_{\ga, l-1}\ra\in\dot{U}[G]$, then there exist pairwise distinct $l$-tuples $\vec{\mu},\vec{\mu}'$ in $M_\ga[G]\cap\dot{U}[G]$  so that for all $k<l$ and all $\iota\in I$, 
$$
\dot{\chi}_\ka[\psi_\iota^{-1}(G\res B_\iota)](\mu_k,\mu'_k)=\dot{f}[\psi_\iota^{-1}(G\res B_\iota)](\nu_{\ga,k}).
$$
Then $\dot{f}$ satisfies the assumptions of Lemma \ref{lemma:annoying} and hence the conclusion of Proposition \ref{prop:name}.
\end{lemma}
\begin{proof} We want to first observe that Lemma \ref{lemma:annoying} follows from its restriction to sequences $z$ which are bijections from some $n<\om$ onto $\biguplus_\iota\lb\nu_{\ga,l}:l<m\rb$, for some $m<\om$. Towards this end, fix $\dot{F}$, $G$, and a tuple $x\in\dot{F}[G]$ as in the statement of Lemma \ref{lemma:annoying}. First, if $x$ is not such a surjection, we may add additional coordinates to $x$ to form a sequence $x'$ which is a surjection onto $\biguplus_\iota\lb\nu_{\ga,l}:l<m\rb$, for some $m<\om$. Then we define the name $\dot{F}'$ as the product of $\dot{F}$ with the requisite, finite number of copies of $\om_1$, so that $x'$ is a member of $\dot{F}'[G]$. Second, if $x'$ contains repetitions, then we make the necessary shifts in $x'$ to eliminate the repetitions and call the resulting sequence $x''$. We then consider the name $\dot{F}''$ of all tuples from $\dot{F}'$ which have the same corresponding shifts in their tuples as $x''$. $\dot{F}''$ still names a closed set and is still an element of $M_\ga$. Thus $x''\in\dot{F}''[G]$, and $x''$ is a bijection from some integer onto $\biguplus_\iota\lb\nu_{\ga,l}:l<m\rb$, for some $m<\om$. By applying the restricted version of Lemma \ref{lemma:annoying} to $x''$ and $\dot{F}''$, we see that the desired result holds for $x$ and $\dot{F}$.

To verify Lemma \ref{lemma:annoying}, fix $\dot{F}$, a generic $G$, and a sequence $x\in\dot{F}[G]$ as in the statement thereof, where we assume that $x$ is a bijection from some $n$ onto $\biguplus_\iota\lb\nu_{\ga,l}:l<m\rb$, for some $m<\om$. Define $\dot{U}$ to be the $\R$-name for the set of all tuples $\vec{\xi}=\la\xi_0,\dots,\xi_{m-1}\ra$ in $\om_1$ such that $\vec{\xi}$ concatenated with itself $|I|$-many times is an element of $\dot{F}$, noting that $\dot{U}$ is still a member of $M_\ga$. Since $x$ is a bijection as described above, $\la\nu_{\ga,0},\dots,\nu_{\ga,m-1}\ra\in\dot{U}[G]$. Now apply the assumptions in the statement of the current lemma to find two pairwise distinct $m$-tuples $\vec{\mu},\vec{\mu}'$ in $M_\ga[G]\cap\dot{U}[G]$ so that for all $l<m$ and $\iota\in I$, 
$$
\dot{\chi}_\ka[\psi_\iota^{-1}(G\res B_\iota)](\mu_l,\mu'_l)=\dot{f}[\psi_\iota^{-1}(G\res B_\iota)](\nu_{\ga,l}).
$$ 
Let $y$ be the $|I|$-fold concatenation of $\vec{\mu}$ with itself, and let $y'$ be defined similarly with respect to $\vec{\mu}'$. Then as $\vec{\mu},\vec{\mu}'\in\dot{U}[G]$, we have $y,y'$ are in $\dot{F}[G]$. And since $\vec{\mu},\vec{\mu}'$ satisfy the appropriate coloring requirements, we have that $y$ and $y'$ satisfy the conclusion of Lemma \ref{lemma:annoying}.
\end{proof}

Lemma \ref{lemma:suffices} gives a sufficient condition for Proposition \ref{prop:name}, and it thus implies that any partition product based upon $\underline{\ps}\res(\ka+1)$ and $\dot{\underline{\Q}}\res(\ka+1)$ is c.c.c. In the next section, we consider how to obtain a $\ps_\ka$-name $\dot{f}$ as in Lemma \ref{lemma:suffices}. 

\section{Constructing Preassignments}

\label{sec.preassign}

In this section, which forms the technical heart of the paper, we show how to obtain a $\ps_\ka$-name $\dot{f}$ satisfying the assumptions of Lemma \ref{lemma:suffices}. In light of Remark \ref{remark:slice} and Lemma \ref{lemma:suffices}, it suffices to define the name $\dot{f}$ separately for each of its restrictions to the slices $[M_\ga\cap\om_1,M_{\ga+1}\cap\om_1)$, and so let $\ga<\om_1$ be fixed for the remainder of this section. To simplify notation, we drop the $\ga$-subscript from the enumeration $\la\nu_{\ga,n}:n<\om\ra$ of $[M_\ga\cap\om_1,M_{\ga+1}\cap\om_1)$, preferring instead to simply write $\la\nu_n:n<\om\ra$. We note that the values of $\dot{f}$ on the countable ordinal $M_0\cap\om_1$ are irrelevant, by Remark \ref{remark:blerg}.

\subsection{Canonical Color Names and the Partition Product Preassignment Property}

In order to define the name $\dot{f}$, we recursively specify the $\ps_\ka$-name equal to $\dot{f}(\nu_k)$, which we call $\dot{a}_k$. Each $\dot{a}_k$ will be a canonical name, which we view as a function from a maximal antichain in $\ps_\ka$ into $\lb 0,1\rb$. We refer to these more specifically as \emph{canonical color names}. By a \emph{partial canonical color name} we mean a function from an antichain in $\ps_\ka$, possibly not maximal, into $\lb 0,1\rb$. When viewing such functions as names $\dot{a}$, we say that $\dot{a}[G]$, where $G$ is generic for $\ps_\ka$, is defined and equal to $i$ if there is some $p\in G$ which belongs to the domain of the function $\dot{a}$ and gets mapped to $i$. The upcoming definition isolates exactly what we need.

\begin{definition}\label{def:p4} Suppose that $\dot{a}_0,\dots,\dot{a}_{l-1}$ are partial canonical color names. We say that they have the \emph{partition product preassignment property at} $\ga$ if for every finitely generated partition product $\R$ with $\ka$-suitable collection $\lb\la B_\iota,\psi_\iota\ra:\iota\in I\rb$, say, all of which are in $M_\ga$; for every $\R$-name $\dot{U}\in M_\ga$ for a collection of $l$-tuples in $\om_1$; and for every generic $G$ for $\R$, the following holds: if $\la\nu_0,\dots,\nu_{l-1}\ra\in\dot{U}[G]$, then there exist two pairwise distinct tuples $\vec{\mu},\vec{\mu}'\in\dot{U}[G]\cap M_\ga[G]$ so that for every $\iota\in I$ and $k<l$, if $\dot{a}_k[\psi_\iota^{-1}(G\res B_\iota)]$ is defined, then 
$$
\dot{\chi}_\ka[\psi_\iota^{-1}(G\res B_\iota)](\mu_k,\mu'_k)=\dot{a}_k[\psi_\iota^{-1}(G\res B_\iota)].
$$
\end{definition}

\begin{definition} In the context of Definition \ref{def:p4}, we say that two sequences $\vec{\mu}$ and $\vec{\mu}'$ of length $l$ \emph{match} $\dot{a}_0,\dots,\dot{a}_{l-1}$ \emph{at }$\iota$\emph{ with respect to }$G$, or \emph{match at $B_\iota$ with respect to }$G$ if for every $k<l$ such that $\dot{a}_k[ \psi_\iota^{-1}(G\res B_\iota)]$ is defined,
$$
\dot{\chi}_\ka[ \psi_\iota^{-1}(G\res B_\iota)](\mu_k,\mu'_k)=\dot{a}_k[ \psi_\iota^{-1}(G\res B_\iota)].
$$
We say that two sequences $\vec{\mu}$ and $\vec{\mu}'$ \emph{match} $\dot{a}_0,\dots,\dot{a}_{l-1}$ \emph{on} $I$ \emph{with respect to }$G$ if for every $\iota\in I$, $\vec{\mu}$ and $\vec{\mu}'$ match $\dot{a}_0,\dots,\dot{a}_{l-1}$ at $\iota$ with respect to $G$. If the filter $G$ is clear from context, we drop the phrase ``with respect to $G$." Furthermore, we will often want to avoid talking about the index set $I$ explicitly, and so we will also say that $\vec{\mu},\vec{\mu}'$ match $\dot{a}_0,\dots,\dot{a}_{l-1}$ on $\mathcal{S}:=\lb\la B_\iota,\psi_\iota\ra:\iota\in I\rb$, if for each $\la B,\psi\ra\in\mathcal{S}$, we have that $\vec{\mu},\vec{\mu}'$ match $\dot{a}_0,\dots,\dot{a}_{l-1}$ at $B$.
\end{definition}

To show that there exists a name $\dot{f}$ satisfying the assumptions of Lemma \ref{lemma:suffices}, and which thereby satisfies Proposition \ref{prop:name}, we recursively construct the sequence $\la\dot{a}_k:k<\om\ra$ in such a way that for each $l<\om$, $\dot{a}_0,\dots,\dot{a}_{l-1}$ have the partition product preassignment property at $\ga$. More precisely, we show that if $\dot{a}_0,\dots,\dot{a}_{l-1}$ are \emph{total} canonical color names with the partition product preassignment property at $\ga$,  then there is a total name $\dot{a}_l$ so that $\dot{a}_0,\dots,\dot{a}_l$ have the partition product preassignment property at $\ga$.

For this in turn it is enough to prove that if $\dot{a}_0,\dots,\dot{a}_{l-1}$ are total canonical color names, $\dot{a}_l$ is a partial canonical color name, $\dot{a}_0,\dots,\dot{a}_l$ have the partition product preassignment property at $\ga$, and $p\in\ps_\ka$ is incompatible with all conditions in the domain of $\dot{a}_l$, then there exist $p^*\leq_{\ps_\ka} p$ and $c\in\lb 0,1\rb$ so that $\dot{a}_0,\dots,\dot{a}_l\cup\lb p^*\mapsto c\rb$ have the partition product preassignment property at $\ga$. By a transfinite iteration of this process we can construct a sequence of names $\dot{a}_l^\xi$ with increasing domains, continuing until we reach a name whose domain is a maximal antichain. This final name is then total.

To prove the ``one condition" extension above, we assume that it fails with $c=0$ and prove that it then holds with $c=1$. Our assumption is the following:

\begin{assumption}\label{ass2} $\dot{a}_0,\dots,\dot{a}_{l-1}$ are total canonical color names, $\dot{a}_l$ is partial, $\dot{a}_0,\dots,\dot{a}_l$ have the partition product preassignment property at $\ga$, $p\in\ps_\ka$ is incompatible with all conditions in $\dom(\dot{a}_l)$, but for every $p^*\leq_{\ps_\ka}p$, $\dot{a}_0,\dots,\dot{a}_l\cup\lb p^*\mapsto 0\rb$ do \emph{not} have the partition product preassignment property at $\ga$.
\end{assumption}

Our goal is to show that $\dot{a}_0,\dots,\dot{a}_l\cup\lb p\mapsto 1\rb$ do have the partition product preassignment property at $\ga$.  The following lemma is the key technical result which allows us to prove that $p\mapsto 1$ works in this sense and thereby continue the construction of the name $\dot{a}_l$. We note that the lemma is stated in terms of enriched partition products; the enrichments are used to propagate the induction hypothesis needed for its proof. After the statement of the lemma, and before its proof, we make a few remarks about the structure of the proof.

\begin{lemma}\label{lemma:main} Let $(\R,\mathcal{B})$ be an enriched partition product with domain $X$ which is finitely generated by a $\ka$-suitable collection $\mathcal{S}=\lb\la  B_\iota,\psi_\iota\ra:\iota\in I\rb$ and auxiliary part $ Z$, all of which belong to $M_\ga$. Let $\bar{p}$ be a condition in $\R$, and let $\vec{\nu}:=\la\nu_0,\dots,\nu_l\ra$. Finally, let $\bar{\mathcal{S}}\seq\mathcal{S}$ be non-empty. Then there exist the following objects: 
\begin{enumerate}
\item[(a)] an enriched partition product $(\R^*,\mathcal{B}^*)$ with domain $X^*$, finitely generated by a $\ka$-suitable collection $\mathcal{S}^*$ and an auxiliary part $ Z^*$, all of which are in $M_\ga$;
\item[(b)] a condition $ p^*\in\R^*$;
\item[(c)] an $\R^*$-name $\dot{U}^*$ in $M_\ga$ for a collection of $l+1$-tuples in $\om_1$;
\item[(d)] a non-empty, finite collection $\mathcal{F}$ in $M_\ga$ of embeddings from $(\R,\mathcal{B})$ into $(\R^*,\mathcal{B}^*)$;
\end{enumerate}
satisfying that for each $\pi\in\mathcal{F}:$
\begin{enumerate}
\item $ p^*\leq_{\R^*}\pi(\bar{p})$;
\end{enumerate}
and also satisfying that $ p^*$ forces the following statements in $\R^*$:

\begin{enumerate}
\item[(2)] $\vec{\nu}\in\dot{U}^*$;
\item[(3)] for any pairwise distinct tuples $\vec{\mu},\vec{\mu}'$ in $\dot{U}^*\cap M_\ga[\dot{G}^*]$, if $\vec{\mu},\vec{\mu}'$ match $\dot{a}_0,\dots,\dot{a}_l$ on $\mathcal{S}^*$, then there is some $\pi\in\mathcal{F}$ such that $\vec{\mu},\vec{\mu}'$ match $\dot{a}_0,\dots,\dot{a}_l\cup\lb p\mapsto 1\rb$ on $\pi\left(\bar{\mathcal{S}}\right)$.
\end{enumerate}
\end{lemma}

\begin{remark} In the proof of Lemma \ref{lemma:main}, we will proceed by induction, first on the height (see Definition \ref{def:height}) of the suitable subcollection $\bar{\cal{S}}$ and second on the finite size of $\bar{\cal{S}}$.

Proving the lemma requires that we embed the initial partition product into a much larger one in a variety of ways in order to see that the starting condition in $\R$ does not force the negation of the desired conclusion.  This larger partition product will be created through appeals to induction, quite a bit of grafting, and, at the base, a use of ``color 0 counterexamples''; these are partition products witnessing, for many $p^*\leq_{\ps_\ka}p$, that $\dot{a}_0,\dots,\dot{a}_l\cup\lb p^*\mapsto 0\rb$ do not have the partition product preassignment property at $\ga$.

The use of color 0 counterexamples is most clearly seen in the base case of the induction, namely, where the height is 0. In this case the heart of ${\mathbb R}^*$ is essentially a product of color 0 counterexamples.  The inductive case then combines products of this kind in more elaborate ways. On a first reading, it may be helpful to think of its simplest instance, namely when  $\htt(\bar{\cal{S}})=1$ and when $\bar{\cal{S}}$ has exactly two elements.
\end{remark}

\begin{proof} For the remainder of the proof, fix the objects $(\R,\mathcal{B})$, $X$, $\mathcal{S}$, $Z$,  $\bar{p}$, and $\bar{\mathcal{S}}$ as in the statement of the lemma. We also set $J:=\lb\iota\in I:\la B_\iota,\psi_\iota\ra\in\bar{\mathcal{S}}\rb$. Before we continue, let us introduce the following ad hoc terminology: suppose that $p'\leq_{\ps_\ka}p$, $c\in\lb 0,1\rb$, and $\tilde{p}\in\ps_\ka$. We say that $\tilde{p}$ is \emph{decisive} about the sequence of names $\dot{a}_0,\dots,\dot{a}_l\cup\lb p'\mapsto c\rb$ if for each $k<l$, $\tilde{p}$ extends a unique element of $\dom(\dot{a}_k)$, and if $\tilde{p}$ either extends a unique element of $\dom(\dot{a}_l)\cup\lb p'\rb$ or is incompatible with all conditions therein. Note that any $\tilde{p}$ may be extended to a decisive condition, as $\dom(\dot{a}_k)$ is a maximal antichain in $\ps_\ka$, for each $k<l$.

For each $\iota\in I$ we set $p_\iota$ to be the condition $\psi_{\iota}^{-1}(\bar{p}\res B_\iota)$ in $\ps_\ka$. By extending $\bar{p}$ if necessary, we may assume that for each $\iota\in I$, $p_\iota$ is decisive about $\dot{a}_0,\dots,\dot{a}_l\cup\lb p\mapsto 1\rb$. Let us also define
$$
J_p:=\lb\iota\in J:p_\iota\leq_{\ps_\ka}p\rb,
$$
noting that for each $\iota\in J\bsl J_p$, $p_\iota$ is incompatible with $p$ in $\ps_\ka$, since $p_\iota$ is decisive.

 We will prove by induction that there exist objects as in (a)-(d) satisfying (1)-(3). The induction concerns properties of $\bar{\mathcal{S}}$, which we will refer to as the \emph{matching core} of $\mathcal{S}$, in light of the requirement in (3) that the desired matching occurs on the image of $\bar{\mathcal{S}}$ under some $\pi\in\mathcal{F}$. The induction will be first on the height of $\bar{\cal{S}}$, as defined in Definition \ref{def:height}, and then on the finite size of $\bar{\cal{S}}$. It is helpful to note here that if $B_{\iota_0}\neq B_{\iota_1}$, then the ordinal $\htt(B_{\iota_0},B_{\iota_1})$ is strictly below $\rho_\ka$ and furthermore that 
 $
 \htt(B_{\iota_0},B_{\iota_1})=\max\lb\al<\rho_\ka:\psi_{\iota_0}[\al]=\psi_{\iota_1}[\al]\rb=\sup\lb\xi+1:\psi_{\iota_0}(\xi)=\psi_{\iota_1}(\xi)\rb.
 $ \\

\textbf{Case 1:} $\text{ht}\left(\bar{\mathcal{S}}\right)=0$ (note that this includes as a subcase $|J|=1$). For each $\iota\in J_p$, $p_\iota$ extends $p$ in $\ps_\ka$, and so, by Assumption \ref{ass2}, $\dot{a}_0,\dots,\dot{a}_l\cup\lb p_\iota\mapsto 0\rb$ do not have the partition product preassignment property at $\ga$. For each $\iota\in J_p$, we fix the following objects as witnesses to this:
\begin{enumerate}
\item[$(1)_\iota$] a partition product $\R_\iota^*$, say with domain $X^*_\iota$, which is finitely generated by the $\ka$-suitable collection $\mathcal{S}^*_\iota=\lb\la B^*_{\iota,\eta},\psi^*_{\iota,\eta}\ra: \eta\in I^*(\iota)\rb$ and auxiliary part $Z^*_\iota$, all of which are in $M_\ga$;
\item[$(2)_\iota$] a condition $ p^*_\iota$ in $\R^*_\iota$; 
\item[$(3)_\iota$] an $\R^*_\iota$-name $\dot{U}^*_\iota$ in $M_\ga$ for a set of $l+1$-tuples in $\om_1$;
\end{enumerate}
such that $p^*_\iota$ forces in $\R_\iota^*$ that 
\begin{enumerate}
\item[$(4)_\iota$] $\vec{\nu}\in\dot{U}^*_\iota$, and for any pairwise distinct tuples $\vec{\mu},\vec{\mu}'$ in $\dot{U}^*_\iota\cap M_\ga[\dot{G}^*_\iota]$, $\vec{\mu}$ and $\vec{\mu}'$ do \emph{not} match $\dot{a}_0,\dots,\dot{a}_l\cup\lb p_\iota\mapsto 0\rb$ on $I^*(\iota)$.
\end{enumerate}
For each $\eta\in I^*(\iota)$, let $p_{\iota,\eta}$ denote the $\ps_\ka$-condition $(\psi^*_{\iota,\eta})^{-1}\left(p^*_\iota\res B^*_{\iota,\eta}\right)$, and note that by extending the condition $p^*_\iota$, we may assume that each $p_{\iota,\eta}$ is decisive about $\dot{a}_0,\dots,\dot{a}_l\cup\lb p_\iota\mapsto 0\rb$. It is straightforward to check that since each such $p_{\iota,\eta}$ is decisive and since, by Assumption \ref{ass2}, $\dot{a}_0,\dots,\dot{a}_l$ do have the partition product preassignment property at $\ga$, we must have that
$$
J^*(\iota):=\lb\eta\in I^*(\iota):p_{\iota,\eta}\leq_{\ps_\ka}p_\iota\rb\neq\es,
$$
as otherwise we contradict $(4)_\iota$. 

Let us introduce some further notation which will facilitate the exposition. For $\iota\in J\bsl J_p$, define $\R^*_\iota$ to be some isomorphic copy of $\ps_\ka$ with domain $X^*_\iota$, say with isomorphism $\psi^*_{\iota,\iota}$; we will denote $X^*_\iota$ additionally by $B^*_{\iota,\iota}$ in order to streamline the notation in later arguments.  For $\iota\in J\bsl J_p$, we set $\mathcal{S}^*_\iota:=\lb\la B^*_{ \iota,\iota },\psi^*_{ \iota,\iota }\ra\rb$ with index set $I^*(\iota)=\lb\iota\rb$ which we also denote by $J^*(\iota)$. Next, we define $p^*_\iota$ to be the image of $p_\iota$ under the isomorphism $\psi^*_{\iota,\iota}$ from $\ps_\ka$ onto $\R^*_\iota$, and we set $\dot{U}^*_\iota$ to be the $\R^*_\iota$-name for all $l+1$-tuples in $\om_1$. We remark here for later use that for each $\iota\in J$ and $\eta\in J^*(\iota)$, 
$$
(\psi^*_{\iota,\eta})^{-1}\left(p^*_\iota\res B^*_{\iota,\eta}\right)\leq_{\ps_\ka} (\psi_\iota)^{-1}(\bar{p}\res B_\iota).
$$ 

Our next step is to amalgamate all of the above into one much larger partition product. Without loss of generality, by shifting if necessary, we may assume that the domains $X^*_\iota$, for $\iota\in J$, are pairwise disjoint. Then, by Corollary \ref{productgraft}, the poset $\R^*(0):=\prod_{\iota\in J}\R^*_\iota$ is a partition product with domain $\bigcup_{\iota\in J}X^*_\iota$. It is also a member of $M_\ga$. Additionally, $\R^*(0)$ is finitely generated by the $\ka$-suitable collection $\mathcal{S}^*:=\bigcup_{\iota\in J}\mathcal{S}^*_\iota$ and auxiliary part $\bigcup_{\iota\in J_p} Z^*_\iota$.   Let us abbreviate $\bigcup_{\iota\in J}B_\iota$ by $X_0$ and $\bigcup_{\iota\in J}X^*_\iota$ by $X^*_0$. We also let $p^*(0)$ be the condition in $\R^*(0)$ whose restriction to $X^*_\iota$ equals $p^*_\iota$, and we let $\dot{U}^*$ be the $\R^*(0)$-name for the intersection of all the $\dot{U}^*_\iota$, for $\iota\in J$.

Now consider the product of indices 
$$
\hat{J}:=\prod_{\iota\in J}J^*(\iota);
$$ 
$\hat{J}$ is non-empty, finite, and an element of $M_\ga$, since $J$ and each $J^*(\iota)$ are. Let $\la h_k:k<n\ra$ enumerate $\hat{J}$. Each $h_k$ selects, for every $\iota\in J$, an image of the $\ps_\ka$-``branch" $B_\iota$ inside $\R^*_\iota$. For each $k<n$, we define the map $\pi_k:X_0\lra X^*_0$ corresponding to $h_k$ by taking $\pi_k\res B_\iota$ to be equal to $\psi^*_{\iota,h_k(\iota)}\circ\psi^{-1}_\iota$, for each $\iota\in J$. This is well-defined since, by our assumption that $\text{ht}\left(\bar{\mathcal{S}}\right)=0$, we know that the sets $B_\iota$, for $\iota\in J$, are pairwise disjoint. We also see that each $\pi_k$ embeds $\R\res X_0$ into $\R^*(0)$, since it isomorphs $\R\res B_\iota$ onto $\R^*(0)\res B_{\iota,h_k(\iota)}^*$, for each $\iota\in J$. In fact, each $\pi_k$ is $(\bar{\mathcal{S}},\mathcal{S}^*)$-suitable by construction, and $h_k$ is the associated 
map $h_{\pi_k}$ (see Definition \ref{def:suitablemap}). Finally, we want to see that $p^*(0)$ extends $\pi_k(\bar{p}\res X_0)$ for each $k<n;$ but this follows by definition of $\pi_k$ and our above observation that for each $\iota\in J$ and $\eta\in J^*(\iota)$,
$$
(\psi^*_{\iota,\eta})^{-1}\left(p^*_\iota\res B^*_{\iota,\eta}\right)\leq_{\ps_\ka} (\psi_\iota)^{-1}(\bar{p}\res B_\iota).
$$ 

Using Lemma \ref{lemma:productcase}, fix an enrichment $\mathcal{B}^*_0$ of $\R^*(0)$ such that $\mathcal{B}^*_0$ contains the image of $\mathcal{B}\res X_0$ under each $\pi_k$ and such that $\lb\la B^*_{\iota,\eta},\psi^*_{\iota,\eta}\ra:\iota\in J\we\eta\in J^*(\iota)\rb$ is $\ka$-suitable with respect to $(\R^*(0),\mathcal{B}^*_0)$. Note that the assumptions of Lemma 
\ref{lemma:productcase}
are satisfied because the sets $X^*_\iota$, for $\iota\in J$, are pairwise disjoint and $\lb\pi_k:k<n\rb$ is a collection of $(\bar{\mathcal{S}},\mathcal{S}^*)$-suitable maps.

We note that this part of the construction makes use of the Collapse condition (\ref{coherent-collapse-comp}) of Definition \ref{coherent-collapse}, through the appeal to Lemma \ref{lemma:productcase}.

Before continuing with the main argument, we want to consider an ``illustrative case" in which we make the simplifying assumption that the domain of $\R$ is just $X_0$. The key ideas of the matching argument are present in this illustrative case, and after working through the details, we will show how to extend the argument to work in the more general setting wherein the domain of $\R$ has elements beyond $X_0$.

Proceeding, then, under the assumption that the domain of $\R_0$ is exactly $X_0$, we specify the objects from (a)-(d) satisfying (1)-(3). Namely, the finitely generated partition product $(\R^*(0),\mathcal{B}^*_0)$, generated by $\mathcal{S}^*$ and $\bigcup_{\iota\in J_p}Z^*_\iota$; the condition $p^*(0)$; the $\R^*(0)$-name $\dot{U}^*$; and the collection $\lb\pi_k:k<n\rb$ of embeddings are the requisite objects. From the fact that $\bar{p}=\bar{p}\res X_0$ we have that $p^*(0)$ is below $\pi_k(\bar{p})$ for each $k<n$. Since $p^*_\iota$ forces that $\vec{\nu}\in\dot{U}^*_\iota$ for each $\iota\in J$, we see that $p^*(0)$ forces that $\vec{\nu}\in\dot{U}^*$. Thus (3) remains to be checked.

Towards this end, fix a generic $G^*$ for $\R^*(0)$ containing $p^*(0)$, and for each $\iota\in J$, set $G^*_\iota:=G^*\res\R^*_\iota$. Also set $U^*:=\dot{U}^*[G^*]$. Let us also fix two pairwise distinct tuples $\vec{\mu}$ and $\vec{\mu}'$ in $U^*\cap M_\ga[G^*]$ which match $\dot{a}_0,\dots,\dot{a}_l$ on $\mathcal{S}^*$. Our goal is to find some $k<n$ such that $\vec{\mu}$ and $\vec{\mu}'$ match $\dot{a}_0,\dots,\dot{a}_l\cup\lb p\mapsto 1\rb$ on $\pi_k\left(\bar{\mathcal{S}}\right)$. We will first show the following claim.\\

\begin{claim}\label{claim:woah} For each $\iota\in J_p$, there is some $\eta\in J^*(\iota)$ such that 
$$
\dot{\chi}_\ka\left[\left(\psi^*_{\iota,\eta}\right)^{-1}\left(G^*_\iota\res B^*_{\iota,\eta}\right)\right](\mu_l,\mu'_l)=1.
$$
\end{claim}
\begin{proof}[Proof of Claim \ref{claim:woah}] Recall that for each $\iota\in J_p$, by $(4)_\iota$ above, we know that the condition $ p^*_\iota$ forces in $\R^*_\iota$ that for any two pairwise distinct tuples $\vec{\xi},\vec{\xi}'$ in $\dot{U}^*_\iota\cap M_\ga[\dot{G}^*_\iota]$, $\vec{\xi}$ and $\vec{\xi}'$ do \emph{not} match $\dot{a}_0,\dots,\dot{a}_l\cup\lb p_\iota\mapsto 0\rb$ on $I^*(\iota)$. Fix some $\iota\in J_p$, and let $U^*_\iota:=\dot{U}^*_\iota[G^*_\iota]$. Now observe that $\vec{\mu}$ and $\vec{\mu}'$ are in $U^*_\iota\cap M_\ga[G^*_\iota]$: first, $U^*\seq U^*_\iota$; second, all of the posets under consideration are c.c.c.\ by Assumption 4.1, and therefore $M_\ga[G^*]$ has the same ordinals as $M_\ga[G^*_\iota]$. Since $\vec{\mu},\vec{\mu}'\in U^*_\iota\cap M_\ga[G^*_\iota]$,  $\vec{\mu},\vec{\mu}'$ fail to match $\dot{a}_0,\dots,\dot{a}_l\cup\lb p_\iota\mapsto 0\rb$ at some $\eta\in I^*(\iota)$. That is to say, one of the following holds:
\begin{enumerate}
\item[(a)] there is some $k\leq l$ such that 
$$
\dot{\chi}_\ka\left[\left(\psi^*_{\iota,\eta}\right)^{-1}\left(G^*_\iota\res B^*_{\iota,\eta}\right)\right](\mu_k,\mu'_k)=1-\dot{a}_k\left[\left(\psi^*_{\iota,\eta}\right)^{-1}\left(G^*_\iota\res B^*_{\iota,\eta}\right)\right]
$$
(and in case $k=l$, $\dot{a}_k\left[\left(\psi^*_{\iota,\eta}\right)^{-1}\left(G^*_\iota\res B^*_{\iota,\eta}\right)\right]$ is defined);
\item[(b)] or $(\lb p_\iota\mapsto 0\rb)\left[\left(\psi^*_{\iota,\eta}\right)^{-1}\left(G^*_\iota\res B^*_{\iota,\eta}\right)\right]$ is defined and 
$$
\dot{\chi}_\ka\left[\left(\psi^*_{\iota,\eta}\right)^{-1}\left(G^*_\iota\res B^*_{\iota,\eta}\right)\right](\mu_l,\mu'_l)=1-(\lb p_\iota\mapsto 0\rb)\left[\left(\psi^*_{\iota,\eta}\right)^{-1}\left(G^*_\iota\res B^*_{\iota,\eta}\right)\right].
$$
\end{enumerate}
However, we assumed that $\vec{\mu}$ and $\vec{\mu}'$ match $\dot{a}_0,\dots,\dot{a}_l$ on  $\mathcal{S}^*$. Therefore (a) is false and (b) holds. This implies in particular that $\psi^*_{\iota,\eta}(p_\iota)\in G^*_\iota\res B^*_{\iota,\eta}$ and that $(\lb p_\iota\mapsto 0\rb)\left[\left(\psi^*_{\iota,\eta}\right)^{-1}\left(G^*_\iota\res B^*_{\iota,\eta}\right)\right]=0$. Thus
$$
\dot{\chi}_\ka\left[\left(\psi^*_{\iota,\eta}\right)^{-1}\left(G^*_\iota\res B^*_{\iota,\eta}\right)\right](\mu_l,\mu'_l)=1-(\dot{a}_l\cup\lb p_\iota\mapsto 0\rb)\left[\left(\psi^*_{\iota,\eta}\right)^{-1}\left(G^*_\iota\res B^*_{\iota,\eta}\right)\right]=1.
$$ 
Since $p^*_\iota\in G^*_\iota$, $p^*_\iota$ and $\psi^*_{\iota,\eta}(p_\iota)$ are compatible, and therefore $p^*_\iota$, being decisive, extends $\psi^*_{\iota,\eta}(p_\iota)$. Thus $\eta\in J^*(\iota)$.
\end{proof}

This completes the proof of the above claim. As a result, we fix some function $h$ on $J_p$ such that for each $\iota\in J_p$, $h(\iota)\in J^*(\iota)$ provides a witness to the claim for $\iota$. Let $k<n$ such that $h=h_k\res J_p$. We now check that $\vec{\mu},\vec{\mu}'$ match $\dot{a}_0,\dots,\dot{a}_l\cup\lb p\mapsto 1\rb$ on $\pi_k\left(\bar{\mathcal{S}}\right)$.

Observe that since $\vec{\mu}$ and $\vec{\mu}'$ match $\dot{a}_0,\dots,\dot{a}_{l}$ on $\mathcal{S}^*$, we only need to check that for each $\iota\in J$, if $p\in (\psi^*_{\iota,h_k(\iota)})^{-1}\left( G^*\res B^*_{\iota,h_k(\iota)}\right)$, then 
$$
\dot{\chi}_\ka\left[\left(\psi^*_{\iota,h_k(\iota)}\right)^{-1}\left(G^*_\iota\res B^*_{\iota,h_k(\iota)}\right)\right](\mu_l,\mu_l')=1.
$$ 
But this is clear: for $\iota\in J_p$, the conclusion of the implication holds, by the last claim and the choice of $h_k$. For $\iota\not\in J_p$ the hypothesis of the implication fails, since $(\psi^*_{\iota,h_k(\iota)})^{-1}(p^*(0))$ extends $p_\iota$ which, for $\iota\not\in J_p$, is incompatible with $p$.\\

We have now completed our discussion of the illustrative case when the domain of $\R$ consists entirely of $X_0$. We next work in full generality to finish with this case; we will proceed by grafting multiple copies of the part of $\R$ outside $X_0$ onto $\R^*(0)$. In more detail, recall that the maps $\pi_k$ each embed $(\R\res X_0,\mathcal{B}\res X_0)$ into $(\R^*(0),\mathcal{B}^*_0)$. Thus we may apply Lemma \ref{lemma:graft} in $M_\ga$, once for each $k<n$, to construct a sequence of enriched partition products $\la (\R^*(k+1),\mathcal{B}^*_{k+1}):k<n\ra$ such that for each $k<n$, letting $X^*_k$ denote the domain of $\R^*(k)$, $X^*_k\seq X^*_{k+1}$, $\R^*(k+1)\res X^*_k=\R^*(k)$, $\mathcal{B}^*_k\seq\mathcal{B}^*_{k+1}$, and such that $\pi_k$ extends to an embedding, which we call $\pi^*_k$, of $(\R,\mathcal{B})$ into $(\R^*(k+1),\mathcal{B}^*_{k+1})$. We remark that by the grafting construction, for each $k<n$, 
$$
\pi^*_k[X\bsl X_0]=X^*_{k+1}\bsl X^*_k.
$$ 
Let us now use $\R^*$ to denote $\R^*(n)$, $X^*$ to denote the domain of $\R^*$, and $\mathcal{B}^*$ to denote $\mathcal{B}^*_n$. Also, observe that $\pi^*_k$ embeds $(\R,\mathcal{B})$ into $(\R^*,\mathcal{B}^*)$, since it embeds $(\R,\mathcal{B})$ into $(\R^*(k+1),\mathcal{B}_{k+1})$ and since $\mathcal{B}_{k+1}\seq\mathcal{B}^*$ and $\R^*(k+1)=\R^*\res X^*_{k+1}$. We claim that $(\R^*,\mathcal{B}^*)$ witnesses the lemma in this case.

We first address item (a). Since $(\R^*(0),\mathcal{B}^*_0)$ and $(\R,\mathcal{B})$ are both finitely generated and since $(\R^*,\mathcal{B}^*)$ was constructed from them by finitely-many applications of the Grafting Lemma, $(\R^*,\mathcal{B}^*)$ is itself finitely generated by Lemma \ref{lemma:graftfg}. Moreover, as all of the partition products under consideration are in $M_\ga$, the suitable collection and auxiliary part for $(\R^*,\mathcal{B}^*)$ are also in $M_\ga$.

For (b), we define a sequence of conditions in $\R^*$ by recursion, beginning with $p^*(0)$. Suppose that we have constructed the condition $p^*(k)$ in $\R^*(k)$ such that if $k>0$, then $ p^*(k)\res\R^*(k-1)= p^*(k-1)$ and $ p^*(k)$ extends $\pi^*_{k-1}(\bar{p})$. To construct $ p^*(k+1)$, note that $ p^*(k)$ extends $\pi_k(\bar{p}\res X_0)$, since $ p^*(0)$ does, as observed before the illustrative case, and since $ p^*(k)\res\R^*(0)= p^*(0)$. Moreover, 
$$
\pi^*_k[X\bsl X_0]\cap\dom(p^*(k))=\es,
$$
as $\dom(p^*(k))\seq X^*_k$, and as $\pi^*_k[X\bsl X_0]\cap X^*_k=\es$. Thus we see that 
$$
p^*(k+1):= p^*(k)\cup\pi_k^*\left(\bar{p}\res \left(X\bsl X_0\right)\right)
$$ 
is a condition in $\R^*(k+1)$ which extends $\pi^*_k(\bar{p})$. This completes the construction of the sequence of conditions, and so we now let $p^*$ be the condition $ p^*(n)$ in $\R^*$.

We take the same $\R^*(0)$-name $\dot{U}^*$ for (c). To address (d), we let $\mathcal{F}:=\lb\pi^*_k:k<n\rb$. 
Each $\pi^*_k$, as noted above, is an embedding of $(\R,\mathcal{B})$ into $(\R^*,\mathcal{B}^*)$ and a member of $M_\ga$.

This now defines the objects from (a)-(d), and so we check that conditions (1)-(3) hold. By the construction of $p^*$ above, $p^*$ extends $\pi^*_k(\bar{p})$ for each $k<n$, so (1) is satisfied. Moreover, we already know that $p^*\Vdash_{\R^*}\vec{\nu}\in\dot{U}^*$, since $p^*(0)\Vdash_{\R^*(0)}\vec{\nu}\in\dot{U}^*$ and since $\R^*\res X^*_0=\R^*(0)$. And finally, the proof of condition (3) is the same as in the illustrative case, using the fact that each $\pi^*_k$ extends $\pi_k$. This completes the proof of the lemma in the case that ht$(\bar{\mathcal{S}})=0$.\\

\textbf{Case 2:} $\htt\left(\bar{\mathcal{S}}\right)>0$ (in particular, $\bar{\mathcal{S}}$ has at least 2 elements). We abbreviate $\htt\left(\bar{\mathcal{S}}\right)$ by $\de$ in what follows. Fix $\iota_0,\iota_1\in J$ which satisfy $\de=\operatorname{ht}(B_{\iota_0}, B_{\iota_1})$, and set $\hat{J}:=J\bsl\lb\iota_0\rb$. 

By Lemma \ref{lemma:suitablecohere}, $\hat{X}_0:=\bigcup_{\iota\in\hat{J}}B_\iota$ coheres with $(\R,\mathcal{B})$. Let $\hat{\R}$ be the partition product $\R\res \hat{X}_0$, and set $\hat{\mathcal{B}}:=\mathcal{B}\res \hat{X}_0$, which, by Lemma \ref{lemma:restrictenrichment}, is an enrichment of $\hat{\R}$. Furthermore, $\hat{\R}$ is finitely generated with an empty auxiliary part and with $\hat{\mathcal{S}}:=\lb \la B_\iota,\psi_\iota\ra:\iota\in\hat{J}\rb$ as $\ka$-suitable with respect to $(\hat{\R},\hat{\mathcal{B}})$. We also let $\hat{p}$ be the condition $\bar{p}\res \hat{X}_0\in\hat{\R}$. Finally, we let $\bar{\R}:=\R\res\bigcup_{\iota\in J}B_\iota$, and $\bar{\mathcal{B}}=\mathcal{B}\res\bigcup_{\iota\in J}B_\iota$, so that $(\bar{\R},\bar{\mathcal{B}})$ is also an enriched partition product.

Since $|\hat{\mathcal{S}}|<|\bar{\mathcal{S}}|$ and $\htt(\hat{\mathcal{S}})\leq\htt(\bar{\mathcal{S}})$, we may apply the induction hypothesis to $(\hat{\R},\hat{\mathcal{B}})$, the condition $\hat{p}$, the $\hat{\R}$-name for all $l+1$-tuples in $\om_1$, and with $\hat{\mathcal{S}}$ as the matching core. This produces the following objects:
\begin{enumerate}
\item[$(a)^*$] an enriched  partition product $(\R^*,\mathcal{B}^*)$ with domain $X^*$, say, finitely generated by a $\ka$-suitable collection $\mathcal{S}^*$ and an auxiliary part $Z^*$, all of which are in $M_\ga$;
\item[$(b)^*$] a condition $p^*\in\R^*$;
\item[$(c)^*$] an $\R^*$-name $\dot{W}^*$ in $M_\ga$ for a collection of $l+1$-tuples in $\om_1$;
\item[$(d)^*$] a nonempty, finite collection $\mathcal{F}$ in $M_\ga$ of embeddings of $(\hat{\R},\hat{\mathcal{B}})$ into $(\R^*,\mathcal{B}^*)$;
\end{enumerate} 
satisfying that for each $\pi\in\mathcal{F}:$
\begin{enumerate}
\item[$(1)^*$] $p^*$ extends $\pi(\hat{p})$ in $\R^*$;
\end{enumerate}
and also satisfying that $p^*$ forces the following statements in $\R^*:$
\begin{enumerate}
\item[$(2)^*$] $\vec{\nu}\in\dot{W}^*$;
\item[$(3)^*$] for any pairwise distinct $l+1$-tuples $\vec{\mu}$ and $\vec{\mu}'$ in $\dot{W}^{*}\cap M_\ga[\dot{G}^*]$, if $\vec{\mu}$ and $\vec{\mu}'$ match $\dot{a}_0,\dots,\dot{a}_l$ on $\mathcal{S}^*$, then there is some $\pi\in\mathcal{F}$ such that $\vec{\mu}$ and $\vec{\mu}'$ match $\dot{a}_0,\dots,\dot{a}_l\cup\lb p\mapsto 1\rb$ on $\pi\left(\hat{\mathcal{S}}\right)$.
\end{enumerate}

Our next step is to restore many copies of the segment $\psi_{\iota_0}[\rho_\ka\bsl\de]$ of the lost branch $ B_{\iota_0}$ in such a way that the restored copies form a $\ka$-suitable collection with smaller height than $\de$; this will allow another application of the induction hypothesis. Towards this end, define 
$$
\mathcal{R}:=\lb\pi\circ\psi_{\iota_1}[\de]:\pi\in\mathcal{F}\rb,
$$
and, recalling that $\mathcal{F}$ is finite, let $x_0,\dots,x_{d-1}$ enumerate $\mathcal{R}$. We choose, for each $k<d$, a map $\pi_k\in\mathcal{F}$ so that $\pi_k\circ\psi_{\iota_1}[\de]=x_k$.

We now work in $M_\ga$ to graft one copy of $\psi_{\iota_0}[\rho_\ka\bsl\de]$ onto $(\R^*,\mathcal{B}^*)$ over $\pi_k$, for each $k<n$. Indeed, since $\pi_k$ embeds $(\hat{\R},\hat{\mathcal{B}})$ into $(\R^*,\mathcal{B}^*)$, we may successively apply the Grafting Lemma to find an enriched partition product $(\R^{**},\mathcal{B}^{**})$ on a domain $X^{**}$ so that $\R^{**}\res X^*=\R^*$, $\mathcal{B}^*\seq\mathcal{B}^{**}$, and so that for each $k<d$, $\pi_k$ extends to an embedding $\pi_k^*$ of $(\bar{\R},\bar{\mathcal{B}})$ into $(\R^{**},\mathcal{B}^{**})$. Since $(\R^{**},\mathcal{B}^{**})$ is finitely generated, by Lemma \ref{lemma:graftfg}, we may let $\mathcal{S}^{**}$ denote the finite, $\ka$-suitable collection for $(\R^{**},\mathcal{B}^{**})$.

Let us make a number of observations about the above situation. First, we want to see that for each $\pi\in\mathcal{F}$, we may extend $\pi$ to embed $(\bar{\R},\bar{\mathcal{B}})$ into $(\R^{**},\mathcal{B}^{**})$. Thus fix $\pi\in\mathcal{F}$, and let $k<d$ such that $\pi\circ\psi_{\iota_1}[\de]=x_k$. We want to apply Lemma \ref{restoretrim}, 
with (following the notation of the lemma) $X_0=\bigcup_{\iota\in\hat{J}} B_\iota$ and $X_1=B_{\iota_0}$. For
this we need to see that 
$\pi[X_0\cap B_{\iota_0}]=\pi_k[X_0\cap B_{\iota_0}]$. 
To verify this, we first claim that $X_0\cap B_{\iota_0}=B_{\iota_1}\cap B_{\iota_0}$. Suppose that this is false, for a contradiction. Then there is some $\al\in X_0\cap B_{\iota_0}\bsl B_{\iota_1}$. Fix 
$\iota\in\hat{J}$ 
s.t. $\al\in B_\iota\cap B_{\iota_0}$. Then $\psi_{\iota_0}^{-1}[B_\iota\cap B_{\iota_0}]\leq\htt(\bar{\mathcal{S}})=\de$, and so $\al\in\psi_{\iota_0}[\de]$. But $\psi_{\iota_0}\res\de=\psi_{\iota_1}\res\de$, and therefore $\al\in B_{\iota_1}$, a contradiction. 

Thus $X_0\cap B_{\iota_0}=B_{\iota_1}\cap B_{\iota_0}$. But $B_{\iota_1}\cap B_{\iota_0}=\psi_{\iota_1}[\de]$, and therefore
$$
\pi[B_{\iota_1}\cap B_{\iota_0}]=\pi\circ\psi_{\iota_1}[\de]=x_k=\pi_k\circ\psi_{\iota_1}[\de]=\pi_k[B_{\iota_1}\cap B_{\iota_0}].
$$
Hence 
$\pi[X_0\cap B_{\iota_0}]=\pi_k[X_0\cap B_{\iota_0}]$.
By Lemma \ref{restoretrim}, the map 
$$
\pi^*:=\pi\cup\pi^*_k\res\left(\psi_{\iota_0}[\rho_\ka\bsl\de] \right)
$$ 
is an extension of $\pi$ which embeds $(\bar{\R},\bar{\mathcal{B}})$ into $(\R^{**},\mathcal{B}^{**})$. We make the observation that $\pi^*[B_{\iota_0}]=\pi^*_k[B_{\iota_0}]$, which will be useful later.

For each $k<d$, we use $x^*_k$ to denote the image of $B_{\iota_0}$ under the map $\pi^*_k$. Let $\bar{\mathcal S}^{**}:=\lb\la x^*_k,\pi^*_k\circ\psi_{\iota_0},\ka\ra:k<d\rb$. Then $\bar{\mathcal S}^{**}\subseteq {\mathcal S}^{**}$, and in particular, $\bar{\mathcal S}^{**}$ is $\kappa$-suitable. For 
$k\not=m$
we have 
$$
(\pi^*_k\circ \psi_{\iota_0})[\delta]=x_k\not= x_m=(\pi^*_m\circ \psi_{\iota_0})[\delta],
$$
and hence $\htt(x^*_k,x^*_m)<\delta$. Therefore $\htt(\bar{\mathcal{S}}^{**})<\de$, since $\bar{\mathcal{S}}^{**}$ is finite.

We now have a collection $\mathcal{F}^{*}:=\lb\pi^*:\pi\in\mathcal{F}\rb$ of embeddings of $(\bar{\R},\bar{\mathcal{B}})$ into $(\R^{**},\mathcal{B}^{**})$ and a finite, $\ka$-suitable subcollection $\bar{\mathcal{S}}^{**}$ of $\mathcal{S}^{**}$ such that the height of $\bar{\mathcal{S}}^{**}$ is less than $\de$. But before we apply the induction hypothesis, we need to extend $(\R^{**},\mathcal{B}^{**})$ to add generics for the full $\R$ and to also define a few more objects. Towards this end, we work in $M_\ga$ to successively apply the Grafting Lemma to each map $\pi^*$ in $\mathcal{F}^{*}$ to graft $(\R,\mathcal{B})$ onto $(\R^{**},\mathcal{B}^{**})$ over $\pi^*$. This results in a partition product $(\R^{***},\mathcal{B}^{***})$ in $M_\ga$ with domain $X^{***}$ so that $\R^{***}\res X^{**}=\R^{**}$, $\mathcal{B}^{**}\seq\mathcal{B}^{***}$, and so that each map $\pi^*\in\mathcal{F}^*$ extends to an embedding $\pi^{***}$ of $(\R,\mathcal{B})$ into $(\R^{***},\mathcal{B}^{***})$. By Lemma \ref{lemma:graftfg}, $(\R^{***},\mathcal{B}^{***})$ is still finitely generated, say with $\ka$-suitable collection $\mathcal{S}^{***}$.

We now want to define a condition $p^{***}$ in $\R^{***}$ by adding further coordinates to the condition $p^*\in\R^*\seq\R^{***}$ from $(a)^*$. By the grafting construction of $\R^{**}$, if 
$k<m<d$, 
then the images of $\psi_{\iota_0}[\rho_\ka\bsl\de]$ under $\pi^*_k$ and $\pi^*_m$ are disjoint. Thus
$$
p^{**}:=p^*\cup\bigcup_{k<d}\pi^*_k\left(\bar{p}\res\psi_{\iota_0}[\rho_{\ka}\bsl\de]\right)
$$
is a condition in $\R^{**}$. Since by $(1)^*$, $p^*$ extends $\pi(\hat{p})$ in $\R^*$ for each $\pi\in\mathcal{F}$, we conclude that $p^{**}$ extends $\pi^*_k(\bar{p}\res\bigcup_{\iota\in J}B_\iota)$ for each $k<d$. Furthermore, if $\pi^*\in\mathcal{F}^*$, then for some $k<d$, $\pi^*$ agrees with $\pi^*_k$ on $B_{\iota_0}$, as observed above. It is straightforward to see that this implies that $p^{**}$ in fact extends $\pi^*(\bar{p}\res\bigcup_{\iota\in J}B_\iota)$ for each $\pi^*\in\mathcal{F}^*$. And finally, by the grafting construction of $\R^{***}$, we know that if $\pi$ and $\si$ are distinct embeddings in $\mathcal{F}$, then the images of $X\bsl\bigcup_{\iota\in J}B_\iota$ under $\pi^{***}$ and $\si^{***}$ are disjoint. Consequently,
$$
p^{***}:=p^{**}\cup\bigcup_{\pi\in\mathcal{F}}\pi^{***}\left(\bar{p}\res\left(X\bsl\bigcup_{\iota\in J}B_\iota\right)\right)
$$
is a condition in $\R^{***}$ which extends $\pi^{***}(\bar{p})$ for each $\pi\in\mathcal{F}$.

 We are now ready to apply the induction hypothesis to the partition product $(\R^{***},\mathcal{B}^{***})$, the condition $p^{***}\in\R^{***}$, and the matching core $\bar{\mathcal{S}}^{**}$, which has height below $\de$. This results in the following objects:
\begin{enumerate}
\item[$(a)^{**}$] an enriched partition product $(\R^{****},\mathcal{B}^{****})$ on a set $X^{****}$ which is finitely generated, say with $\ka$-suitable collection $\mathcal{S}^{****}$  and auxiliary part $Z^{****}$, all of which are in $M_\ga$;
\item[$(b)^{**}$] a condition $ p^{****}$ in $\R^{****}$;
\item[$(c)^{**}$] an $\R^{****}$-name $\dot{U}^{****}$ in $M_\ga$ for a collection of $l+1$ tuples in $\om_1$;
\item[$(d)^{**}$] a nonempty, finite collection $\mathcal{G}$ in $M_\ga$ of embeddings of $(\R^{***},\mathcal{B}^{***})$ into $(\R^{****},\mathcal{B}^{****})$;
\end{enumerate}
satisfying that for each $\si\in\mathcal{G}$
\begin{enumerate}
\item[$(1)^{**}$]  $p^{****}$ extends $\si(p^{***})$ in $\R^{****}$;
\end{enumerate}
and such that $ p^{****}$ forces in $\R^{****}$ that
\begin{enumerate}
\item[$(2)^{**}$] $\vec{\nu}\in\dot{U}^{****}$;
\item[$(3)^{**}$] for any pairwise distinct tuples $\vec{\mu},\vec{\mu}'$ in $M_\ga[\dot{G}^{****}]\cap \dot{U}^{****}$ such that $\vec{\mu},\vec{\mu}'$ match $\dot{a}_0,\dots,\dot{a}_l$ on $\mathcal{S}^{****}$, there is some $\si\in\mathcal{G}$ such that $\vec{\mu},\vec{\mu}'$ match $\dot{a}_0,\dots,\dot{a}_l\cup\lb p\mapsto 1\rb$ on $\si\left(\bar{\mathcal{S}}^{**}\right)$.
\end{enumerate}

This completes the construction of our final partition product. To finish the proof, we will need to define a number of embeddings from our original partition product $(\R,\mathcal{B})$ into $(\R^{****},\mathcal{B}^{****})$ and check that the appropriate matching obtains. For $\si\in\mathcal{G}$ and $\pi\in\mathcal{F}$, we define the map $\tau(\pi,\si)$ to be the composition $\si\circ\pi^{***}$, which embeds $(\R,\mathcal{B})$ into $(\R^{****},\mathcal{B}^{****})$. We also observe that $p^{****}\leq\tau(\pi,\si)(\bar{p})$ for each such $\pi$ and $\si$ since $p^{****}$ extends $\si(p^{***})$ in $\R^{****}$, and since $p^{***}$ extends $\pi^{***}(\bar{p})$ in $\R^{***}$. Now define the $\R^{****}$-name $\dot{V}^*$ to be
$$
\dot{U}^{****}\cap\bigcap_{\si\in\mathcal{G}}\dot{W}^*\left[\si^{-1}\left(\dot{G}^{****}\right)\res X^*\right].
$$
We observe that this is well-defined, since for each $\si\in\mathcal{G}$ and generic $G^{****}$ for $\R^{****}$, $\si^{-1}\left(G^{****}\right)$ is generic for $\R^{***}$, and hence its restriction to $X^*$ is generic for $\R^*$. We also see that $p^{****}$ forces that $\vec{\nu}\in\dot{V}^*$ because $p^{****}$ forces $\vec{\nu}\in\dot{U}^{****}$,  $p^*$  is in $\si^{-1}\left(G^{****}\right)$ for any generic $G^{****}$ containing $p^{****}$, and $p^*$ forces in $\R^*$ that $\vec{\nu}\in\dot{W}^*$.

We finish the proof of Lemma \ref{lemma:main} in this case by showing that the partition product $(\R^{****},\mathcal{B}^{****})$, the condition $p^{****}\in\R^{****}$, the name $\dot{V}^*$, and the collection $\lb\tau(\pi,\si):\pi\in\mathcal{F}\we\si\in\mathcal{G}\rb$ of embeddings satisfy (1)-(3). We already know that $p^{****}$ extends $\tau(\pi,\si)(\bar{p})$ for each $\pi$ and $\si$ and that $p^{****}\Vdash\vec{\nu}\in\dot{V}^*$. So now we check the matching condition. Towards this end, fix a generic $H$ for $\R^{****}$ and two pairwise distinct tuples $\vec{\mu},\vec{\mu}'$ in $\dot{V}^*[H]\cap M_\ga[H]$ which match $\dot{a}_0,\dots,\dot{a}_l$ on $\mathcal{S}^{****}$. We need to find some $\pi$ and $\si$ such that $\vec{\mu},\vec{\mu}'$ match $\dot{a}_0,\dots,\dot{a}_l\cup\lb p\mapsto 1\rb$ on $\tau(\pi,\si)\left(\bar{\mathcal{S}}\right)$. 

By $(3)^{**}$, we know that we can find some $\si$ such that 
\begin{enumerate}
\item[(i)] $\vec{\mu}$ and $\vec{\mu}'$ match $\dot{a}_0,\dots,\dot{a}_l\cup\lb p\mapsto 1\rb$ on $\si\left(\bar{\mathcal{S}}^{**}\right)$.
\end{enumerate}
Let $t$ denote the triple $\la B_{\iota_0},\psi_{\iota_0},\ka\ra$. By construction of the maps $\pi^*$, for each $\pi\in {\mathcal F}$, there is some $k$ so that $\pi^{***}(t)=\pi^*(t)=\pi^*_k(t)\in \bar{\mathcal S}^{**}$. Using (i) it follows that: 
\begin{enumerate}
\item[(ii)] for every $\pi\in {\mathcal F}$, $\vec{\mu}$ and $\vec{\mu}'$ match $\dot{a}_0,\dots,\dot{a}_l\cup\{p\mapsto 1\}$ at $\sigma\circ \pi^{***}(t)=\tau(\pi,\sigma)(t)$. 
\end{enumerate}

 Now consider the filter $G^*_\si:=\si^{-1}\left(H\right)\res X^*$, which is generic for $\R^{*}$ and contains $p^{*}$. By Assumption 4.1, we know that all the posets under consideration are c.c.c., and therefore the models $M_\ga[H]$ and $M_\ga[G^*_\si]$ have the same ordinals, namely those of $M_\ga$. Thus $\vec{\mu},\vec{\mu}'\in M_\ga[G^*_\si]$. Furthermore, by definition of $\dot{V}^*[H]$, we have that $\vec{\mu},\vec{\mu}'\in\dot{W}^*[G^*_\si]$, and as a result $\vec{\mu},\vec{\mu}'\in M_\ga[G^*_\si]\cap\dot{W}^*[G^*_\si]$. Thus by $(3)^*$, we can find some $\pi\in\mathcal{F}$ so that $\vec{\mu},\vec{\mu}'$ match $\dot{a}_0,\dots,\dot{a}_l\cup\lb p\mapsto 1\rb$ on $\pi\left(\hat{\mathcal{S}}\right)$. Because $\pi^{***}$ extends $\pi$, we may rephrase this to say that $\vec{\mu},\vec{\mu}'$ match $\dot{a}_0,\dots,\dot{a}_l\cup\lb p\mapsto 1\rb$ on $\pi^{***}\left(\hat{\mathcal{S}}\right)$. Since $\si$ embeds $(\R^{***},\mathcal{B}^{***})$ into $(\R^{****},\mathcal{B}^{****})$, 
\begin{enumerate}
\item[(iii)]$\vec{\mu},\vec{\mu}'$ match $\dot{a}_0,\dots,\dot{a}_l\cup\lb p\mapsto 1\rb$ on $\tau(\pi,\si)\left(\hat{\mathcal{S}}\right)$. 
\end{enumerate}
Finally, (ii) and (iii) imply that $\vec{\mu},\vec{\mu}'$ match $\dot{a}_0,\dots,\dot{a}_l\cup\lb p\mapsto 1\rb$ on $\tau(\pi,\si)\left(\bar{\mathcal{S}}\right)$, as $\bar{\mathcal{S}}=\hat{\mathcal{S}}\cup\lb t\rb$. This completes the proof of Lemma \ref{lemma:main}.

\end{proof}

\begin{corollary}\label{cor:includename} Under the assumptions of Lemma \ref{lemma:main}, suppose that $\dot{U}$ is an $\R$-name in $M_\ga$ for a set of $l+1$-tuples in $\om_1$ such that $\bar{p}\Vdash_\R\vec{\nu}\in\dot{U}$. Then the conclusion of Lemma \ref{lemma:main} may be strengthened to say that $p^*\Vdash_{\R^*}\dot{U}^*\seq\bigcap_{\pi\in\mathcal{F}}\dot{U}\left[\pi^{-1}\left(\dot{G}^*\right)\right]$.
\end{corollary}
\begin{proof} Let $\dot{U}$ be fixed, and let $\dot{U}^*$ be as in the conclusion of Lemma \ref{lemma:main}. Define $\dot{U}^{**}$ to be the name $\dot{U}^*\cap\bigcap_{\pi\in\mathcal{F}}\dot{U}\left[\pi^{-1}\left(\dot{G}^*\right)\right]$, and observe that this name is still in $M_\ga$. By condition (1) of Lemma \ref{lemma:main}, we know that $p^*$ forces that $\bar{p}$ is in $\pi^{-1}\left(\dot{G}^*\right)$, for each $\pi\in\mathcal{F}$. Since each such $\pi^{-1}\left(\dot{G}^*\right)$ is forced to be $V$-generic for $\R$ and since $\bar{p}\Vdash_\R\vec{\nu}\in\dot{U}$, this implies that $p^*$ forces that $\vec{\nu}$ is a member of $\dot{U}^{**}$. Finally, condition (3) of Lemma \ref{lemma:main} still holds, since $\dot{U}^{**}$ is forced to be a subset of $\dot{U}^*$.
\end{proof}

\begin{corollary}\label{cor:dogs}\emph{(Under Assumption \ref{ass2})} $\dot{a}_0,\dots,\dot{a}_l\cup\lb p\mapsto 1\rb$ have the partition product preassignment property at $\ga$. 
\end{corollary}
\begin{proof} Suppose otherwise, for a contradiction. Then there exists a partition product $\R$, say with domain $X$, finitely generated by $\mathcal{S}=\lb\la B_\iota,\psi_\iota\ra:\iota\in I\rb$ and an auxiliary part $Z$, all of which are in $M_\ga$; an $\R$-name $\dot{U}$ in $M_\ga$; and a condition $\bar{p}\in\R$ (not necessarily in $M_\ga$), such that $\bar{p}$ forces that $\vec{\nu}\in\dot{U}$, but also that for any pairwise distinct tuples $\vec{\mu},\vec{\mu}'$ in $\dot{U}\cap M_\ga[\dot{G}]$, there exists some $\iota_0\in I$ such that $\vec{\mu},\vec{\mu}'$ fail to match $\dot{a}_0,\dots,\dot{a}_l\cup\lb p\mapsto 1\rb$ at $\iota_0$. Apply Lemma \ref{lemma:main} and Corollary \ref{cor:includename} to these objects, with $\bar{\mathcal{S}}:=\mathcal{S}$ and with the enrichment 
$$
\mathcal{B}:=\lb\la b(\xi),\pi_\xi,\ind(\xi)\ra:\xi\in X\rb\cup\mathcal{S},
$$
to construct the objects as in the conclusions of Lemma \ref{lemma:main} and Corollary \ref{cor:includename}. Also, fix a generic $G^*$ for $\R^*$ which contains the condition $p^*$. 

We now apply the fact that $\dot{a}_0,\dots,\dot{a}_l$ have the partition product preassignment property at $\ga$ to the objects in the conclusion of Lemma \ref{lemma:main}: since $\vec{\nu}\in U^*:=\dot{U}^*[G^*]$, we can find two pairwise distinct tuples $\vec{\mu},\vec{\mu}'$ in $U^*\cap M_\ga[G^*]$ which match $\dot{a}_0,\dots,\dot{a}_l$ on 
${\cal S}^*$. 
Thus by (3) of Lemma \ref{lemma:main}, there is some embedding $\pi$ of $(\R,\mathcal{B})$ into $(\R^*,\mathcal{B}^*)$ so that $\vec{\mu},\vec{\mu}'$ match $\dot{a}_0,\dots,\dot{a}_l\cup\lb p\mapsto 1\rb$ on $\pi\left(\mathcal{S}\right)$. Now consider $G:=\pi^{-1}\left(G^*\right)$, which is generic for $\R$ and contains the condition $\bar{p}$, since $p^*\leq_{\R^*}\pi(\bar{p})$. Since $\vec{\mu},\vec{\mu}'$ match $\dot{a}_0,\dots,\dot{a}_l\cup\lb p\mapsto 1\rb$ on $\pi\left(\mathcal{S}\right)$ and $\pi$ is an embedding, $\vec{\mu},\vec{\mu}'$ match $\dot{a}_0,\dots,\dot{a}_l\cup\lb p\mapsto 1\rb$ on $\mathcal{S}$ with respect to the filter $G$. Finally, observe that $\vec{\mu}$ and $\vec{\mu}'$ are both in $\dot{U}[G]\cap M_\ga[G]$: they are in $\dot{U}[G]$ by Corollary \ref{cor:includename}, since $U^*$ is a subset of $\dot{U}[G]$. They are both in $M_\ga$, hence in $M_\ga[G]$, since by Assumption 4.1, $\R^*$ is c.c.c. However, this contradicts what we assumed about $\bar{p}$.
\end{proof}

\subsection{Putting it together} 

Let us now put together the results so far.

\begin{lemma}\label{lemma:alltogether} Suppose that $\dot{a}_0,\dots,\dot{a}_{l-1}$ are total canonical color names which have the partition product preassignment property at $\gamma$. Then there is a total canonical color name $\dot{a}_l$ so that $\dot{a}_0,\dots,\dot{a}_l$ have the partition product preassignment property at $\ga$.
\end{lemma}
\begin{proof} We recursively construct a sequence $\dot{a}^\xi_l$ of names, taking unions at limit stages,
and starting with the empty name $\dot{a}^0_l=\emptyset$. 
If $\dot{a}^\xi_l$ has been constructed and $\dom(\dot{a}^\xi_l)$ is a maximal antichain in $\ps_\ka$, we set $\dot{a}_l=\dot{a}^\xi_l$. Otherwise, we pick some condition $p\in\ps_\ka$ incompatible with all conditions therein. If there is some extension $p^*\leq_{\ps_\ka}p$ so that $\dot{a}_0,\dots,\dot{a}_{l-1},\dot{a}^\xi_l\cup\lb p^*\mapsto 0\rb$ have the partition product preassignment property at $\ga$, we pick some such $p^*$ and set $\dot{a}^{\xi+1}_l:=\dot{a}^\xi_l\cup\lb p^*\mapsto 0\rb$. Otherwise, Assumption \ref{ass2} is satisfied, and hence by Corollary \ref{cor:dogs}, $\dot{a}_0,\dots,\dot{a}_{l-1},\dot{a}^\xi_l\cup\lb p\mapsto 1\rb$ have the partition product preassignment property at $\ga$. In this case we set $\dot{a}^{\xi+1}_l:=\dot{a}^\xi_l\cup\lb p\mapsto 1\rb$. Note that the construction of the sequence $\dot{a}^\zeta_l$ halts at some countable stage, since $\ps_\ka$ is c.c.c., by Assumption 4.1.
\end{proof}

We now prove Proposition \ref{prop:name}:
\begin{proof}[Proof of Proposition \ref{prop:name}] Recall that for each $\ga<\om_1$, $\la\nu_{\ga,l}:l<\om\ra$ enumerates the slice $[M_\ga\cap\om_1,M_{\ga+1}\cap\om_1)$. By Lemma \ref{lemma:alltogether}, we may construct, for each $\ga<\om_1$, a sequence of $\ps_\ka$-names $\la\dot{a}_{\ga,l}:l<\om\ra$ such that for each $l<\om$, $\dot{a}_{\ga,0},\dots,\dot{a}_{\ga,l}$ have the partition product preassignment property at $\ga$. We now define a function $\dot{f}$ by taking $\dot{f}(\nu_{\ga,l})=\dot{a}_{\ga,l}$, for each $\ga<\om_1$ and $l<\om$. The values of $\dot{f}$ on ordinals $\nu<M_0\cap\om_1$ are irrelevant, so we simply set $\dot{f}(\nu)$ to name 0 for each such $\nu$. Then $\dot{f}$ satisfies the assumptions of Lemma \ref{lemma:suffices} and hence satisfies Proposition \ref{prop:name}.
\end{proof}

\section{Constructing Partition Products in $L$}
\label{sec.constructing}

In this section, we show how to construct the desired partition products in $L$. In particular, we will construct a partition product $\ps_{\om_2}$, which will have domain $\omega_3$. Forcing with $\ps_{\om_2}$ will provide the model which witnesses our theorem. We assume for this section that $V=L$. 

Before we introduce some more definitions, let us fix a finite fragment $F$ of $\zfc-\mathsf{Powerset}$ (hence satisfied in $H(\omega_3)$) large enough to prove the existence and bijectability with $X$ of elementary Skolem hulls of $X$ in levels of $L$, and to construct the partition product $\bb{P}_\ka\res\ga$ of Subsection \ref{subsection-building-the-partition-products} for $\gamma<\gamma(\kappa)$, from $\gamma$ and the sequences $\la\de_i(\kappa):i<\gamma\ra$ and ${\mathbb A}\res\kappa$. We will spell out exactly what $F$ needs to prove in Remark \ref{what-is-F}, after we establish the relevant notation. It is not hard to check that these constructions (including the rearrangements that go into the construction in Subsection \ref{subsection-building-the-partition-products}) use only $\zfc-\mathsf{Powerset}$. 

As a matter of notation, by the G{\"o}del pairing function, we view each ordinal $\ga$ as coding a pair of ordinals $(\ga)_0$ and $(\ga)_1$. We will use this for bookkeeping arguments later, where $(\ga)_0$ will select elements under $<_L$ and where $(\ga)_1$ will select various prior stages in an iteration.

\subsection{Local $\om_2$'s and Witnesses}

\begin{definition} Let $\om_1<\ka\leq\om_2$, and let $\bb{A}$ be a sequence of elements of $L_\ka$ so that $\dom(\bb{A})\seq\ka$. We say that $\ka$ is a \emph{local} $\om_2$ \emph{with respect to} $\bb{A}$ if there is some $\de>\ka$ such that $L_\de$ is closed under $\om$-sequences, contains $\bb{A}$ as an element, and such that
$$
L_\de\models\ka=\aleph_2\we F\we\ka \text{ is the largest cardinal}.
$$

If $\ka$ is a local $\om_2$ with respect to $\bb{A}$, we will refer to any such $\de$ as above as  a \emph{witness for }$\ka$ \emph{with respect to $\bb{A}$} or simply as a \emph{witness for} $\ka$ if $\bb{A}$ is clear from context. 
\end{definition}

The symbol ``$\bb{A}$" in the above definition stands for ``alphabet," and it will later represent some initial segment of the sequence of alphabetical partition products. 

Fix some $k\geq 2$, large enough that all statements in $F$ are $\Sigma_k$. We will use $\Sigma_k$ hulls and $\Sigma_k$ elementarity throughout the section. The fact that $k\geq 2$ allows reflecting basic statements into the hulls, such as being a largest cardinal, and the fact that all statements in $F$ are $\Sigma_k$ allows reflecting $F$.

We begin our discussion with the following 
lemma. 
The proof is straightforward using the closure of $L_{\omega_1}$ under countable sequences and the fact that $\Sigma_k$ admits a universal formula.

\begin{lemma}\label{lemma:omsequences} Suppose that $L_\de$ is closed under $\om$-sequences, and let $p\in L_\de$. Then $\emph{Hull}_{k}^{L_\de}(\om_1\cup\lb p\rb)$ is also closed under $\om$-sequences.
\end{lemma}

The next lemma shows how a local $\om_2$ with respect to one parameter can project to another.

\begin{lemma}\label{lemma:project} Suppose that $\de$ is a witness for $\ka$ with respect to $\bb{A}$, and define $H:=\emph{Hull}_{k}^{L_\de}(\om_1\cup\lb\bb{A}\rb)$. Suppose further that $H\cap\ka=\bar{\ka}<\ka$. Then $\bar{\ka}$ is a local $\om_2$ with respect to $\bb{A}\res\bar{\ka}$, and $\text{ot}(H\cap\de)$ is a witness for $\bar{\ka}$ with respect to $\bb{A}\res\bar{\ka}$.
\end{lemma}
\begin{proof} Let $\pi:H\lra L_{\bar{\de}}$ be the transitive collapse, so that $\pi(\ka)=\bar{\ka}$ and $\bar{\de}=\text{ot}(H\cap\de)$. Since $H$ is closed under $\om$-sequences, by Lemma \ref{lemma:omsequences}, $L_{\bar{\de}}$ is too. Since $\pi$ is 
$\Sigma_k$
elementary, we will be done once we verify that $\pi(\bb{A})=\bb{A}\res\bar{\ka}$. Indeed, by the elementarity of $\pi$, $\pi(\bb{A})$ is a sequence with domain $\dom(\bb{A})\cap\bar{\ka}$. Furthermore, for each $i\in\dom(\bb{A})$, since $\bb{A}(i)\in L_\ka$ and since $L_\de$ satisfies that $\ka=\aleph_2$, we have that $\bb{A}(i)$ has size $\leq\aleph_1$ in $L_\de$. Thus for each $i\in\dom(\bb{A})\cap\bar{\ka}$, $\bb{A}(i)$ is not moved by $\pi$, and consequently $\pi(\bb{A})=\bb{A}\res\bar{\ka}$. 
\end{proof}

If $\ka$ is a local $\om_2$ with respect to $\bb{A}$, we define the \emph{canonical sequence of witnesses for} $\ka$ \emph{with respect to} $\bb{A}$, denoted $\la\de_i(\ka,\bb{A}):i<\ga(\ka,\bb{A})\ra$. We set $\de_0(\ka,\bb{A})$ to be the least witness for $\ka$. Suppose that $\la\de_i(\ka,\bb{A}):i<\ga\ra$ is defined, for some $\ga$. If there exists a witness $\tilde{\de}$ for $\ka$ such that $\tilde{\de}>\sup_{i<\ga}\de_i(\ka,\bb{A})$, then we set $\de_\ga(\ka,\bb{A})$ to be the least such. Otherwise, we halt the construction and set $\ga(\ka,\bb{A}):=\ga$. If we have $\ga<\ga(\ka,\bb{A})$, then we also define $H(\ka,\ga,\bb{A})$ to be 
$$
H(\ka,\ga,\bb{A}):=\text{Hull}_{k}^{L_{\de_\ga(\ka,\bb{A})}}(\om_1\cup\lb\bb{A}\rb).
$$

\begin{remark}\label{remark:csdefinable} It is straightforward to check that if $\ka$ is a local $\om_2$ with respect to $\bb{A}$ and $\ga<\ga(\ka,\bb{A})$, then because $L_{\de_\ga(\ka,\bb{A})}$ is countably closed, being a witness for $\ka$ with respect to $\bb{A}$ is absolute between $L_{\de_\ga(\ka,\bb{A})}$ and $V$. 
Indeed, the requirements on $L_\delta$ for a witness $\delta$, other than countable closure, are $\Delta_0$ in $L_\delta$.
Thus the sequence $\la\de_i(\ka,\bb{A}):i<\ga\ra$ 
is definable in $L_{\de_\ga(\ka,\bb{A})}$ as the longest sequence of witnesses for $\ka$ with respect to $\bb{A}$. 
The defining formula is $\Pi_2$, so the sequence is absolute to $H(\kappa,\gamma,\bb{A})$. Consequently both the sequence and $\gamma$ belong to this hull.
Furthermore, in the case that $\ka=\om_2$, we see that $\ga(\om_2,\bb{A})=\om_3$.
\end{remark}

For the rest of the subsection, we fix $\ka$ and $\bb{A}$; for the sake of readability, we will often drop explicit mention of the parameter $\bb{A}$ in notation of the from $\de_\ga(\ka,\bb{A})$ and $H(\ka,\ga,\bb{A})$, preferring instead to write, respectively, $\de_\ga(\ka)$ and $H(\ka,\ga)$.

Suppose that $\ka$ is such that $\ga(\ka)$ is a successor, say $\ga+1$, and further suppose that $H(\ka,\ga)$ contains $\ka$ as a subset. Then we refer to $\de_\ga(\ka)$, the final element on the canonical sequence of witnesses for $\ka$ with respect to $\bb{A}$, as the \emph{stable witness for} $\ka$ \emph{with respect to} $\bb{A}$. It is stable in the sense that we cannot condense the hull further.

\begin{lemma}\label{lemma:below} Suppose that $\ga+1<\ga(\ka)$. Then $H(\ka,\ga)\cap\ka\in\ka$.
\end{lemma}
\begin{proof} Suppose otherwise. Then $\ka\seq H(\ka,\ga)$. Since $\ga+1<\ga(\ka)$, we know that $\hat{\de}:=\de_{\ga+1}(\ka)$ exists, and in particular, $\de_\ga(\ka)<\hat{\de}$. Observe that $H(\ka,\ga)$ is a member of $L_{\hat{\de}}$; this follows from the choice of the finite fragment $F$ and the facts that $\delta(\gamma,\kappa),{\mathbb A}\in L_{\hat{\de}}$ and $L_{\hat{\de}}\models F$. Therefore, again using the fact that $L_{\hat{\de}}\models F$, we may find a surjection from $\om_1$ onto $H(\ka,\ga)$ in $L_{\hat{\de}}$. Since $\ka\seq H(\ka,\ga)$, this contradicts our assumption that $L_{\hat{\de}}$ satisfies that $\ka$ is $\aleph_2$.
\end{proof}

If $\ga+1<\ga(\ka)$, then the collapse of $H(\ka,\ga)$ moves $\ka$. The level to which $H(\ka,\ga)$ collapses is then the stable witness for the images of $\ka$ and $\bb{A}$, as shown in the following lemma.

\begin{lemma}\label{lemma:collapsestable} Suppose that $\ga+1<\ga(\ka)$, and set $\bar{\ka}:=H(\ka,\ga)\cap\ka$. Let $j$ denote the collapse map of $H(\ka,\ga)$ and $\tau$ the level to which $H(\ka,\ga)$ collapses. Finally, set $\bar{\ga}:=j(\ga)$. Then $\bar{\ga}+1=\ga(\bar{\ka})$ and $\tau=\de_{\bar{\ga}}(\bar{\ka})$ is the stable witness for $\bar{\ka}$ and $\bb{A}\res\bar{\ka}$.
\end{lemma}
\begin{proof} Let us abbreviate $H(\ka,\ga)$ by $H$. By Remark \ref{remark:csdefinable}, we have that $\la\de_i(\ka):i<\ga\ra\in H(\ka,\ga)$; let $\la \de_i:i<\bar{\ga}\ra$ denote the image of this sequence under $j$. By the elementarity of $j$ and the absoluteness of Remark \ref{remark:csdefinable}, $\la\de_i:i<\bar{\ga}\ra$ is exactly equal to $\la\de_i(\bar{\ka}):i<\bar{\ga}\ra$, the canonical sequence of witnesses for $\bar{\ka}$ with respect to $\bb{A}\res\bar{\ka}$. 

We next verify that $\tau=\de_{\bar{\ga}}(\bar{\ka})$. By Lemma \ref{lemma:project}, we know that $\tau$ is a witness for $\bar{\ka}$ with respect to $\bb{A}\res\bar{\ka}$. Furthermore, $\tau$ is the least witness for $\bar{\ka}$ above $\sup_{i<\bar{\ga}}\de_i(\bar{\ka})$: suppose that there were a witness $\bar{\de}$ for $\bar{\ka}$ between $\sup_{i<\bar{\ga}}\de_i(\bar{\ka})$ and $\tau$. Then $L_\tau$ satisfies that $\bar{\de}$ is a witness for $\bar{\ka}$. By the elementarity of $j^{-1}$, setting $\de:=j^{-1}(\bar{\de})$, we see that $L_{\de_\ga(\ka)}$ satisfies that $\de$ is a witness for $\ka$. Since $L_{\de_\ga(\ka)}$ is closed under $\om$-sequences, $\de$ is in fact a witness for $\ka$ (with respect to $\bb{A}$). As $\de$ is between $\sup_{i<\ga}\de_i(\ka)$ and $\de_\ga(\ka)$, this is a contradiction. Therefore $\tau$ is the least witness for $\bar{\ka}$ with respect to $\bb{A}\res\bar{\ka}$ which is above $\sup_{i<\bar{\ga}}\de_i(\bar{\ka})$. However, because $L_\tau$ is the collapse of $H$, we see that $\text{Hull}_{k}^{L_\tau}(\om_1\cup\lb\bb{A}\res\bar{\ka}\rb)$ is all of $L_\tau$. Therefore $\tau$ is the stable witness for $\bar{\ka}$ with respect to $\bb{A}\res\bar{\ka}$.
\end{proof}

\subsection{Building the Partition Products}
\label{subsection-building-the-partition-products}

In this subsection, we construct the set $C$, the alphabet $\underline{\ps}=\la\ps_\de:\de\in C\ra$, $\underline{\dot{\Q}}=\la\dot{\Q}_\de:\de\in C\ra$, and the collapsing system $\vec{\varphi}$, that we will use to prove Theorem \ref{theorem:main}. At the same time we will construct ${\ps_{\om_2}}$, a partition product based upon $\underline{\ps},\underline{\dot{\Q}}$. $\ps_{\om_2}$ will force $\mathsf{OCA}_{ARS}$ and $2^{\aleph_0}=\aleph_3$. We will also show how to adapt our construction so that our model additionally satisfies $\mathsf{FA}(\aleph_2,\text{Knaster}(\aleph_1))$; recall that this axiom asserts that we can meet any $\aleph_2$-many dense subsets of an $\leq\aleph_1$-sized poset with the Knaster property.

The collapsing system $\vec{\varphi}$ can be specified right away: for each $\kappa\in C$ and $\gamma<\rho_\kappa$, we take $\varphi_{\kappa,\gamma}$ to be the $<_L$-least surjection of $\kappa$ onto $\gamma$. 

The remaining objects are defined by recursion. Suppose that we've defined the set $C$ up to an ordinal $\ka\leq\om_2$ as well as $\underline{\ps}\res\ka$ and $\underline{\dot{\Q}}\res\ka$ in such a way that the following recursive assumptions are satisfied, where $\bb{A}\res\ka$ denotes the alphabet sequence $\la\la\ps_{\bar{\ka}},\dot{\Q}_{\bar{\ka}}\ra:\bar{\ka}\in C\cap\ka\ra$:
\begin{enumerate}
\item[(a)]  for each $\bar{\ka}\in C\cap\ka$, $\bar{\ka}$ is a local $\om_2$ with respect to $\bb{A}\res\bar{\ka}$, $\ps_{\bar{\ka}}$ is a partition product based upon $\underline{\ps}\res\bar{\ka}$ and $\underline{\dot{\Q}}\res\bar{\ka}$, and $\dot{\Q}_{\bar{\ka}}$ is a $\ps_{\bar{\ka}}$-name;
\item[(b)] every partition product based upon $\underline{\ps}\res\ka$ and $\underline{\dot{\Q}}\res\ka$ is c.c.c.
\end{enumerate}

At limit points $\kappa$, condition (a) trivially follows from the same condition below $\kappa$, and condition (b) follows by Lemma \ref{lemma.limit-ccc}. Thus we need only work on the successor case.

If $\kappa$ is not a local $\om_2$ with respect to $\bb{A}\res\ka$, then we do not place $\kappa$ in $C$, thus leaving $\ps_\ka$ and $\dot{\Q}_\ka$ undefined. Observe that as a result, $\underline{\ps}\res(\ka+1)=\underline{\ps}\res\ka$, and similarly for $\underline{\dot{\Q}}\res\ka$ and $\bb{A}\res\ka$. Suppose, on the other hand, that $\ka$ is a local $\om_2$ with respect to $\bb{A}\res\ka$. We aim to define the partition product $\ps_\ka$ and, in the case that $\ka<\om_2$, to place $\ka$ in $C$ and define the $\ps_\ka$-name $\dot{\Q}_\ka$; defining $\dot{\Q}_\ka$ will involve selecting a $\ps_\ka$-name $\dot{\chi}_\ka$ for a coloring and, by appealing to the results of the previous two sections, constructing the name $\dot{f}_\ka$ for a preassignment. 
We then need to prove (b), namely that every partition product based on $\underline{\ps}\res\ka+1$ and $\underline{\dot{\Q}}\res\ka+1$ is c.c.c.

We begin by defining $\ps_\kappa$. Fix $\gamma$ and assume inductively that $\ps_\ka\res\ga$ has been defined, as well as the base and index functions $\base_\ka\res\ga$ and $\ind_\ka\res\ga$. We divide the definition at $\gamma$ into two cases.\\

Case 1: $\ga+1=\ga(\ka)$, or $\ga=\ga(\ka)$ is a limit.\\

If Case 1 obtains, then we halt the construction, setting $\rho_\ka=\ga$ and $\ps_\ka=\ps_\ka\res\ga$. If $\ka<\om_2$, then we need to define the name $\dot{\Q}_\ka$. Recall from the beginning of this section that for an ordinal $\xi$, $(\xi)_0$ and $(\xi)_1$ are the two ordinals coded by $\xi$ under the G{\"o}del pairing function. Suppose that the $(\ga)_0$-th element under $<_L$ is a pair $\la\dot{S}_\ka,\dot{\chi}_\ka\ra$ of $\ps_\ka$-names, where $\dot{S}_\ka$ names a countable basis for a second countable, Hausdorff topology on $\om_1$ and $\dot{\chi}_\ka$ names a coloring on $\om_1$ which is open with respect to the topology generated by $\dot{S}_\ka$. Then let $\dot{f}_\ka$ be the $<_L$-least $\ps_\ka$-name satisfying Proposition \ref{prop:name}, and set $\dot{\Q}_\ka:=\Q(\dot{\chi}_\ka,\dot{f}_\ka)$, so that by Corollary \ref{cor:name}, any partition product based upon $\underline{\ps}\res(\ka+1)$ and $\underline{\dot{\Q}}\res(\ka+1)$ is c.c.c.  If $(\ga)_0$ does not code such a pair, then we simply let $\dot{\Q}_\ka$ name Cohen forcing for adding a single real. It is clear in this case also, by Lemma \ref{lemma:products}, that any partition product based upon $\underline{\ps}\res(\ka+1)$ and $\underline{\dot{\Q}}\res(\ka+1)$ is c.c.c.

On the other hand, if $\ka=\om_2$, then the partition product $\ps_{\om_2}$ is defined. After completing the rest of the construction, we show that forcing with $\ps_{\om_2}$ provides the desired model witnessing our theorem.\\

Case 2: $\ga+1<\ga(\ka)$.\\

In this case, we continue the construction 
of ${\ps}_\kappa$, extending from ${\ps}_\kappa\res\gamma$, which inductively we already know, to ${\ps}_\kappa\res\gamma+1$.
Let $\bar{\ka}:=H(\ka,\ga)\cap\ka$, which is below $\ka$ by Lemma \ref{lemma:below}; recall that we are suppressing explicit mention of the parameter $\bb{A}\res\ka$. We also let $j$ be the transitive collapse map of $H(\ka,\ga)$ and set $\bar{\ga}:=j(\ga)$. We halt the construction if either $\ps_\ka\res\ga$ is not a member of $H(\ka,\ga)$, or if it is a member of $H(\ka,\ga)$ and either $\bar{\ka}\notin C$ or $\ps_\ka\res\ga$  is not mapped to $\ps_{\bar{\ka}}\res\bar{\ga}$ by $j$ (we will later show that this does not in fact occur).

Suppose, on the other hand, that $\bar{\ka}\in C$ and that $\ps_\ka\res\ga$ is a member of $H(\ka,\ga)$ which is mapped by $j$ to $\ps_{\bar{\ka}}\res\bar{\ga}$. We shall specify the next name $\dot{\bb{U}}_\ga$, 
which will be the $\gamma$th iterand in ${\ps}_\kappa\res\gamma+1$,
as well as the values $\base_\ka(\ga)$ and $\ind_\ka(\ga)$. By Lemma \ref{lemma:collapsestable}, we have that $\bar{\ga}+1=\ga(\bar{\ka},\bb{A}\res\bar{\ka})$. By recursion, this means that $\bar{\ga}=\rho_{\bar{\ka}}$, i.e., that $\ps_{\bar{\ka}}=\ps_{\bar{\ka}}\res\bar{\ga}$. We now pull these objects back along $j^{-1}$.

In more detail, we observe that, setting $\pi_\ga:=j^{-1}$, $\pi_\ga\res\rho_{\bar{\ka}}$ provides an acceptable rearrangement of $\ps_{\bar{\ka}}$, since $\pi_\ga$ is order-preserving. In fact, the $\pi_\ga$-rearrangement of $\ps_{\bar{\ka}}$ is exactly equal to $(\ps_\ka\res\ga)\res\pi_\ga[\rho_{\bar{\ka}}]$, by Lemma \ref{lemma:ambiguous}; this Lemma applies since for each $\delta\in C\cap\bar{\ka}$, $\pi_\ga$ is the identity on $\ps_\de\ast\dot{\Q}_\de\cup\lb\ps_\de,\dot{\Q}_\de\rb$. Let $\dot{\bb{U}}_\ga$ be the $\pi_\ga$-rearrangement of $\dot{\Q}_{\bar{\ka}}$ (see Lemma \ref{lemma:RL} or Definition \ref{rearrangement-restricted-memory}). Note that this rearrangement need not be an element of $L_{\de_\ga(\ka)}$. We now set $\base_\ka(\ga):=(\pi_\ga[\rho_{\bar{\ka}}],\pi_\ga\res\rho_{\bar{\ka}})$ and set $\ind_\ka(\ga):=\bar{\ka}$. In particular, we observe that 
$$
b_\ka(\ga)=H(\ka,\ga)\cap\ga
$$
is an initial segment of the ordinals of $H(\ka,\ga)$.\\

\begin{claim}\label{claim:needsaname} $\base_\ka\res(\ga+1)$ and $\ind_\ka\res(\ga+1)$ support a partition product based upon $\underline{\ps}\res\ka$ and $\underline{\dot{\Q}}\res\ka$.
\end{claim}

\begin{proof}[Proof of Claim \ref{claim:needsaname}] Condition (1) of Definition \ref{def:support} follows from the comments in the above paragraph. Condition (2) holds at $\ga$ by the elementarity of $\pi_\ga$ and at all smaller ordinals by recursion. So we need to check condition (3), where it suffices to verify the coherent collapse condition for $\ga$ and some $\be<\ga$. So suppose that there is some $\xi\in b_\ka(\be)\cap b_\ka(\ga)$. We define $\bar{\ka}^*$ to be $H(\ka,\be)\cap\ka$, so that $\bar{\ka}^*=\ind_\ka(\be)$. We also let $j_{\ka,\be}$ denote the transitive collapse map of $H(\ka,\be)$ and $j_{\ka,\ga}$ the transitive collapse map of $H(\ka,\ga)$. Finally, let $\pi_\be$ denote $j^{-1}_{\ka,\be}$. 

In both of the models $H(\ka,\be)$ and $H(\ka,\ga)$, $\ka$ is the largest cardinal. Moreover, each of them is closed under the map which takes an ordinal $\zeta$ to $\vp_{\ka,\zeta}$, the $<_L$-least surjection from $\ka$ onto $\zeta$. As a result,
$$
H(\ka,\ga)\cap\xi=\vp_{\ka,\xi}[H(\ka,\ga)\cap\ka]=\vp_{\ka,\xi}[\bar{\ka}],
$$
and therefore $b_\ka(\ga)\cap\xi=\vp_{\ka,\xi}[\bar{\ka}]$. Similarly, $b_\ka(\be)\cap\xi=\vp_{\ka,\xi}[\bar{\ka}^*].$

With this observation in mind, we now verify that item (3) of Definition \ref{def:support}  holds for $\gamma$ and $\beta$. Suppose that $\bar{\ka}^*\leq\bar{\ka}$; the proof in case $\bar{\ka}\leq\bar{\ka}^*$ is similar. Let $\zeta_0:=\pi_\be^{-1}(\xi)$ and $\zeta_1:=\pi_\ga^{-1}(\xi)$. If $\bar{\ka}^*=\bar{\ka}$, then by the calculations in the previous paragraph, (3) holds trivially, since the models $H(\ka,\be)$ and $H(\ka,\ga)$ have the same intersection with $\xi+1$. Thus we proceed under the assumption that $\bar{\ka}^*<\bar{\ka}$. Since the above paragraph shows that $\pi_\be[\zeta_0]\seq\pi_\ga[\zeta_1]$, we need to check that $A:=\pi_\ga^{-1}[\pi_\be[\zeta_0]]$ coherently collapses $\la\bar{\ka},\zeta_1\ra$ to $\la\bar{\ka}^*,\zeta_0\ra$.

Now $\pi_\be[\zeta_0]=b_\ka(\be)\cap\xi$ has the form $\vp_{\ka,\xi}[\bar{\ka}^*]$. Since $\bar{\ka}^*<\bar{\ka}$, we have that $\ka$, $\xi$, and $\bar{\ka}^*$ are all in $H(\ka,\ga)$. Thus so is $\pi_\be[\zeta_0]$. Applying the elementarity of $j_{\ka,\ga}=\pi_\ga^{-1}$, we see that $\pi_\ga^{-1}\circ\vp_{\ka,\xi}\res\bar{\ka}^*=\vp_{\bar{\ka},\zeta_1}\res\bar{\ka}^*$, which shows that $A$ has the form $\vp_{\bar{\ka},\zeta_1}[\bar{\ka}^*]$. Therefore condition (\ref{coherent-collapse-critical}) in Definition \ref{coherent-collapse} holds. Additionally, if we let $\si$ denote the transitive collapse of $A$, then we see that $\si\circ\pi_\ga^{-1}$ is the transitive collapse of $\pi_\be[\zeta_0]=\vp_{\ka,\xi}[\bar{\ka}^*]$, which is just $\pi_\be^{-1}=j_{\ka,\be}$. However, the elementarity of $\pi_\be^{-1}$ implies that $\pi_\be^{-1}\circ\vp_{\ka,\xi}\res\bar{\ka}^*=\vp_{\bar{\ka}^*,\zeta_0}$, and therefore $\si\circ\vp_{\bar{\ka},\zeta_1}\res\bar{\ka}^*=\vp_{\bar{\ka}^*,\zeta_0}$. This gives condition (\ref{coherent-collapse-comp}) of Definition \ref{coherent-collapse}. And finally, to see that (\ref{coherent-collapse-closure}) of the definition holds, we first observe that $b_\ka(\be)\cap\xi$ is closed under limit points of cofinality $\om$ below its supremum, because $H(\ka,\be)$ is closed under $\om$-sequences. Since $b_\ka(\be)\cap\xi$ is in $H(\ka,\ga)$, by applying $j_{\ka,\ga}$, we conclude that the collapse of $H(\ka,\ga)$, denoted $L_{\tau}$, satisfies that $A$ is closed under limit points of cofinality $\om$ below its supremum. However, $L_{\tau}$ is closed under $\om$-sequences, and therefore $A$ is in fact closed under limit points of cofinality $\om$ below its supremum. Thus (\ref{coherent-collapse-closure}) is satisfied. This completes the proof of the claim.
\end{proof}

We have now completed the construction of the desired sequence of partition products. Before we prove our main theorem, we need to verify that for each $\ka\in C\cup\lb\om_2\rb$, we obtain a partition product of the appropriate length, i.e., that the construction does not halt prematurely, as described at the beginning of Case 2.

\begin{lemma}\label{lemma:doesnothalt} For each $\ka\in C\cup\lb\om_2\rb$, $\rho_\ka=\ga(\ka)$ if $\ga(\ka)$ is a limit or equals $\ga(\ka)-1$ if $\ga(\ka)$ is a successor.
\end{lemma}
\begin{proof} Suppose that $\ka\in C\cup\lb\om_2\rb$ and that $\ga+1<\ga(\ka)$. We need to show that $\ps_\ka\res\ga$ is a member of $H(\ka,\ga)$, that $\bar{\ka}\in C$, and that $\ps_\ka\res\ga$ gets mapped by $j,$ the collapse map of $H(\ka,\ga)$, to $\ps_{\bar{\ka}}\res\bar{\ga}$, where $\bar{\ka}=j(\ka)$ and $\bar{\ga}=j(\ga)$. By choice of $F$, since $L_{\de_\ga(\kappa)}\models F$, and since $\la\de_i(\ka):i<\ga\ra$ and $\bb{A}\res\ka$ belong to $L_{\de_\ga(\ka)}$, we have $\ps_\ka\res\ga\in L_{\de_\ga(\ka)}$. Since $j(\bb{A}\res\ka)=\bb{A}\res\bar{\ka}$, it is also straightforward to verify that $\bar{\ka}$ is a local $\om_2$ with respect to the sequence $\bb{A}\res\bar{\ka}$, and hence $\bar{\ka}\in C$. Finally, Case 2 of the construction of partition products is uniform, in the sense that $\ps_\ka\res\ga$ is definable in $L_{\de_\ga(\ka)}$ from $\la\de_i(\ka):i<\ga\ra$ and $\bb{A}\res\ka$ by the same definition which defines $\ps_{\bar{\ka}}\res\bar{\ga}$ in $L_{\de_{\bar{\ga}}(\bar{\ka})}$ from $\la\de_i(\bar{\ka}):i<\bar{\ga}\ra$, and $\bb{A}\res\bar{\ka}$. Thus $\ps_\ka\res\ga$ is a member of $H(\ka,\ga)$ and gets mapped to $\ps_{\bar{\ka}}\res\bar{\ga}$ by $j$.
\end{proof}

\begin{remark}
\label{what-is-F}
We now have the notation and context to specify exactly what the finite fragment $F$ of $\zfc-\mathsf{Powerset}$, that we fixed at the start of the section, needs to prove. We fix $F$ that proves that for every collapsing system $\vec{\varphi}$ and alphabet ${\mathbb A}=\langle \ps_\kappa,\dot{\Q}_\kappa \mid\kappa\in C\subseteq \omega_2\rangle$ with respect to $\vec{\varphi}$, both in $L$, 
for every $n<\omega$,
for every ordinal $\gamma$, and for every sequence of ordinals $\langle \delta_i\mid i<\gamma\rangle$ with $\omega_1,{\mathbb A}\in L_{\delta_i}$:
\begin{enumerate}
\item 
\label{what-is-F-1}
Each of the hulls $H(\omega_2,i)=\text{Hull}_{n}^{L_{\de_i}}(\om_1\cup\lb\bb{A}\rb)$ exists, a bijection between the hull and $\omega_1$ exists, and the sequence $\langle H(\omega_2,i)\mid i<\gamma\rangle$ exists. 
\item
The transitive collapse $M_i$ of $H(\omega_2,i)$, the collapse embedding $j_i$, and its inverse $\pi_i$ all exist, as do the corresponding sequences over $i<\gamma$. 
\item
Let $\kappa_i=j_i(\omega_2)$ and suppose that $\kappa_i\in C$ and the domain $\rho_{\kappa_i}$ of $\ps_{\kappa_i}$ is contained in $M_i$, for each $i$. Then the rearranged poset name $\pi_i(\dot{\Q}_{\kappa_i})$ exists, and so does the sequence $\langle \pi_i(\dot{\Q}_{\kappa_i}) \mid i<\gamma\rangle$. 
\item
The function $i\mapsto \kappa_i$ and the sequence $\langle \langle \pi_i[\rho_{\kappa_i}],\pi_i\res \rho_{\kappa_i}\rangle \mid i<\gamma\rangle$ exist. 
\item
If $i\mapsto \kappa_i$ as index function, $\langle \langle \pi_i[\rho_{\kappa_i}],\pi_i\res \rho_{\kappa_i}\rangle \mid i<\gamma\rangle$ as base function, and $\langle \pi_i(\dot{\Q}_{\kappa_i}) \mid i<\gamma\rangle$ as the sequence of iterands satisfy the requirements for determining a partition product based on ${\mathbb A}$ with respect to $\vec{\varphi}$, then this partition product exists. 
\end{enumerate}
Note that the complexity of these statements is independent of $n$; indeed, condition (\ref{what-is-F-1}) is expressed by a $\Sigma_1$ formula with $n$ among its variables.
\end{remark}

We finish by proving Theorem \ref{theorem:main}.

\begin{proof}[Proof of Theorem \ref{theorem:main}] We force over $L$ with $\ps_{\om_2}$. By Lemma \ref{lemma:doesnothalt}, $\ps_{\om_2}$ is a partition product with domain $\ga(\om_2)$, and by Remark \ref{remark:csdefinable}, $\ga(\om_2)=\om_3$ (we suppress  mention of the parameter $\bb{A}$). Let us denote the sequence of names used to form $\ps_{\om_2}$ by $\la\dot{\bb{U}}_\ga:\ga<\om_3\ra$. Since $\ps_{\om_2}$ is a partition product based upon $\underline{\ps}\res\om_2$ and $\underline{\dot{\Q}}$, it is c.c.c. Hence all cardinals are preserved. Since $\ps_{\om_2}$ has size $\aleph_3$ and is c.c.c., it forces that the continuum has size no more than $\aleph_3$. However, $\ps_{\om_2}$ adds $\aleph_3$-many reals, and to see this, we first recall that by Remark \ref{remark:densesubset}, $\ps_{\om_2}$ is a dense subset of the finite support iteration of the names $\la\dot{\bb{U}}_\ga:\ga<\om_3\ra$. Next, each $\dot{\bb{U}}_\ga$ either names Cohen forcing or one of the homogeneous set posets, and each of the latter adds a real. Thus $\ps_{\om_2}$ forces that the continuum has size exactly $\aleph_3$. We now want to see that $\ps_{\om_2}$ forces that $\mathsf{OCA}_{ARS}$ holds.

Towards this end, let $\la\dot{S},\dot{\chi}\ra$ be a pair of $\ps_{\om_2}$-names, where $\dot{S}$ names a countable basis for a second countable, Hausdorff topology on $\om_1$ and $\dot{\chi}$ names a coloring which is open with respect to the topology generated by $\dot{S}$. Let $\ga<\om_3$ so that $\la\dot{S},\dot{\chi}\ra$ is the $(\ga)_0$-th element under $<_L$ and so that $\la\dot{S},\dot{\chi}\ra$ is a $\ps_{\om_2}\res(\ga)_1$-name. Note that $\la\dot{S},\dot{\chi}\ra$ is an element of $H(\om_2,\ga)$ since, by Remark \ref{remark:csdefinable}, $\ga$ is, and also notice that $H(\om_2,\ga)$ satisfies that $\la\dot{S},\dot{\chi}\ra$ is a $\ps_{\om_2}\res\ga$-name. Let $j$ denote the transitive collapse map of $H(\om_2,\ga)$ and let $\pi:=j^{-1}$ denote the anticollapse map. Set $\bar{\ga}:=j(\ga)$ and $\ka:=j(\om_2)$, and observe that by Lemma \ref{lemma:collapsestable}, $j$ collapses $H(\om_2,\ga)$ onto $L_{\de_{\bar{\ga}}(\ka)}$, and $\gamma(\kappa)=\bar{\gamma}+1$. The latter implies that ${\ps}_\kappa=\ps_\kappa\res\bar{\gamma}$.

We will be done if we can show that $G$ adds a partition of $\om_1$ into countably-many $\dot{\chi}[G]$-homogeneous sets, and towards this end, let $G$ be $V$-generic over $\ps_{\om_2}$. We use $G_\ga$ to denote the generic $G$ adds for $\dot{\bb{U}}_\ga[G\res\ga]$ over $V[G\res\ga]$. Set $\bar{G}$ to be $j\left[(G\res\ga)\cap H(\om_2,\ga)\right]$, and observe that $\bar{G}$ is generic for the poset $j(\ps_{\om_2}\res\ga)=\ps_\kappa\res\bar{\ga}=\ps_\ka$ over $L_{\de_{\bar{\ga}}(\ka)}$. Since $\ps_\ka$ is c.c.c.\ and $L_{\de_{\bar{\ga}}(\ka)}$ is countably closed, $\bar{G}$ is also $V$-generic over $\ps_\ka$. In particular, $\pi$ extends to a
$\Sigma_k$
elementary embedding 
$$
\pi^*:L_{\de_{\bar{\ga}}(\ka)}[\bar{G}]\lra L_{\de_\ga(\om_2)}[G],
$$
and since $\operatorname{crit}(\pi^*)>\om_1$, we see that $\dot{S}[G]=j(\dot{S})[\bar{G}]$ and $\dot{\chi}[G]=j(\dot{\chi})[\bar{G}]$.

By the elementarity of $\pi^*$ and absoluteness, $\la j(\dot{S}),j(\dot{\chi})\ra$ is the $(\bar{\ga})_0$-th 
element under $<_L$ and is a
pair of $\ps_\ka$-names where the first coordinate names a countable basis for a second countable, Hausdorff topology on $\om_1$ and the second names a coloring which is open with respect to the topology generated by that basis. By the construction of $\dot{\Q}_\ka$, this means that $\dot{\Q}_\ka$ names the poset to decompose $\om_1$ into countably-many $j(\dot{\chi})$-homogeneous sets with respect to the preassignment $\dot{f}_\ka$. Thus forcing with $\dot{\Q}_\ka[\bar{G}]$ adds a decomposition of $\om_1$ into countably-many $j(\dot{\chi})[\bar{G}]=\dot{\chi}[G]$-homogeneous sets. We will be done if we can show that $G$ adds a generic for $\dot{\Q}_\ka[\bar{G}]$.

To see this, we recall from Case 2 of the construction that $\dot{\bb{U}}_\ga$ is the $\pi\res\rho_\ka$-rearrangement of $\dot{\Q}_\ka$. Moreover, as also described in Case 2, Lemma \ref{lemma:ambiguous} applies. Thus $\dot{\Q}_\ka[\bar{G}]=\dot{\bb{U}}_\ga[G]$.  $G_\ga$ is therefore $V[G\res\ga]$-generic for $\dot{\Q}_\ka[\bar{G}]$, which finishes the proof.
\end{proof}

We wrap up by sketching a proof of Theorem \ref{theorem:FA}.

\begin{proof}[Proof Sketch of Theorem \ref{theorem:FA}] We first describe how to build the names on the sequence $\underline{\dot{\Q}}$. The only modification to the construction for the previous theorem is that if, in Case 1 above,
the $(\gamma)_0$-th element under $<_L$
names a Knaster poset of size $\aleph_1$, then we set $\dot{\Q}_\ka$ to be this Knaster poset. With this modification to the sequence $\underline{\dot{\Q}}$, we still maintain the recursive assumption that for each $\ka\in C$, any partition product based upon $\underline{\ps}\res\ka$ and $\underline{\dot{\Q}}\res\ka$ is c.c.c.; this follows by Lemma \ref{lemma:products}, Lemma \ref{lemma:RL}, and since the product of Knaster and c.c.c.\ posets is still c.c.c.
 
Now we want to see that forcing with this modified $\ps_{\om_2}$ gives the desired model. The proof that the extension satisfies $\mathsf{OCA}_{ARS}$ and $2^{\aleph_0}=\aleph_3$ is the same as before. To prove that it satisfies $\mathsf{FA}(\aleph_2,\text{Knaster}(\aleph_1))$, suppose that $\dot{\mathbb K}$ is forced in ${\mathbb P}_{\omega_2}$ to be a Knaster poset of size $\aleph_1$. We may assume without loss of generality that $\dot{\mathbb K}$ is forced to be a subset of $\omega_1$. Fix $\gamma$ so that $(\gamma)_0$ codes $\dot{\mathbb K}$, making $\gamma$ large enough so that $\dot{\bb{K}}$ is a $(\ps_{\om_2}\res\ga)$-name and so that all the dense sets we need to meet belong to $V[G\res\gamma]$.  Next, arguing as in the proof of Theorem \ref{theorem:main}, we have $\kappa<\omega_2$, $j\colon H(\omega_2,\gamma)\lra L_{\de_{\bar{\ga}}(\ka)}$, and an extension 
$$
\pi^*\colon L_{\de_{\bar{\ga}}(\ka)}[\bar{G}]\lra L_{\de_\ga(\om_2)}[G\res\gamma]
$$
of the inverse $\pi$ of $j$. By the modified Case 1 construction we have that $\dot{\mathbb Q}_\kappa=j(\dot{\mathbb K})$. By Case 2 in the construction of ${\mathbb P}_{\omega_2}$, $\dot{\mathbb U}_\gamma$ is the rearrangement of $\dot{\mathbb Q}_\kappa$ by $\pi\res \rho_\kappa$. However, by the final clause in Lemma \ref{lemma:ambiguous}, and since $\dot{\mathbb Q}_\ka$ names a poset contained in $\omega_1<\ka=\operatorname{crit}(\pi)$, this rearrangement is exactly $\pi(\dot{\mathbb Q}_\ka)=\dot{\mathbb K}$. So $G_\gamma$ is generic for $\dot{\mathbb K}[G\res\gamma]$ over $V[G\res\ga]$, and hence $G_\gamma$ is a filter in $V[G]$ for $\dot{\bb{K}}[G\res\gamma]$ which meets the desired dense sets.
\end{proof}

\section*{Acknowledgments}
This material is based upon work supported by the National Science Foundation under grant No. DMS-1764029.

\end{document}